\documentclass[numbers,webpdf]{ima-authoring-template}%
\usepackage{geometry,graphicx,amssymb,amsmath,amsbsy,eucal,amsfonts,mathrsfs,amscd,bm,esint,yhmath}
\usepackage{stmaryrd}
\usepackage{xifthen}
\usepackage[normalem]{ulem}
\usepackage[all]{xy}
\usepackage{tikz}
\usetikzlibrary{arrows,shapes}
\usetikzlibrary{arrows, decorations.markings,fit}
\usetikzlibrary{calc,3d}
\usepackage{tikz-3dplot}
\usepackage{soul}
\usepackage{lipsum}

\DeclareFontFamily{U}{matha}{\hyphenchar\font45}
\DeclareFontShape{U}{matha}{m}{n}{
      <5> <6> <7> <8> <9> <10> gen * matha
      <10.95> matha10 <12> <14.4> <17.28> <20.74> <24.88> matha12
      }{}
\DeclareSymbolFont{matha}{U}{matha}{m}{n}
\DeclareFontSubstitution{U}{matha}{m}{n}

\DeclareFontFamily{U}{mathx}{\hyphenchar\font45}
\DeclareFontShape{U}{mathx}{m}{n}{
      <5> <6> <7> <8> <9> <10>
      <10.95> <12> <14.4> <17.28> <20.74> <24.88>
      mathx10
      }{}
\DeclareSymbolFont{mathx}{U}{mathx}{m}{n}
\DeclareFontSubstitution{U}{mathx}{m}{n}

\DeclareMathDelimiter{\vvvert}{0}{matha}{"7E}{mathx}{"17}

\tikzset{
    >=stealth',
    punkt/.style={
           rectangle,
           rounded corners,
           draw=black, very thick,
           text width=6.5em,
           minimum height=2em,
           text centered},
    pil/.style={
           ->,
           thick,
           shorten <=2pt,
           shorten >=2pt,}
}

\usepackage{listings}

\numberwithin{equation}{section}

\allowdisplaybreaks[3]

\theoremstyle{thmstyletwo}%
\newtheorem{theorem}{Theorem}[section]
\newtheorem{lemma}[theorem]{Lemma}

\newtheorem{corollary}[theorem]{Corollary}

\theoremstyle{definition}

\theoremstyle{remark}
\newtheorem{remark}[theorem]{Remark}

\newcommand{\ba}{\bm a}

\newcommand{\nab}{\nabla}

\newcommand{\bld}[1]{\boldsymbol{#1}}
\newcommand{\bt}{\bld{t}}

\newcommand{\bI}{{\bf I}}

\newcommand{\bA}{\bld{A}}

\newcommand{\bw}{\bld{w}}

\newcommand{\bp}{\bld{p}}
\newcommand{\bn}{\bld{n}}
\newcommand{\bu}{\bld{u}}

\newcommand{\bW}{\bld{W}}
\newcommand{\bV}{\bld{V}}

\newcommand{\bPi}{\bld{\Pi}}
\newcommand{\bz}{\boldsymbol{z}}

\newcommand{\bH}{{\bf H}}

\newcommand{\by}{\bld{y}}

\newcommand{\bX}{{\bm X}}

\newcommand{\bPhi}{{\bm \Phi}}

\newcommand{\calE}{\mathcal{E}}

\newcommand{\calS}{\mathcal{S}}
\newcommand{\calT}{\mathcal{T}}

\newcommand{\bPsi}{{\bm \Psi}}

\newcommand{\calN}{\mathcal{N}}
\newcommand{\bnu}{{\bm \nu}}
\newcommand{\calP}{\mathcal{P}}

\newcommand{\HT}{{\bf HT}}

\newcommand{\pt}{\wideparen}

\makeatletter
\def\widebreve{\mathpalette\wide@breve}
\def\wide@breve#1#2{\sbox\z@{$#1#2$}%
     \mathop{\vbox{\m@th\ialign{##\crcr
\kern0.08em\brevefill#1{0.8\wd\z@}\crcr\noalign{\nointerlineskip}%
                    $\hss#1#2\hss$\crcr}}}\limits}
\def\brevefill#1#2{$\m@th\sbox\tw@{$#1($}%
  \hss\resizebox{#2}{\wd\tw@}{\rotatebox[origin=c]{90}{\upshape(}}\hss$}
\makeatletter

\newcommand{\ipt}{\breve}

\begin{document}

\DOI{DOI HERE}
\copyrightyear{}
\vol{00}
\pubyear{}
\access{Advance Access Publication Date: Day Month Year}
\appnotes{Paper}
\copyrightstatement{Published by Oxford University Press on behalf of the Institute of Mathematics and its Applications. All rights reserved.}
\firstpage{1}


\title[$H({\rm div})$-conforming FEM for surface Stokes]{Error analysis of penalized $H({\rm div})$-conforming finite element methods for surface Stokes problems}

\author{Alan Demlow*
\address{\orgdiv{Department of Mathematics}, \orgname{Texas A\&M University}, \orgaddress{
\postcode{77843}, \state{Texas}, \country{USA}}}}
\author{Orsan Kilicer 
\address{\orgdiv{Department of Computer Technology}, \orgname{Istanbul Bilgi University}, \orgaddress{
\state{Istanbul}, \country{Turkey}}}}


\authormark{A. Demlow and O. Kilicer}

\corresp[*]{Corresponding author: \href{email:demlow@tamu.edu}{demlow@tamu.edu}}

\received{Date}{0}{Year}
\revised{Date}{0}{Year}
\accepted{Date}{0}{Year}


\abstract{In recent years a number of finite element methods have been developed for the surface Stokes equations.  Constructing standard conforming surface finite element methods for the velocity-pressure formulation of this problem would require simultaneously enforcing $H^1$ conformity and tangentiality of the velocity field, which is not possible on the generalized polyhedral approximating surfaces on which surface FEM are typically posed.  Thus methods for this problem weakly enforce either tangentiality or $H^1$ conformity of the velocity field.  In this paper we provide an error analysis for an $H({\rm div})$-conforming method based on classical Brezzi-Douglas-Marini elements.  In this method continuity of the velocity field is weakly enforced using an interior penalty-type formulation.  Advantages of this formulation include exact enforcement of the incompressibility constraint and its construction based on a standard mixed finite element pair.  Error analysis of this method in previous work was only valid for lowest-order elements and did not account for ``geometric errors'', i.e., variational crimes due to approximation of the continuous surface on which the PDE is posed by an approximating surface on which the finite element method is posed.  Here we provide a full error analysis in a natural discontinuous Galerkin energy norm, with geometric errors accounted for.  Optimal estimates for the $L_2$ velocity error are also presented.  In the case of lowest-order (polyhedral) surface approximations, our proof of optimal $L_2$ estimates requires a stricter than usual condition on the placement of mesh nodes relative to the surface.  Numerical experiments are used to illustrate our results.   }


\maketitle

\thispagestyle{empty}

\section{Introduction}  In this paper we consider approximations to the surface Stokes problem 
\begin{equation}
\label{eq:surfstokes}
\begin{aligned}
-2 \bPi {\rm div}_\gamma {\rm Def}_\gamma \bu+\nab_\gamma p+\bu & = {\bm f}\qquad&\text{on }\gamma,\\
{\rm div}_\gamma \bu  = 0\text{ and } \bu\cdot \bnu &= 0 \qquad &\text{on }\gamma.
\end{aligned}
\end{equation}
Here $\gamma \subset \mathbb{R}^3$ is a smooth, closed two-dimensional surface, $\bnu$ is the outward unit normal to $\gamma$, and $\bPi = {\bf I}-\bnu \otimes \bnu$ is the projection onto the tangent plane of $\gamma$.  Also, $p$ is the pressure and $\bu$ is the velocity, and ${\rm Def}_\gamma$ and $\nabla_\gamma$ are the tangential rate of strain tensor and tangential scalar gradient, respectively.  We emphasize that due to the constraint $\bu \cdot \bnu=0$, the velocity $\bu$ is tangent to the surface $\gamma$.  Note that the mass term ``$+\bu$'' in \eqref{eq:surfstokes} is typically not present in corresponding Euclidean Stokes equations.  However, the potential presence of degeneracies due to Killing fields significantly complicates analysis of finite element methods for surface Stokes problems with no added mass term (cf. \cite{BDL20}).  The space of Killing fields is the kernel of the rate of strain tensor, and is nontrivial only when $\gamma$ possesses intrinsic isometries.  It is difficult to robustly detect whether such nondegeneracies exist and properly account for them in a typical surface finite element setting, so studies of finite element methods for the surface Stokes problem typically incorporate a mass term in order to ensure a well-posed problem independent of the surface geometry.

There has been significant recent interest in constructing effective finite element schemes for surface Stokes equations.  In the weak form of the velocity-pressure formulation \eqref{eq:surfstokes} the velocity field $\bu$ is taken to lie in the space $\HT^1(\gamma)$ consisting of tangential and componentwise $H^1$ vector fields.  In a typical surface finite element method (SFEM), the continuous surface $\gamma$ is first approximated by a polyhedral surface $\Gamma_h$ (or higher-degree polynomial analog) which is merely Lipschitz.  In order to construct a conforming method it would be necessary to construct a subspace of $\HT^1(\Gamma_h)$.  As was however explained in \cite{DN24}, it is not possible to construct meaningful finite element subspaces of $\HT^1(\Gamma_h)$ because any continuous tangential vector field must vanish at vertices of $\Gamma_h$ having at least three incident noncoplanar faces.  Thus surface FEM for the velocity-pressure formulation of the surface Stokes equations and related vector Laplace problems constructed on merely $C^0$ approximating surfaces weakly enforce either $H^1$ conformity or tangentiality of the velocity field.  These include works which enforce $H^1$ conformity strongly and enforce tangentiality weakly via either a Lagrange multiplier approach \cite{Fries18, GJOR18} or via penalization \cite{ BJPRV22, GJOR18, HansboLarsonLarsson20, HP23, HP25, JORZ21,  OQRY18,  ORZ21,  Maxim19}.  These approaches have also been used in the construction of trace-FEM formulations in which the surface mesh and finite element space are inherited from an ambient bulk space.   An alternative approach is to enforce tangentiality strongly and $H^1$ conformity weakly. In the recent papers \cite{DN24, DN26, KNPPP}, Piola transforms were used to interpret vector degrees of freedom at vertices, resulting in surface counterparts to standard Euclidean Stokes spaces including MINI, Taylor-Hood, and Scott-Vogelius that are $H({\rm div}, {\Gamma_h})$- and tangentially-conforming.  These spaces are not $H^1$ conforming but possess sufficient weak continuity to avoid penalization.   Finally, discretizations of stream function formulations of the surface Stokes equations have been considered in \cite{BJPRV22, BR20, Re20}.  

Our focus in this work is penalized $H({\rm div}, {\Gamma_h})$-conforming methods employing standard $BDM$ (Brezzi-Douglas-Marini) mixed spaces.   Such methods were originally proposed for Euclidean Stokes equations in \cite{CockburnEtal05, CKS07}.  Here lack of $H^1$ conformity--more precisely, lack of tangential continuity along element edges--is accounted for via an interior penalty type formulation.  These methods were adapted to the surface Stokes context in the papers \cite{BDL20, SurfaceStokes2}, and a streamfunction-harmonic formulation based on BDM spaces and similar penalty methods was recently studied in \cite{BLVWPP}.  In \cite{BDL20} the authors considered lowest-order ($\mathbb{P}^1-\mathbb{P}^0$) BDM spaces.  This work focused especially on formulating methods which robustly account for Killing fields. The paper also contained some error analysis, but did not rigorously account for ``geometric errors'' due to the approximation of $\gamma$ by $\Gamma_h$ and also did not include proof of a discrete inf-sup condition or error estimates for the pressure.  In \cite{SurfaceStokes2}  the authors formulated similar methods assuming arbitrary degree of finite element space and surface approximation. The focus of the latter work was algorithmic and computational, and no rigorous error analysis was provided.  Two major advantages of this method are its standard construction via canonical $H({\rm div})$-conforming elements and an interior penalty formulation, and its ability to exactly enforce the incompressibility constraint in \eqref{eq:surfstokes}.

Our main goal in this work is to provide a complete energy norm error analysis for penalized $H({\rm div})$-conforming SFEM for the surface Stokes equations employing $BDM_r$ ($\mathbb{P}^r-\mathbb{P}^{r-1}$, $r \ge 1$) finite element spaces and surface approximations $\Gamma_h^k$ ($k \ge 1$) that are piecewise $\mathbb{P}^k$.  Letting $\|\cdot\|_{DG,\gamma}$ denote a natural discontinuous Galerkin (DG) energy norm for the above problem which we shall define below, we prove that for velocity approximation $\bu_h$ and pressure approximation $p_h$ there holds
\begin{equation}
    \label{eq:introenergy}
\|  \bu-\pt{\bu}_h \|_{DG,\gamma}+ \|p -p_h^\ell\|_{L_2(\gamma)} \lesssim h^r+h^k
\end{equation}
Here $\pt \bu_h$ is a Piola transform of $\bu_h$ to $\gamma$, and similarly $p_h^\ell$ is a lift of $p_h$ to $\gamma$.  Also, the first term $h^{r}$ represents a ``Galerkin error'' arising from employing a finite element space in the variational formulation, while the second term $h^k$ represents a geometric error due to the approximation of $\gamma$ by its discrete counterpart in the formulation of the finite element method.  Note that this estimate is of optimal order when isoparametric ($k=r$) elements are used. We additionally provide $L_2$ error estimates of the form
\begin{equation} \label{intro:L2} 
\|\bu-\pt \bu_h\|_{L_2(\gamma)} \lesssim h^{r+1} + h^{k+1}.
\end{equation}
This estimate is also optimal in a typical isoparametric setting.  In the lowest-order case $k=1$ our proof of this result requires a nonstandard condition.  More precisely, for all edges $e$ of $\Gamma_h$ we require that $|d(z_0)-d(z_1)| \lesssim h^3$.  Here $z_0, z_1$ are the vertices of $e$ and $d$ is the signed distance function for $\gamma$.  It is standard and necessary to assume $\|d\|_{L_\infty(\Gamma_h)} \lesssim h^2$, so this condition is stricter than usual.  It holds for example when the vertices of the approximating polyhedron lie on the continuous surface or have either a small ($O(h^3)$) bias or a consistent bias away from the surface.  The proof of the relevant geometric error bounds is also substantially more involved when $k=1$ as compared with the higher-order cases $k \ge 2$.  Numerical experiments below indicate that this added condition may not be necessary to obtain optimal convergence when $k=1$.


We remark on some of the technical challenges that arise in this work.  As is typical when employing $H({\rm div})$-conforming spaces, we use Piola transforms to map vector quantities between surfaces.  Piola transforms, which involve multiplying vector fields by Jacobian-weighted tangent maps, are natural pullbacks for $H({\rm div})$ spaces.  However, in the context of surface Stokes equations we need to instead compare $H^1$-type forms of Piola transformed vector fields.  These comparisons require relatively involved technical computations.  The recent papers \cite{DN24, DN26} analyzing nonconforming MINI and Taylor-Hood methods similarly required comparison of $H^1$ inner products of Piola-transformed quantities.  The methods in those works do not however require penalty terms, so a substantial part of the novelty of the present work lies in the analysis of geometric errors arising from comparison of penalty terms of Piola transformed quantities.  We also prove ancillary results, such as discrete Korn-type inequalities, which are similarly significantly more involved in the context of penalized methods.  Finally, we establish inf-sup stability of our FEM, which was missing in previous analyses of penalized $H({\rm div})$-conforming methods for surface Stokes.    

An outline of the paper is as follows.  In Section \ref{sec:prelims} we give a number of preliminaries about surfaces, surface approximations, and finite element spaces.  In Section \ref{sec:geofoundations} we give important foundational results concerning norm equivalences on surfaces, discrete Korn inequalities, and comparison of the Stokes bilinear form on discrete and continuous surfaces.  Section \ref{sec:errors} contains error analysis for our finite element method.  Section \ref{sec:numerics} contains numerical examples.  Finally, we note that more detailed proofs of many of the results below related to energy-norm error estimates are contained in the Ph.D. thesis \cite{Kil25} of the second author.  

\section{Preliminaries}
\label{sec:prelims}

In this section we give a number of preliminaries concerning surface geometry and transformations between surfaces.  Much of this material concerns assumptions about higher-order surface approximations and transformation of scalar and vector functions between surfaces, including via Piola transforms.  In many cases we closely follow the recent work \cite{DN26}, which similar to this work concerns analysis of $H({\rm div})$-conforming FEM for surface Stokes equations on arbitrary-degree surface approximations.  Piola transforms of vector fields play a central role in both cases, but analysis of the method considered here presents substantial additional challenges due to the presence of DG-type jump  and average terms.

\subsection{Geometric preliminaries}
We assume that $\gamma$ is a smooth, closed, and orientable 2 dimensional surface embedded in $\mathbb{R}^3$.  Denote by $d$ the signed distance function for $\gamma$.  $d$ is defined on a tubular neighborhood $U$ about $\gamma$ with width given by the inverse of the maximum principal curvature on $\gamma$, and $\gamma=\{x \in U: d(x)=0\}$.  In addition, $\bnu = \nabla  d $ is the outward unit normal to $\gamma$, and $\Pi= \bI-\bnu\otimes \bnu$ is the projection onto the tangent plane to $\gamma$.  
  Let $\bp(x)=x-d(x)\nabla d(x) = x-d(x) \bnu(x)$, $x \in U$, be the closest point projection onto $\gamma$.  ${\bf H}(x) = D^2d = \nabla \bnu$ is the shape operator (Weingarten map).  

We also define tangential differential operators.  Assume that $u, \bw$ are respectively scalar and vector functions defined on $U$.  Then $\nabla_\gamma u = (\nabla u) \bPi$, $\nabla_\gamma \bw = \bPi \nabla \bw \bPi$, and ${\rm Def}_\gamma \bw = \frac{1}{2} ( \nabla_\gamma \bw + \nabla_\gamma \bw^\top) = \frac{1}{2} \bPi ( \nabla \bw + \nabla \bw^\top) \bPi$ is the rate of strain tensor.  These tangential differential operators are intrinsic in that they depend only on $u|_\gamma$ and $\bw|_\gamma$ even though their definition employs extensions to $U$.  Also, ${\rm div}_\gamma \bw = {\rm trace} \nabla_\gamma \bw$, and for a matrix ${\bf A}$ the divergence ${\rm div}_\gamma {\bf A}$ is computed by applying the vector tangential divergence operator rowwise.

\subsection{Surface approximations, meshes, and reference mappings}
We largely follow \cite{DN26} in our description of the triangulation.  Let $\bar{\Gamma}_h \subset U$ be a polyhedral surface approximation to $\gamma$ having triangular faces $\bar \calT_h$ and outward unit normal $\bar \bnu_h$.  The tangential projection operator on $\bar \Gamma_h$ is then $\bar \bPi_h = {\bf I}-\bar \bnu_h \otimes \bar \bnu_h$.  Let also $\hat{K} \subset \mathbb{R}^2$ be the standard reference triangle, and given $\bar{K} \in \bar\calT_h$ let $F_{\bar K}: \hat{K} \rightarrow \bar K$ be an affine reference map.  $\bar\calT_h$ is assumed to be shape regular in the sense that
\begin{equation} \label{eq:map1}
\|\nabla F_{\bar K}\|_{L_\infty(\hat K)} \lesssim h_{\bar K}, ~~~{\rm det} (\nabla F_{\bar K}^\top \nabla F_{\bar K}) \simeq h_{\bar K}^4,
\end{equation}
where $h_{\bar K} ={\rm diam}(\bar K)$ and $\nabla F_{\bar K} \in \mathbb{R}^{3 \times 2}$ is the Jacobian of $F_{\bar K}$.  We shall write $a \lesssim b$ to mean that $a \le Kb$ with $K>0$ a generic constant independent of $h$ but possibly depending on the penalty parameter $\rho$ defined below, and similarly for $\gtrsim$ and $\simeq$.  In a few instances we employ a generic constant $C$ which is taken to be independent of both $\rho$ and $h$.

Next we define a family $\Gamma_h^k$ ($k \ge 1$) of higher order surface approximations.  First let $\Gamma_{h,1} = \bar \Gamma_h$.  We assume that there exists a continuous map $\bPsi: \bar \Gamma_h \rightarrow \mathbb{R}^3$ such that $\bPsi_{\bar K} = \bPsi|_{\bar K} \in [\mathbb{P}_k(\bar K)]^3$ for $\bar K \in \bar \calT_h$ and 
\begin{equation}
\label{eq:map2}
|\bPsi(x)-\bp(x)| \lesssim h^{k+1}, ~~~ |\nabla \bPsi(x)-\nabla \bp(x)| \lesssim h^k.
\end{equation}
$\Gamma_h^k = \bPsi(\bar \Gamma_h)$ is then our degree-$k$ approximation to $\gamma$.  $\calT_{h,k} = \{ \bPsi(\bar K): \bar K \in \bar \calT_h \}$ is the (curved) triangulation associated to $\Gamma_h^k$.  Letting $\bnu_h$ be the outward unit normal to $\Gamma_h^k$, there then holds that for $x \in \Gamma_h^k$
\begin{equation}
\label{eq:geoest}
|d(x)| + h |(\bnu-\bnu_h)(x)| \lesssim h^{k+1}.
\end{equation}
We also let $\bPi_h = {\bf I}-\bnu_h \otimes \bnu_h$ denote the projection onto the tangent plane to $\Gamma_h^k$. Let also $\bt_h^k$ be a unit tangent vector to $e$ and $\bt_\gamma$ a unit tangent vector to $e^\gamma=\bPsi(e)$.  For the sake of concreteness, let $e=K_1 \cap K_2$ with $K_1, K_2 \in \calT_{h,k}$ and the numbering of $K_1$ and $K_2$ fixed relative to $e$.  We then let $\bt_h^k$ point in the counterclockwise direction on $K_1$.  Similarly, let  $\bn_{\bar e}$ be the outward unit conormal to $\bar{K}_1$ on $\bar e$ and $\bn_h^k$ the outward unit conormal to $K_1$ on $e$.  We also denote by $\bn_{h1}^k$ and $\bn_{h2}^k$ the outward conormals to $K_1$ and $K_2$ on $e$. On an edge $e^\gamma=K_1^\gamma \cap K_2^\gamma$, $\bn_\gamma$ is the outward unit conormal to $K_1^\gamma$ on $e^\gamma$, and thus the outward conormal to $K_2^\gamma$ on $e^\gamma$ is $-\bn_\gamma$.  Note that $\bn_\gamma \cdot \bn_{h1}^k>0$ while $\bn_\gamma \cdot \bn_{h2}^\gamma<0$.  Also, let $h_K = h_{\bar K}$ with $K = \bPsi(\bar K)$, and set $ h =\max_{K \in \calT_{h,k}} h_K$.  Finally, we assume that $\bar \calT_h$ is quasi-uniform, that is, $h_K =h_{\bar K} \simeq h$ for all $ \bar K \in \bar \calT_h$.  

Next we note that $\bPsi$ is typically realized concretely as a polynomial mapping from the reference element $\hat K \subset \mathbb{R}^2$.  More precisely,
\begin{equation}
\bPsi_{\bar K} = \ba_K \circ F_{\bar K}^{-1},
\end{equation}
where $\ba_K: \hat{K} \rightarrow K$ is a tri-degree $k$ diffeomorphism with $K = \bPsi (\bar K) \in \calT_{h,k}$.  The following properties are assumed:
\begin{equation}
\label{eq:map5}
|\ba_K|_{W_\infty^m(\hat K)} \lesssim h^m, ~~~~|\ba_K^{-1} |_{W_\infty^m(K)} \lesssim h^{-1} ~(1 \le m \le k+1), ~~~~{\rm det}( \nabla \ba_K^\top \nabla \ba_K) \simeq h^4.
\end{equation}
These estimates imply that
\begin{equation}
|\bPsi|_{W_\infty^m(\hat K)} \lesssim |F_{\hat K}^{-1} |_{W_\infty^1 (\bar K)}^m |\ba_K|_{W_\infty^m(\hat K)} \lesssim 1, ~~~~\bar K \in \bar \calT_h, ~~m \ge 1.
\end{equation}

Without loss of generality we assume that $F_{\bar K}(\hat{a}) = \ba_K (\hat a)$ for all vertices $\hat a \in \hat K$, i.e., $F_{\bar K}$ is the linear interpolant of $\ba_K$.  It then easily follows that
\begin{equation} \label{eq:map7}
\|\ba_k - F_{\bar K}\|_{W_\infty^m(\hat K)} \lesssim h^2, ~~~m \ge 0.
\end{equation}

Given $K \in \calT_{h,k}$ we let $K^\gamma = \bp(K)$ denote its image on the surface $\gamma$. The corresponding mesh on $\gamma$ is denoted by $\calT_{h,k}^\gamma = \{ K^\gamma: K \in \calT_{h,k} \}$.  Also, we let $\bw_K=\bw|_K$ for $K \in \calT_{h,k}$, and similarly for $\bw_{\bar K}$ and $\bw_{K^\gamma}$.  Finally, denote by $\varSigma$ the mesh skeleton (the union of triangle edges) of $\calT_{h,k}$ and by $\varSigma^\gamma$ its image on $\gamma$, and by $\calE$ and $\calE^\gamma$ the sets of edges on $\Gamma_h^k$ and $\gamma$, respectively.

{\bf Extensions and lifts}   Given a scalar function $w$ defined on the exact surface $\gamma$, let $w^e(x) = w^e(\bp(x))$ denote its extension to $U$ (and similarly for vector functions).  If on the other hand $w$ is defined on $\Gamma_h^k$, let $\tilde w (\bp(x)) = w(x)$ for $x \in \Gamma_h^k$.  We then define the lift $w^\ell(x) = \tilde{w}(\bp(x))$. With some abuse of notation we denote by $w^\ell$ both lifts and extensions, i.e., we write $w^\ell(x)=w(\bp(x))= w^e(x)$; this notation is intuitive because in the case of both lifts and extensions information is transmitted along normals to $\gamma$.  
 Also, $\bar{w}(x) = w(\bPsi(x))$, $x \in \bar\Gamma_h$.  

Next define $\mu_h: \Gamma_h^k \rightarrow  \mathbb{R}$ and $\bar \mu_h: \bar \Gamma_h \rightarrow \mathbb{R}$ by
\begin{equation}
\label{eq:mudef}
\mu_h = \bnu \cdot \bnu_h (1-d \kappa_1)(1-d \kappa_2), ~~~~\bar \mu_h = \frac{\sqrt{ {\rm det} \nabla {\ba}^\top \nabla  \ba}}{\sqrt{{\rm det} \nabla F_{\bar K}^\top \nabla F_{\bar K}}},
\end{equation}
where $\kappa_1$ and $\kappa_2$ are or the nonzero (in general) eigenvalues of ${\bf H} = D^2 d$ (these are the principal curvatures when evaluated on $\gamma$).  Then
\begin{equation}
 \int_{\Gamma_h^k} w^\ell \mu_h =\int_\gamma w, ~~~~w \in L_1(\gamma), \hbox{ and } \int_{\bar \Gamma_h} \overline{w} \bar \mu_h = \int_{\Gamma_h^k} w, ~~~~w \in L_1(\Gamma_h^k).  
\end{equation}
 Let also $\bar e$ and $e$ be edges of triangles in $\bar \calT_h$ and $\calT_{h,k}$, respectively, with $e = \bPsi(\bar e)$.  Let $\bt_{\bar e}$ denote a unit tangent vector to $\bar e$.  Then
 \begin{equation}
 \int_e q = \int_{ \bar e} \mu_{\bar e} q \circ \bPsi = \int_{e^\gamma} \mu_e^{-1} q \circ \bp^{-1}, ~~~q \in L_1(e).
 \end{equation}
Here $\mu_{\bar e}  = |\nabla \bPsi \bt_{\bar e}|$ and $\mu_e = |\nabla \bp|_{e} \bt_h^k|=|(\bPi-d {\bf H})\bt_h^k|$.   
\begin{lemma}  With definitions as above, there hold for $h$ sufficiently small
\begin{equation} \label{eq:tangent_conormal}
|1-\mu_h| + |1-\mu_e| + |\bn_\gamma-\bPi \bn_h^k|+|\bn_\gamma +\bPi \bn_{h2}^k|+|\bt_\gamma-\bPi \bt_h^k| \lesssim h^{k+1},
\end{equation}
\begin{equation} \label{eq:tc2}
|\bt_\gamma-\bt_h^k|+|\bn_\gamma-\bn_h^k|+|\bn_\gamma+\bn_{h2}^k| \lesssim h^k,
\end{equation}
and
\begin{equation} \label{eq:tc3}
|1-\bar \mu_h | + |1-\mu_{\bar e}| +|\bar \bnu_h-\bnu_h|+ |\bt_{\bar \Gamma_h} -\bt_h^k|+ |\bn_{\bar \Gamma_h}-\bn_h^k|  \lesssim h.
\end{equation}
In addition, for $q \in L_1(e)$ ($e \in \calE$) and $1 \le p \le \infty$ there holds
\begin{equation} \label{eq:edgenormeq}
\|q \|_{L_p(e)} \simeq \|q \circ \bPsi\|_{L_p(\bar e)} \simeq \|q \circ \bp^{-1} \|_{L_p(e^\gamma)}.
\end{equation}
\end{lemma}
\begin{proof} \hspace{.1cm}
The first inequality in \eqref{eq:tangent_conormal} can be found in \cite{Demlow09} and the second in \cite{CD16}.  The third and fourth are found in \cite{ADMSSV15}, while the fifth is an easy consequence of the third.  

The first inequality in \eqref{eq:tc2} can be derived by noting that $\bt_\gamma = \frac{(\bPi-d{\bH}) \bt_{\Gamma_h^k}}{\left |(\bPi-d{\bH}) \bt_{\Gamma_h^k} \right |}$ and employing \eqref{eq:geoest}.  The second and third follow from the first and \eqref{eq:geoest} by noting that $\bn_\gamma=\bt_\gamma \times \bnu$ and similarly for $\bn_h^k$.  

Let $F_i$ and $a_i$ be the $i-th$ columns of $\nabla F_{\bar K}$ and $\nabla \ba_K$, respectively.   To prove the first inequality in \eqref{eq:tc3}, note that 
\begin{equation}
\label{eq:muh1}
\begin{aligned}
1-\bar\mu_h &= 1-\frac{\sqrt{{\rm det} \nabla \ba_K^\top \nabla \ba_K}}{\sqrt{{\rm det} \nabla F_{\bar K}^\top \nabla F_{\bar K}}} = \frac{{\rm det} \nabla F_{\bar K}^\top \nabla F_{\bar K}- {\rm det} \nabla \ba_K^\top \nabla \ba_K}{\sqrt{{\rm det} \nabla F_{\bar K}^\top \nabla F_{\bar K}}\left (\sqrt{{\rm det} \nabla F_{\bar K}^\top \nabla F_{\bar K}} + \sqrt{{\rm det} \nabla \ba_K^\top \nabla \ba_K} \right)}
\\ & = \frac{[(F_1\cdot F_1)(F_2\cdot F_2)-(F_1 \cdot F_2)^2] - [(a_1 \cdot a_1)(a_2\cdot a_2) -(a_1 \cdot a_2)^2]}{\sqrt{{\rm det} \nabla F_{\bar K}^\top \nabla F_{\bar K}}\left (\sqrt{{\rm det} \nabla F_{\bar K}^\top \nabla F_{\bar K}} + \sqrt{{\rm det} \nabla \ba_K^\top \nabla \ba_K} \right)} 
\end{aligned}
\end{equation}
In order to bound the numerator, we may use \eqref{eq:map1}, \eqref{eq:map5}, and \eqref{eq:map7} to compute
\begin{equation}
\label{eq:muh2}
\begin{aligned}
(F_1 \cdot F_2)^2-(a_1 \cdot a_2)^2 & = (F_1\cdot F_2-a_1 \cdot a_2)(F_1 \cdot F_2 + a_1 \cdot a_2) 
\\ & = ((F_1 -a_1) \cdot F_2+a_1 \cdot (F_2-a_2))(F_1 \cdot F_2 + a_1 \cdot a_2) 
\\ & \lesssim (h^2 h-h h^2)(h^2) = h^5,
\end{aligned}
\end{equation}
and similarly for the term $(F_1 \cdot F_1)(F_2 \cdot F_2)-(a_1 \cdot a_1)(a_2 \cdot a_2)$.  The denominator is bounded from below by $Ch^4$ by \eqref{eq:map1} and \eqref{eq:map5}, which yields the desired bound $|1-\bar \mu_h| \lesssim h$ when combined with \eqref{eq:muh1} and \eqref{eq:muh2}.  The third inequality in \eqref{eq:tc3} follows from $\bar \bnu_h = \frac{F_1 \times F_2}{|F_1 \times F_2|}$, $\bnu_h = \frac{a_1 \times a_2}{|a_1 \times a_2|}$, \eqref{eq:map1}, \eqref{eq:map5}, and \eqref{eq:map7}.   Next note that for $x \in \bar\Gamma_h$, $|(\nabla \bPsi-\bar \bPi_h)(x)| \le |(\nabla \bPsi-\nabla p)(x)|+|(\nabla p-\bar \bPi_h)(x)| \lesssim h^k + |\bPi-d\bH -\bar\bPi_h| \lesssim h$.  Here we have applied \eqref{eq:map2} and then \eqref{eq:geoest} with $k=1$.  This yields the second inequality in \eqref{eq:tc3} after noting that $1-\mu_{\bar e} =1-|\nabla \bPsi \bt_{\bar e}| $.  The fourth inequality similarly follows by noting that $\bt_h^k =\frac{\nabla \bPsi \bt_{\bar e}}{|\nabla \bPsi \bt_{\bar e}|}$.  The fifth inequality follows from the fourth and \eqref{eq:geoest} after noting that $\bn_{\bar \Gamma_h} = \bt_{\bar \Gamma_h} \times \bar \bnu_h$ and $\bn_h^k = \bt_h^k \times \bnu_h^k$.  

Finally, \eqref{eq:edgenormeq} follows from \eqref{eq:tangent_conormal} and \eqref{eq:tc3} for $h$ sufficiently small.  
\end{proof}

\subsection{Piola transforms}   Following \cite{CD16} (cf. \cite{Steinmann08}), we next describe Piola transforms between surfaces in $\mathbb{R}^3$.  Let $\bPhi: \calS_0 \rightarrow \calS_1$ be a diffeomorphism between surfaces $\calS_0$ and $\calS_1$.  Let $d \sigma_i$ be surface measure on $\calS_i$ and let $\mu$ satisfy $\mu d \sigma_0 = d\sigma_1$.  The Piola transform and inverse Piola transform of vector fields $\bu: \calS_0 \rightarrow \calS_1$, $\bw: \calS_1 \rightarrow \calS_0$ with respect to the mapping $\bPhi$ are defined by
\begin{equation}
(\calP_{\bPhi} \bu) \circ \bPhi = \mu^{-1} \nabla \bPhi \bu, ~~~~(\calP_{\bPhi^{-1}} \bw) \circ \bPhi^{-1} = (\mu \circ \bPhi^{-1}) \nabla \bPhi^{-1} \bw.
\end{equation}
There in addition holds
\begin{equation}
\label{eq:pioladiv}
{\rm div}_{\calS_0} \bu = \mu {\rm div}_{\calS_1} \calP_\bPhi \bu.
\end{equation}

When $\calS_0 = \Gamma_h^k$, $\calS_1=\gamma$, $\bPhi = \bp$, $\bPi_h \bu = \bu$, and $\bPi \bw =\bw$, we have
\begin{equation} \label{eq:pioladef}
\pt \bu \circ \bp := \calP_{\bp} \bu = \frac{1}{\mu_h} [ \bPi-d {\bf H}] \bu, \hspace{.5cm} \ipt \bw := \calP_{\bp^{-1}} \bw = \mu_h \left [ {\bf I}-\frac{\bnu \otimes \bnu_h}{\bnu \cdot \bnu_h} \right ] [ {\bf I}-d {\bf H}]^{-1} (\bw \circ \bp).
\end{equation}
We additionally define for $\bu \in H({\rm div}, \Gamma_H^k)$
\begin{equation}
\bar \bu = \calP_{\bPsi^{-1}} \bu \in H({\rm div}, \bar \Gamma_h).
\end{equation}

Relationships between mapped and Piola transformed functions on various surfaces play a critical role in our analysis.  The following results \eqref{eq:pprops}--\eqref{eq:defequiv} concerning such relationships are summarized from  \cite[\S 2.3]{DN26}.  

\begin{lemma}
Assume that the conditions \eqref{eq:map1}-\eqref{eq:map7} concerning the construction of $\bar{\Gamma}_h$ and $\Gamma_h^k$ hold.   First, identifying with slight abuse of notation a Piola transform $\calP$ with the matrix representing it we have
\begin{equation}
\label{eq:pprops}
|\mu_h|_{W_\infty^m(K)}+|\mu_h^{-1}|_{W_\infty^m(K)}+\|\calP_{\bp}\|_{W_\infty^m(K)} + \|\calP_{\bp^{-1}} \|_{W_\infty^m(K^\gamma)} \lesssim 1, ~~m \ge 0 \hbox{ and } K \in \calT_{h,k}, 
\end{equation}
and recalling that $\pt \bw = \calP_{\bp} \bw$, 
\begin{equation}
\label{eq:normequiv}
\|\pt \bw \|_{W_p^m(K^\gamma)} \simeq \|\bw \|_{W_p^m(K)}, \hspace{.75cm} \bw \in W_p^m(K), ~~K \in \calT_{h,k},~~m \ge 0, \hbox{ and } p \in [1,\infty].
\end{equation}
 Then  
\begin{equation}\label{eq:Ppsiequiv}
\begin{aligned}
|\calP_{\bPsi_{\bar K}} \bar \bw-\bar \bw| & \lesssim h|\bar\bw|, \hspace{.5cm} \bar \bw \cdot \bar \bnu_h = 0,
\\|\calP_{\bPsi_{K}^{-1}} \bw -\bw| & \lesssim h|\bw|,\hspace{.5cm} \bw \cdot \bnu_h=0.
\end{aligned}
\end{equation}
Given a matrix $\bA$, denote by $\bA^\dagger = (\bA^\top \bA)^{-1} \bA^\top$ its Moore-Penrose inverse.  There then holds for $m \ge 0$ and $K \in \calT_{h,k}$ that
\begin{equation}
\label{eq:piolaprops}
\begin{aligned}
|\calP_{\ba_K}|_{W_\infty^m(\hat{K})} & \lesssim h^{m-1}, \qquad & |(\calP_{\ba_K})^{\dagger} |_{W_\infty^m(\hat{K})}  \lesssim h^{1+m},
\\ \|\calP_{\bPsi}\|_{W_\infty^m(\bar K)} & \lesssim 1, & \|\calP_{\bPsi^{-1}}\|_{W_\infty^m(K)}  \lesssim 1.
\end{aligned} 
\end{equation}
Thus if $\bw \in H^m(K)$ and $\hat \bw \in H^m(\hat K)$ with $\bw= \calP_{\ba_K} \hat \bw$, there holds
\begin{equation}
|\bw|_{H^m(K)} \lesssim \sum_{\ell=0}^m h^{-\ell} |\hat \bw|_{H^\ell (\hat K)}, \qquad |\hat \bw|_{H^m(\hat K)} \lesssim h^m \|\bw \|_{H^m(K)}.
\end{equation}
In addition, for $\bar \bw \in [H^m(\bar K)]^3$ with $\bar \bw \cdot \bar \bnu_h=0$ there holds
\begin{equation} \label{eq:normequiv2}
\|\calP_{\bPsi}  \bar \bw\|_{H^m(K)} \simeq \|\bar \bw\|_{H^m( \bar K)}, \qquad K = \bPsi(\bar K).
\end{equation}
Finally, for $\bw \in \HT^1(K)$ there holds
\begin{equation}
\label{eq:defequiv}
|{\rm Def}_\gamma \pt \bw -({\rm Def}_{\Gamma_h^k} \bw) \circ \bp^{-1}| \lesssim h^k ( |(\nabla_{\Gamma_h^k} \bw ) \circ \bp^{-1} | + |\bw \circ \bp^{-1}|) \hbox{  on } K^\gamma.
\end{equation}

\end{lemma}

\subsection{Norms and function spaces}  Discrete norms play a critical role in our development.  For $m \ge 0$ let $H_h^m(\gamma)= \{ \bw \in L_2(\gamma): \bw|_{K^\gamma} \in H^m(K^\gamma) \forall K^\gamma \in \calT_{h,k}^\gamma\}$, and for $\bw \in H_h^m(\Gamma)$ let 
\begin{equation}
\|\bw \|_{H_h^m(\gamma)}^2 = \sum_{K^\gamma \in \calT_{h,k}^\gamma } \|\bw\|_{H^m(K^\gamma)}^2.
\end{equation}
Similar definitions apply to $H_h^m(\bar \Gamma_h)$ and $H_h^m(\Gamma_h^k)$.  From \eqref{eq:normequiv} and \eqref{eq:normequiv2} we obtain
\begin{equation}
\label{eq:hnormequiv}
\|\bw\|_{H_h^m(\Gamma_h^k)} \simeq \|\pt \bw \|_{H_h^m(\gamma)} \simeq \|\bar \bw \|_{H_h^m(\bar \Gamma_h)}
\end{equation}
whenever these quantities are defined.  We also employ standard $H({\rm div})$ spaces given by
\begin{equation}
H({\rm div}, \Gamma_h^k) = \{ \bw \in [L_2(\Gamma_h^k)]^3 \hbox{ s.t. } {\rm div}_{\Gamma_h^k} \bw \in L_2(\Gamma_h^k) \hbox{ and } \bw \cdot \bnu_h =0\}
\end{equation}
and similarly for $H({\rm div}, \bar \Gamma_h)$ and $H({\rm div}, \gamma)$.

Given $e^\gamma \in \calE^\gamma$, denote by $K_1^\gamma, K_2^\gamma\in \calT_{h,k}^\gamma$ the triangles sharing $e^\gamma$ and by $\bt_\gamma$ the unit tangent to $e^\gamma$. Let also $\pt \bw_i = \pt \bw|_{K_i^\gamma}$, $i=1,2$.  Similar definitions apply to $\bar \bw$ and $\bw$.  With $[\bw]=\bw_1-\bw_2$, we then define
\begin{equation}
\label{eq:jumpdef}
\begin{aligned}
[ \pt \bw]_{\gamma} & = [\pt \bw]\cdot \bt_\gamma \bt_\gamma = (\pt \bw_2-\pt \bw_1) \cdot \bt_\gamma \bt_\gamma,
\\  [\bw]_{\Gamma_h^k} & = [\bw]\cdot \bt_h^k \bt_h^k =(\bw_2-\bw_1) \cdot \bt_h^k \bt_h^k,
\\ [\bar \bw]_{\bar \Gamma_h} & = [\bar \bw] \cdot \bt_{\bar e} \bt_{\bar e} =  (\bar \bw_2-\bar \bw_1) \cdot \bt_{\bar e} \bt_{\bar e}.
\end{aligned}
\end{equation}

Given a penalty parameter $\rho \ge C>0$ independent of $h$, we then define for $\bw \in H_h^1(\Gamma_h^k) \cap H({\rm div}_{\Gamma_h^k})$
\begin{equation}
\begin{aligned}
\|\pt \bw\|_{DG,\gamma}^2 & = \| {\rm Def}_{ \gamma} \pt \bw \|_{L_2(\gamma)}^2+ \|\pt \bw \|_{L_2(\gamma)}^2 + \frac{\rho}{h} \sum_{e^\gamma \in \calE^\gamma} \| [ \pt \bw ]_\gamma \|_{L_2(e^\gamma)}^2,
\\ \|\bw \|_{DG, \Gamma_h^k}^2 & = \| {\rm Def}_{ \Gamma_h^k} \bw \|_{L_2(\Gamma_h^k)}^2+ \|\bw \|_{L_2(\Gamma_h^k)}^2 + \frac{\rho}{h} \sum_{e \in \calE} \| [ \bw ]_{\Gamma_h^k}  \|_{L_2(e)}^2,
\\ \|\bar \bw  \|_{DG, \bar \Gamma_h}^2 & = \| {\rm Def}_{ \bar \Gamma_h} \bar \bw \|_{L_2(\bar \Gamma_h)}^2+ \|\bar \bw \|_{L_2(\bar \Gamma_h)}^2 + \frac{\rho}{h} \sum_{\bar e \in \bar \calE} \| [ \bar \bw ]_{\bar \Gamma_h}  \|_{L_2(\bar e)}^2.
\end{aligned}
\end{equation}
Here and below, with slight abuse of notation differential operators are defined elementwise in cases where the argument does not possess the relevant global regularity.  

We finally define the norm
\begin{equation}
\label{eq:triplebardef}
\vvvert \bw \vvvert_{\Gamma_h^k} = \|\bw\|_{DG, \Gamma_h^k} + \|\bw\|_{H_h^1(\Gamma_h^k)} + h \|\bw \|_{H_h^2(\Gamma_h^k)}
\end{equation}
and similarly for $\vvvert \cdot \vvvert_{\gamma}$. We shall show below that for discrete velocity fields $\bw_h$ there holds $\vvvert \bw_h \vvvert_{\Gamma_h^k} \simeq \|\bw\|_{DG, \Gamma_h^k}$ and similarly for the corresponding norms on $\bar \Gamma_h$ and $\gamma$.

\subsection{Finite element spaces}
We employ $BDM_r$ finite element spaces with $r \ge 1$.  Let $\hat K$ be the reference triangle and define $\hat \bV= [\mathbb{P}^r(K)]^2$.  Let then
\begin{equation}
\begin{aligned}
\bar \bV_h & = \{\bar \bw \in H({\rm div}, \bar \Gamma_h): \forall \bar K \in \bar \calT_h \exists \hat \bw \in \hat \bV: \bar \bw|_{\bar K} =\calP_{F_K} \hat \bw \}, 
\\ \bar Q_h & = \{ q \in L_2(\bar \Gamma_h): q|_{\bar K} \in \mathbb{P}^{r-1}(\bar K), \forall \bar K \in \bar \calT_h\}, ~~~Q_h^0 = Q_h \cap L_2^0(\Gamma_h^k),
\end{aligned}
\end{equation}
where $L_2^0(\Gamma_h^k)= \{ u \in L_2(\Gamma_h^k): \int_{\Gamma_h^k} u=0 \}$, and
\begin{equation}
\begin{aligned}
\bV_h & = \{ \calP_{\bPsi_h} \bar \bw: \bar \bw \in \bar \bV_h\} = \{ \bw \in H({\rm div}, {\Gamma_h^k}): \forall K \in \calT_{h,k} \exists \hat \bw \in \hat \bV: \bw = \calP_{a_k} \hat \bw \}, 
\\ Q_h & =  \{ q  \in L_2(\Gamma_h^k): \exists \bar q \in \bar Q_h : q=\bar q \circ\bPsi \}.
\end{aligned}
\end{equation}

\subsection{Bilinear forms}

We first define $a_\gamma :[H_h^1(\gamma)]^3 \cap H({\rm div}; \gamma)  \times [H_h^1(\gamma)]^3 \cap H({\rm div};\gamma) \rightarrow \mathbb{R}$ by 
\begin{equation}
\label{eq:adef}
a_\gamma(\bu,\bw) =  2\int_{\gamma} \mathrm{Def}_{\gamma} \bu 
:\mathrm{Def}_{\gamma}\bw
+ \int_{\gamma} \bu\cdot \bw + j_\gamma(\bu, \bw).  
\end{equation}
where
\begin{equation}
\begin{aligned}
\label{eq:jadef}
 j_{\gamma}(\bu,\bw)&=2 \left[-\int_{\varSigma^\gamma}\{\mathrm{Def}_{\gamma}\bu\}_{\gamma}
 \cdot [\bw]_{\gamma} 
 -
 \int_{\varSigma^\gamma}\{\mathrm{Def}_{\gamma}\bw \}_{\gamma} \cdot[\bu]_{\gamma}\right] 
 + \frac{2\rho}{h}\int_{\varSigma^\gamma }[\bu]_{\gamma}\cdot[\bw]_{\gamma}. 
\end{aligned}
\end{equation}
As above, $\rho>0$ is an $h$-independent penalty parameter which will be taken to be sufficiently large below.  

The corresponding discontinuous Galerkin bilinear form on the discrete surface $\Gamma_h^k$ is given by 
\begin{equation} \label{eq:ahdef}
a_{\Gamma_h^k}(\bu_h,\bw_h)
= 2\int_{\Gamma_h^k}\mathrm{Def}_{\Gamma_h^k}\bu_h
:\mathrm{Def}_{\Gamma_h^k}\bw_h
+
\int_{\Gamma_h^k}\bu_h \cdot \bw_h
+
j_{\Gamma_h^k}(\bu_h,\bw_h),
\end{equation}
where
\begin{equation}
\label{eq:jdef}
\begin{aligned}
 j_{\Gamma_h^k}(\bu_h,\bw_h)&=2 \left[-\int_{\varSigma}\{\mathrm{Def}_{\Gamma_h^k}\bu_h \}_{\Gamma_k^h}
 \cdot [\bw_h]_{\Gamma_h^k} 
 -
 \int_{\varSigma}\{\mathrm{Def}_{\Gamma_h^k}\bw_h\}_{\Gamma_k^h} \cdot[\bu_h]_{\Gamma_h^k}\right] 
 \\
 &+ \frac{2\rho}{h}\int_{\varSigma}[\bu_h]_{\Gamma_h^k}\cdot[\bw_h]_{\Gamma_h^k}. 
\end{aligned}
\end{equation}
The average operators above are defined by
\begin{equation}
\begin{aligned}
\{ {\rm Def}_\gamma \bu  \}_{\gamma}  & = \frac{1}{2}  ( {\rm Def}_\gamma\bu|_{K_1^\gamma} \bn_{\gamma1}-{\rm Def}_\gamma \bu|_{K_2^\gamma} \bn_{\gamma2}), 
\\ \{ {\rm Def}_{\Gamma_h^k} \bu \} _{\Gamma_h^k}& = \frac{1}{2} ({\rm Def}_{\Gamma_h^k}\bu|_{K_1} \bn_{h1}^k-{\rm Def}_{\Gamma_h^k} \bu|_{K_2} \bn_{h2}^k ),
\end{aligned} 
\end{equation}
where $\bn_{hi}^k$  ($i=1,2$) are the outward unit conormals to an edge shared by $K_1,K_2 \in \calT_{h,k}$ and similarly for quantities on $\gamma$.  We recall that $\bn_{\gamma 1} = -\bn_{\gamma 2}$ but $\bn_{h 1}^k \neq -\bn_{h 2}^k$ in general.  

\begin{remark}  The definitions of $j_{\gamma}$ and $j_{\Gamma_h^k}$ in the previous works \cite{BDL20,SurfaceStokes2} on penalized $H({\rm div})$-conforming methods for surfaces Stokes problems differ slightly from those here.  In \cite{BDL20}, the jump term on $\Gamma_h^k$ is defined by $[\bw]=\bw_{K_1}-\bw_{K_2}=(\bw_{K_1}-\bw_{K_2})\cdot \bt_h^k \bt_h^k + (\bw_{K_1} \cdot \bn_{h1}^k \bn_{h1}^k-\bw_{K_2}\cdot \bn_{h2}^k \bn_{h2}^k)$, while in \cite{SurfaceStokes2} the jump term is defined as $(\bw_{K_1}-\bw_{K_2})\cdot \bt_h^k \bt_h^k + (\bw_{K_1} \cdot \bn_{h1}^k +\bw_{K_2}\cdot \bn_{h2}^k) \frac{\bn_{h1}^k+\bn_{h2}^k}{2}$. The former definition is not the same as ours in \eqref{eq:jumpdef} as ours completely omits the conormal jump.  The latter definition corresponds with ours when using $H({\rm div})$ conforming elements as we do here since $\bw_{K_1} \cdot \bn_{h1}^k +\bw_{K_2}\cdot \bn_{h2}^k=0$ for $\bw \in H({\rm div}, \Gamma_h^k)$; the additional term allows for definition of penalty methods with fully discontinuous elements which we do not consider here.  The corresponding definitions on $\gamma$ all overlap in our context because for elementwise sufficiently smooth $\bu \in H({\rm div},\gamma)$ there hold $ \bu_{K_1^\gamma} \cdot \bn_{\gamma 1}+\bu_{K_2^\gamma} \cdot \bn_{\gamma 2}=0$ and $\bn_{\gamma 1}=-\bn_{\gamma 2}$ and thus the conormal components of all three types of jump terms disappear.  Our definition modestly simplifies the analysis and requires similar effort to implement.  Note as well that using integration by parts, all three definitions lead to consistency of the bilinear form $a_\gamma$ in the sense that 
\begin{equation}
    a_\gamma(\bu, \bw) =  \int_\gamma -2\bPi {\rm div_\gamma} {\rm Def}_\gamma \bu \cdot \bw+  \bu \cdot \bw, ~~\bu \in H^2(\gamma) \hbox{ and } \bw \in H({\rm div}, \gamma) \cap H_h^2(\gamma).  
\end{equation}
\end{remark}

Given ${\mathcal{S}} \in \{\gamma, \Gamma_h^k, \bar \Gamma_h\}$, we also define for $(\bu, q) \in H({\rm div}; \mathcal{S}) \times L_2(\mathcal{S})$
\begin{equation}
b_{\mathcal{S}} (\bu, q) = \int_{\mathcal{S}} q {\rm div}_{\mathcal{S}} \bu.
\end{equation}
An important property of Piola transforms is that given $(\bw, q) \in H({\rm div}; \Gamma_h^k) \times L_2(\Gamma_h^k)$, there holds
\begin{equation}
\label{eq:piola_property}
b_{\gamma} ( \pt \bw, q^\ell) = b_{\Gamma_h^k} (\bw, q) = b_{\bar \Gamma_h} (\bar \bw, \bar q).
\end{equation}

\section{Norm equivalences, Korn inequalities, and geometric errors}
\label{sec:geofoundations}

In this section we establish equivalences between norms of jumps on the discrete and continuous surfaces, which lead to similar equivalences for full discontinuous Galerkin norms.  We then prove discrete Korn inequalities.  

\subsection{Equivalence of jumps}

We first prove a lemma establishing equivalence of the sum of the jump and $L_2$ portions of the DG norm on the various surfaces.  

\begin{lemma} \label{lem:jumps}
Let $ \bw \in \bV_h$.  Then for $e \in \calE$ with $e = K_1 \cap K_2$ ($K_1,K_2 \in \calT_{h,k}$),
\begin{equation} \label{eq:jumpequivs}
\begin{aligned}
h^{-1} \|[\bar \bw]_{\bar \Gamma_h}\|_{L_2 (\bar e)}^2  + \|\bar \bw \|_{\bar K_1 \cup \bar K_2}^2 & \simeq h^{-1} \|[\bw]_{\Gamma_h^k}\|_{L_2(e)}^2  + \|\bw \|_{L_2(K_1 \cup K_2)}^2
\\ & \simeq h^{-1} \|[\pt \bw]_{\gamma} \|_{L_2(e^\gamma)}^2 + \|\pt \bw \|_{L_2(K_1^\gamma \cup K_2^\gamma)}^2.
\end{aligned}
\end{equation}
\end{lemma}

\begin{proof} \hspace{.1cm} 
Let $\bw_i = \bw_{K_i}$, and similarly for other quantities such as $\bPsi$.  We first compute that on $e$,
\begin{equation}
\label{eq:ptwisejumps}
\begin{aligned}
| [\bw]_{\Gamma_h^k} | & = | (\calP_{\bPsi_1} \bar \bw_1 - \calP_{\bPsi_2} \bar \bw_2)\cdot \bt_h^k \bt_h^k|
\\ & \le |  ((\calP_{\bPsi_1} -{\bf I}) \bar \bw_1 -(\calP_{\bPsi_2}-{\bf I}) \bar \bw_2) \cdot \bt_h^k \bt_h^k | +| (\bar \bw_1-\bar \bw_2) \cdot \bt_h^k \bt_h^k|
 \\ & \le | ((\calP_{\bPsi_1} -{\bf I}) \bar \bw_1 -(\calP_{\bPsi_2}-{\bf I}) \bar \bw_2)\cdot \bt_h^k \bt_h^k| 
 \\ & ~~~~ + |(\bar \bw_1-\bar \bw_2) \cdot (\bt_h^k-\bt_{\bar e}) \bt_h^k| + |(\bar \bw_1-\bar \bw_2) \cdot \bt_{\bar e} (\bt_h^k-\bt_{\bar e})| +| [\bar \bw]_{\bar \Gamma_h} | 
\\ & \lesssim h (|\bar \bw_1| + |\bar \bw_2|) + | [\bar \bw]_{\bar \Gamma_h}|. 
\end{aligned}
\end{equation}
In the last inequality we have applied \eqref{eq:tc3} and the first inequality in \eqref{eq:Ppsiequiv}.  Taking the norm of both sides of this inequality over $e$, employing norm equivalence \eqref{eq:edgenormeq}, using a standard scaled trace inequality 
\begin{equation}
 \label{eq:scaledtrace} 
\|\bw_i\|_{L_2(e)} \lesssim h^{-1/2} \|\bw_i\|_{L_2(K_i)} + h^{1/2} \|\nabla_{\Gamma_h^k} \bw_i\|_{L_2(K_i)},  
\end{equation}
and an inverse inequality,  we find
\begin{equation}
h^{-1} \|[\bw]_{\Gamma_h^k}\|_{L_2(e)}^2 \lesssim h^{-1} \|[\bar \bw]_{\bar \Gamma_h}\|_{L_2(\bar e)}^2 + \|\bar \bw\|_{L_2(\bar K_1 \cup \bar K_2)}^2.
\end{equation}
The first inequality in \eqref{eq:jumpequivs} then follows from the norm equivalence \eqref{eq:normequiv2}.  The other inequalities follow from similar arguments employing in addition the second inequality in \eqref{eq:Ppsiequiv} and similar estimates $|\calP_{\bp} \bw-\bw|\lesssim h^k |\bw|$, $|\calP_{\bp^{-1}} \pt \bw-\pt \bw| \lesssim h^k |\pt \bw|$ which may be easily established using \eqref{eq:pioladef}.  
\end{proof}

\subsection{Discrete Korn inequality}

We begin by stating a continuous Korn inequality \cite{JankuhnEtal18}.

\begin{lemma}  Given $\bw \in {\bf HT}^1(\gamma)$, there holds
\begin{equation} \label{eq:contkorn}
\|\bw \|_{H^1(\gamma)} \lesssim \|{\rm Def}_\gamma \bw \|_{L_2(\gamma)} + \|\bw \|_{L_2(\gamma)}.
\end{equation}
\end{lemma}

Because $\calP_{\bp} \bV_h \not \subset {\bf HT}^1(\gamma)$, \eqref{eq:contkorn} does not generally hold for discrete functions.  We instead have the following analog; cf. \cite{BDL20} for a similar estimate in the case $r=1$.  Note that when $\pt \bw \in {\bf HT}^1(\gamma)$, this result reduces to \eqref{eq:contkorn}.  
\begin{lemma} Given $\bw \in \bV_h$, there holds
\begin{equation} \label{eq:disckorn}
\|\pt \bw\|_{H_h^1(\gamma)} \lesssim \|\pt \bw \|_{DG, \gamma}.
\end{equation}
\end{lemma}
\begin{proof} \hspace{.1cm}
We proceed by constructing a conforming analog $\bW$ to $\pt \bw$.  Let $\bar W_h\subset H^1 (\bar \Gamma_h)$ be the standard degree-$r$ Lagrange space on $\bar \Gamma_h$, and let ${\bf YT}= \{\bPi (\by \circ \bPsi^{-1} \circ \bp^{-1}): \by \in [\bar W_h]^3\} \subset {\bf HT}^1(\gamma)$.  Let also $\calN$ be the set of Lagrange nodes on $\Gamma_h^k$.   Let $\calN= \calN_{int} \cup \calN_{vert}\cup \calN_{edge}$, where $\calN_{vert}$, $\calN_{edge}$, and $\calN_{int}$ are respectively the sets of nodes lying on vertices, element edges, and element interiors.  Let $\bar \Phi_z \in \bar W_h$ be the standard Lagrange basis function satisfying $\bar \Phi_z(\tilde{z})=1$ if $\tilde{z}=z$, and $\bar \Phi_z(\tilde z) = 0 $ if $z \neq \tilde z \in \calN$.  Letting ${\rm Val} (z)$ be the valence of a vertex $z$, we then set 
\begin{equation}
\bW =\bW_1+\bW_2+\bW_3
\end{equation}
where
\begin{equation}
\begin{aligned}
 \bW_1(x) & = \bPi \left ( \sum_{z \in \calN_{int}} \pt \bw(z) \ast \bPhi_z(x) \right ),
\\ \bW_2(x)& = \bPi \left ( \sum_{z \in \calN_{vert} } \left  ( \frac{1}{{\rm Val}( z)} \sum_{K^\gamma  \in \calT_{h,k}^\gamma: z \in K^\gamma} \pt \bw|_{K^\gamma}(z) \right ) \ast \bPhi_z(x) \right ),
\\ \bW_3(x) & = \bPi \left ( \sum_{z \in \calN_{edge}} \left ( \frac{1}{2} \sum_{K^\gamma \in \calT_{h,k}^\gamma: z \in K^\gamma} \pt \bw|_{K^\gamma}(z) \right ) \ast \bPhi_z(x) \right ).
\end{aligned}
\end{equation}
Here $\bPhi_z = [\bar \Phi_z, \bar \Phi_z, \bar \Phi_z]^\top \circ \bPsi^{-1} \circ \bp^{-1}$, and $\ast$ is defined as componentwise multiplication of two vectors:  $\bw \ast \by = [w_1 y_1, w_2 y_2, w_3 y_3]^\top$.  

For $z \in \calN$, let $\omega_z$ denote the patch of elements touching $z$.  By a standard scaling argument there holds $\|\bPhi_z\|_{L_2(\gamma)} \lesssim h$.  Using this relationship, the boundedness of $\calP_{\bp^{-1}}$ and $\calP_{\bPsi^{-1}}$, an inverse estimate, bounded overlap of the patches $\omega_z$, and norm equivalence, we have
\begin{equation}
\label{eq:WL2}
\begin{aligned}
\|\bW\|_{L_2(\gamma)}^2&  \lesssim \sum_{z \in \calN} h^2 \|\pt \bw\|_{L_\infty(\omega_z)} \lesssim \sum_{z \in \calN} h^2 \|\bar \bw\|_{L_\infty(\bPsi^{-1} \circ \bp^{-1} (\omega_z))}^2  
\\ & \lesssim \sum_{z \in \calN} \|\bar \bw\|_{L_2(\omega_z)}^2 \lesssim \|\bar \bw\|_{L_2(\bar \Gamma_h)}^2 \lesssim \|\pt \bw \|_{L_2(\gamma)}^2.
\end{aligned} 
\end{equation}

We next bound $\sum_{K_\gamma \in \calT_{h,k}^\gamma}  \|\nabla_\gamma (\pt \bw-\bW)\|_{L_2(K^\gamma) }^2.$  For a given $K^\gamma\in \calT_{h,k}^\gamma$, let
\begin{equation}
\bW_{K^\gamma} = \bW_{1,K^\gamma} + \bW_{2,K^\gamma}+\bW_{3, K^\gamma}
\end{equation}
where
\begin{align}
\begin{aligned}
\bW_{1, K^\gamma}(x) & = \bPi \sum_{z \in \calN_{int, K^\gamma}} \pt \bw(z) \ast \bPhi_z(x),
\\ \bW_{2, K^\gamma}(x) & = \bPi \sum_{z \in \calN_{vert, K^\gamma}} \pt \bw|_{K^\gamma}(z) \ast \bPhi_z(x),
\\ \bW_{3, K^\gamma}(x) & = \bPi \sum_{z \in \calN_{edge, K^\gamma}} \pt \bw|_{K^\gamma} (z) \ast \bPhi_z(x),
\end{aligned}
\end{align}
where $\calN_{int, K^\gamma}=\{z \in \calN_{int}:z \in K^\gamma\}$ and similarly for $\calN_{vert, K^\gamma}$ and $\calN_{edge, K^\gamma}$.  We then employ the splitting
\begin{equation}
\label{eq: Wsplit}
\|\nabla_\gamma (\pt \bw - \bW)\|_{L_2(K^\gamma)} \le \|\nabla_\gamma(\pt \bw - \bW_{K^\gamma})\|_{L_2(K^\gamma)} + \sum_{i=1}^3 \|\nabla_\gamma( \bW_{i, K^\gamma}-\bW_i)\|_{L_2(K^\gamma)}.
\end{equation}

First note that $\bW_{1,K^\gamma}-\bW_1=0$.  Standard scaling yields $\|\bPhi_z\|_{H^1(K^\gamma)} \lesssim 1$, so for $i=2,3$
\begin{equation}
\sum_{i=2}^3 \|\nabla_\gamma(\bW_{i,K^\gamma}-\bW_i)\|_{L_2(K^\gamma)} \lesssim \sup_{z \in \calN_{vert, K^\gamma} \cup \calN_{edge, K^\gamma}} \sup_{K' \in \calT_{h,k}^\gamma, z \in \overline{K'}} |(\pt \bw_{K^\gamma}-\pt \bw_{K'})(z)|.
\end{equation}
If $K^\gamma$ and $K'$ share an edge, then because $\pt \bw \in H({\rm div}; \gamma)$ there holds $(\pt \bw_{K^\gamma}-\pt \bw_{K'})\cdot \bnu=(\pt \bw_{K^\gamma}-\pt \bw_{K'})\cdot \bn_\gamma=0$ and thus $(\pt \bw_{K^\gamma}-\pt \bw_{K'})=[\pt \bw]=[\pt \bw]_\gamma$.  Thus using the triangle inequality, we have for $z \in \calN_{vert, K^\gamma} \cup \calN_{edge, K^\gamma}$
\begin{equation}
\label{eq314}
|(\pt \bw_{K^\gamma}-\pt \bw_{K'})(z)| \le \sum_{e^\gamma \ni z} |(\pt \bw_{K_{1,e^\gamma}^\gamma}-\pt \bw_{K_{2,e^\gamma}^\gamma})(z)| = \sum_{e^\gamma \ni z} |[\pt \bw|_{e^\gamma}]_\gamma(z)|.
\end{equation}
Here $K_{1,e^\gamma}^\gamma, K_{2,e^\gamma}^\gamma \in \calT_{h,k}^\gamma$ are the two mesh elements sharing $e^\gamma$.  As observed in the proof of Lemma \ref{lem:jumps}, we may employ \eqref{eq:ptwisejumps} with $[\pt \bw]_\gamma$ on the left hand side.  Employing \eqref{eq314}, \eqref{eq:ptwisejumps}, inverse estimates, and \eqref{eq:jumpequivs}, we thus obtain
\begin{align}
\begin{aligned}
|(\pt \bw_{K^\gamma}-\pt \bw_{K'})(z)| & \lesssim h \|\bar \bw\|_{L_\infty( \bar \omega_z)} + \sum_{e^\gamma \ni z} \|[\bar \bw]_{\bar\Gamma_h}\|_{L_\infty(\bar e)}
\\ & \lesssim \|\bar \bw\|_{L_2(\bar \omega_z )} + h^{-1/2} \sum_{ e^\gamma \ni z} \|[ \bar \bw]_{\bar \Gamma_h}\|_{L_2(\bar e)}
\\ & \lesssim \|\pt \bw \|_{L_2(\omega_z^\gamma)} + h^{-1/2} \sum_{e^\gamma \ni z} \|[\pt \bw]_{\gamma}\|_{L_2(e^\gamma)}.
\end{aligned}
\end{align}
Here $\omega_z$ is the element patch about $z$.  Employing finite overlap of element patches, we thus obtain
\begin{equation}
\label{eq:West}
\sum_{K \in \calT_{h,k}} \sum_{i=1}^3 \|\nabla_\gamma(\bW_{i,K^\gamma}-\bW_i)\|_{L_2(K^\gamma)}^2 \lesssim \|\pt \bw \|_{L_2(\gamma)}^2+ h^{-1} \sum_{e \in \calE} \|[\pt \bw]_\gamma\|_{L_2(e^\gamma)}^2.
\end{equation}

Given $z \in \calN$, let $\bar z = \bPsi^{-1} \bp^{-1} (z)$ and $z_k = \bp^{-1} (z)$.  Employing $\|\bPhi\|_{L_2(K^\gamma)} \lesssim h$ and $\|\bPhi\|_{H^1(K^\gamma)} \lesssim 1$, we next compute using inverse estimates, $\|\bar g-\bar g(\bar z)\|_{L_\infty(\bar K)} \lesssim h\|\nabla \bar g\|_{L_\infty(\bar K)}$ when $\bar g \in W_\infty^1(\bar K)$ and $\bar z \in \bar K$, \eqref{eq:pprops} and \eqref{eq:piolaprops}, and norm equivalence that
\begin{align}
\label{eq:vWest}
\begin{aligned}
\|\nabla (\pt \bw & -\bW_{K^\gamma})\|_{L_2(K^\gamma)}
\\ & = \|\nabla [ \bPi (\sum_{z \in \calN_{K^\gamma}} (\calP_{\bp} \calP_{\Psi} \bar \bw(\bar z) \ast \bPhi_z-(\calP_{\bp}\circ \calP_{\bPsi})(\bar z) \bar \bw(\bar z) \ast \bPhi_z )]\|_{L_2(K^\gamma)}
\\& \lesssim \|\nabla \bPi\|_{L_\infty(K^\gamma)}\max_{z \in K^\gamma} \|\calP_{\bp} \calP_\bPsi -(\calP_{\bp} \calP_\bPsi)(\bar z)\|_{L_\infty(K^\gamma)} \|\bar \bw \|_{L_\infty(\bar K)} \|\bPhi_z\|_{L_2(K^\gamma)}
\\ & ~~~~+ \|\bPi\|_{L_\infty(K^\gamma)} \|\nabla (\calP_{\bp} \circ \calP_\bPsi) \nabla (\bPsi^{-1} \bp^{-1})\|_{L_\infty(K^\gamma)}\|\bar \bw \|_{L_\infty(K^\gamma)} \|\bPhi_z\|_{L_2(K^\gamma)}
\\ & ~~~~+  \|\bPi\|_{L_\infty(K^\gamma)}\max_{z \in K^\gamma} \|(\calP_{\bp} \calP_\bPsi)(z)-\calP_{\bp} \calP_\bPsi\|_{L_\infty(K^\gamma)} \|\bar \bw \|_{L_\infty(K^\gamma)} \|\nabla \bPhi_z\|_{L_2(K^\gamma)}
\\ & \lesssim h \|\calP_{\bp} \calP_\bPsi\|_{W_\infty^1(\bar K)}\|\bar \bw\|_{L_\infty(\bar K)} \|\nabla_\gamma (\bPsi^{-1} \bp^{-1})\|_{L_\infty(K^\gamma)}
  \lesssim \|\bar \bw\|_{L_2(\bar K)}
 \lesssim \|\pt \bw\|_{L_2(K^\gamma)}.
\end{aligned}
\end{align}
Here we have also used $|\nabla_\gamma \bp^{-1}| =|(\bI-d\bH)^{-1} [ \bI-\frac{\bnu\otimes \bnu_h}{\bnu\cdot \bnu_h}| \lesssim 1$ and $|\nabla (\bPsi^{-1})| \lesssim 1$; the former follow from the smoothness of $d$ and $\bnu_h$ and the latter follows by combining \eqref{eq:map1} and \eqref{eq:map5}.  

Combining \eqref{eq:WL2}, \eqref{eq: Wsplit}, \eqref{eq:West}, and \eqref{eq:vWest} yields
\begin{equation}
\|\pt \bw-\bW\|_{H_h^1(\gamma)}^2 \lesssim \|\pt \bw \|_{L_2(\gamma)} + h^{-1} \sum_{e \in \calE} \|[ \pt \bw]_\gamma\|_{L_2(e^\gamma)}^2 .
\end{equation}
Combining this result with \eqref{eq:contkorn}, we finally obtain
\begin{align}
\label{eq:kornnorho}
\begin{aligned}
\|\pt \bw\|_{H_h^1(\gamma)}^2 &\lesssim \|\pt \bw-\bW\|_{H_h^1(\gamma)}^2 + \|\bW\|_{H^1(\gamma)}^2
\\ & \lesssim \|\pt \bw \|_{L_2(\gamma)}^2 + h^{-1} \sum_{e \in \calE} \|[ \pt \bw]_\gamma\|_{L_2(e^\gamma)}^2 + \|\bW\|_{L_2(\gamma)}^2 + \|{\rm Def}_\gamma \bW\|_{L_2(\gamma)}^2
\\ & \lesssim  \|\pt \bw \|_{L_2(\gamma)}^2 + h^{-1} \sum_{e \in \calE} \|[ \pt \bw]_\gamma\|_{L_2(e^\gamma)}^2 + \|{\rm Def}_{\gamma} \pt \bw \|_{L_2(\gamma)}^2 + \|\pt \bw-\bW\|_{H_h^1(\gamma)}^2
\\ & \lesssim \|\pt \bw \|_{L_2(\gamma)}^2 + h^{-1} \sum_{e \in \calE} \|[ \pt \bw]_\gamma\|_{L_2(e^\gamma)}^2 + \|{\rm Def}_{\gamma} \pt \bw \|_{L_2(\gamma)}^2. 
\end{aligned}
\end{align}
Employing $\rho \ge C >0$ completes the proof. 
\end{proof}

\begin{corollary}
Given $\bw \in \bV_h$, there holds
\begin{equation}
\label{eq:korndiscsurf}
\|\bw\|_{H_h^1(\Gamma_h^k)} \lesssim \|\bw \|_{DG, \Gamma_h^k}, ~~~~\|\bar \bw \|_{H_h^1(\bar \Gamma_h)} \lesssim \|\bar \bw\|_{DG, \bar \Gamma_h}
\end{equation}
\end{corollary}
\begin{proof} \hspace{.1cm}
Employing norm equivalence, \eqref{eq:disckorn}, \eqref{eq:jumpequivs}, \eqref{eq:defequiv} and norm equivalence, an inverse inequality and norm equivalence, and finally norm equivalence again yields
\begin{equation}
\begin{aligned}
\|\bw\|_{H_h^1(\Gamma_h^k)}  & \lesssim \|\pt \bw \|_{H_h^1(\gamma)} \lesssim \|\pt \bw\|_{DG, \gamma} \lesssim \|\bw\|_{DG, \Gamma_h^k} + \|{\rm Def}_{\gamma} \pt \bw \|_{L_2(\gamma)}
\\ & \lesssim \|\bw\|_{DG, \Gamma_h^k} + h^k \|\bw\|_{H_h^1(\Gamma_h^k)}
 \lesssim \|\bw\|_{DG, \Gamma_h^k} + h^k \|\bar \bw\|_{H_h^1(\bar \Gamma_h)}
\\ & \lesssim \|\bw \|_{DG, \Gamma_h^k} + \|\bar \bw \|_{L_2(\bar \Gamma_h)}
 \lesssim \|\bw \|_{DG, \Gamma_h^k}.
\end{aligned}
\end{equation}
This is the first inequality in \eqref{eq:korndiscsurf}.  The second follows by the same argument after momentarily setting $k=1$. 
\end{proof}

Combining these results, we obtain equivalence of the DG {and $\vvvert\cdot \vvvert$} norms on the base discrete, discrete, and continuous surfaces.
\begin{lemma}
Given $\bw \in \bV_h$, there holds
\begin{equation}
\label{eq:dgequiv}
\|\bw\|_{DG, \Gamma_h^k} \simeq \|\pt \bw \|_{DG,\gamma} \simeq \|\bar \bw \|_{DG, \bar \Gamma_h}
\end{equation}
and
\begin{equation}
\label{eq:triplebarequiv}
\|\bw \|_{DG, \Gamma_h^k} \simeq \vvvert \bw \vvvert_{\Gamma_h^k}, \hspace{.2cm} \|\pt \bw \|_{DG, \gamma} \simeq \vvvert \pt \bw \vvvert_\gamma, \hspace{.2cm} \|\bar \bw \|_{\bar \Gamma_h} \simeq \vvvert \bar \bw \vvvert_{\bar \Gamma_h}.
\end{equation}
\end{lemma}
\begin{proof} \hspace{.1cm}
To prove \eqref{eq:dgequiv}, note that equivalence of the $L_2$ and jump terms follows from \eqref{eq:jumpequivs} and $\rho \ge C>0$.  In order to bound the deformation gradient term on each surface by the DG norm on the other surfaces we employ \eqref{eq:hnormequiv} with $m=1$ and the discrete Korn inequalities \eqref{eq:disckorn} and \eqref{eq:korndiscsurf}.  In order to prove \eqref{eq:triplebarequiv}, recall from \eqref{eq:triplebardef} that the $DG$ norms on each surface are trivially bounded by the corresponding $\vvvert \cdot \vvvert$ norms.  To prove the reverse inequality, note that by an inverse inequality and discrete Korn inequality $\|\bw\|_{H_h^1(\Gamma_h^k)} + h \|\bw\|_{H_h^2(\Gamma_h^k)} \lesssim \|\bw\|_{H_h^1(\Gamma_h^k)} \lesssim \|\bw\|_{DG,\Gamma_h^k}$, and similarly on $\bar \Gamma_h$ and $\gamma$.  Inserting this relationship into \eqref{eq:triplebardef} completes the proof.
\end{proof}


\subsection{The BDM interpolant} Let $\bar I_h: H({\rm div}, \bar \Gamma_h) \cap [H_h^1(\bar \Gamma_h)]^3 \rightarrow \bar \bw_h$ be the standard BDM interpolant on the affine discrete surface $\bar \Gamma_h$; the elementwise definition is exactly as in the planar case \cite{BF91}.  Given $\bw \in H({\rm div}, \Gamma_h^k) \cap [H_h^1(\Gamma_h^k)]^3$, we then define $I_h \bw = \calP_{\bPsi}  \bar I_h \bar \bw$.  Let in addition $\pi_h^1$ be the $L_2(\bar\Gamma_h)$ projection onto $\bar Q_h$ and $\pi_h^k$ the $L_2(\Gamma_h^k)$ projection onto $Q_h$.   $\bar I_h$ then satisfies the standard commuting diagram property 
\begin{equation}
\label{eq:cdp}
{\rm div}_{\bar \Gamma_h} I_h \bar \bw = \pi_h^1 {\rm div}_{\bar \Gamma_h} \bar \bw,
\end{equation}
but the corresponding commuting diagram property does {\it not} hold for $I_h$ and $\pi_h^k$.  Instead there holds the ``weak commuting diagram property'' \cite[Proposition 2]{De26} 
\begin{equation}
\begin{aligned}
b_{\Gamma_h^k} ( {\rm div}_{\Gamma_h^k} I_h \bw, q) & = b_{\Gamma_h^k}  ({\rm div}_{\Gamma_h^k} \bw, (\pi_h^1 \bar q ) \circ \bPsi^{-1}), 
\\ & q \in L_2(\Gamma_h^k), ~\bw \in H({\rm div}, \Gamma_h^k)\cap [H_h^1(\Gamma_h^k)]^3.
\end{aligned}
\end{equation}
We shall apply this property in a slightly weaker form.  Namely, since  $(\pi_h^1 \bar q_h ) \circ \bPsi^{-1}=q_h$ for $q_h \in Q_h$, there holds
\begin{equation}
\label{eq:weakcdp}
b_{\Gamma_h^k}( {\rm div}_{\Gamma_h^k} I_h \bw , q_h) = b_{\Gamma_h^k} ({\rm div}_{\Gamma_h^k} \bw, q_h), ~~~q_h \in Q_h, ~~\bw \in H({\rm div}, \Gamma_h^k) \cap [H_h^1(\Gamma_h^k)]^3.
\end{equation}
We also need stability and interpolation properties for $I_h$, which we summarize in the following lemma.  

\begin{lemma}
If for some $1 \le m \le r+1$ there holds $\bw \in H({\rm div}, \Gamma_h^k) \cap [H_h^m(\Gamma_h^k)]^3$, then for $0 \le \ell \le m$ 
\begin{equation}
\label{eq:vecinterp}
\|\bw - I_h \bw\|_{H_h^\ell(\Gamma_h^k)} \lesssim h^{m-\ell}\|\bw\|_{H_h^m(\Gamma_h^k)}
\end{equation}
Also, if $\bw \in H({\rm div}, \Gamma_h^k) \cap [H_h^m(\Gamma_h^k)]^3$ with $2 \le m \le r+1$ then
\begin{equation}
\label{eq:dginterp}
\begin{aligned}
\vvvert \bw-I_h \bw\vvvert_{\Gamma_h^k} & = \|\bw-I_h \bw\|_{DG, \Gamma_h^k} + \|\bw-I_h \bw \|_{H_h^1(\Gamma_h^k)} + h \|\bw -I_h \bw\|_{H_h^2(\Gamma_h^k)} 
\\ & \lesssim h^{m-1} \|\bw \|_{H_h^m(\Gamma_h^k)}.
\end{aligned}
\end{equation}
In addition, if for $0 \le m \le r$ there holds $q \in H_h^m(\Gamma_h^k)$, then for $0 \le \ell \le m$
\begin{equation}
\label{eq:l2interp}
\|q-\pi_h^k q \|_{H_h^\ell(\Gamma_h^k)} \lesssim h^{m-\ell} \|q\|_{H_h^m(\Gamma_h^k)}.
\end{equation}
If $\bw \in {\bf HT}^1(\gamma)$, then
\begin{equation}
\label{eq:interp_stability}
\|\pt{I_h \ipt{\bw}}\|_{DG,\gamma} \lesssim \|\bw\|_{H^1(\gamma)}.
\end{equation}
If $\bw \in H({\rm div}, \Gamma_h^k) \cap H_h^2(\Gamma_h^k)$, then
\begin{equation}
\label{eq:interp_stability2}
\vvvert I_h \bw\vvvert_{\Gamma_h^k} \lesssim \|\bw\|_{DG, \Gamma_h^k} + h\|\bw \|_{H_h^2(\Gamma_h^k)}.
\end{equation}
\end{lemma} 
\begin{proof} \hspace{.1cm}
\eqref{eq:vecinterp} is a standard property of $I_h$ on Euclidean domains (cf. \cite[Proposition 3.6]{BF91}).  It holds elementwise and can be extended to curved triangles using the definition of $I_h$ and norm equivalence.  \eqref{eq:l2interp} is also standard.  To prove \eqref{eq:dginterp} we employ \eqref{eq:vecinterp} for the volume terms on the left hand side.  For the jump terms we employ the scaled elementwise trace inequality \eqref{eq:scaledtrace}, norm equivalence, and \eqref{eq:vecinterp} to compute
\begin{equation}
\label{eq:interpstabilityproof}
\begin{aligned}
 \frac{1}{h} \sum_{e \in \calE} \|[\bw- I_h \bw ]_{\Gamma_h^k}\|_{L_2(e)}^2 & \lesssim h^{-2} \|\bw-I_h \bw\|_{L_2(\Gamma_h^k)}^2 + \|\bw-I_h \bw\|_{H_h^1(\Gamma_h^k)}^2 
 \\ & \lesssim h^{2(m-1)} \|\bw\|_{H_h^m(\Gamma_h^k)}^2.
\end{aligned}
\end{equation}
\eqref{eq:interp_stability} is proved by a similar calculation after noting that for $\bw \in {\bf HT}^1(\gamma)$, $[ \pt{I_h \ipt{\bw}}]_\gamma = [\pt{I_h \ipt{\bw}}-\bw]_\gamma$.  

To prove \eqref{eq:interp_stability2}, we apply \eqref{eq:triplebarequiv}, add and subtract $\bw$ inside the resulting $DG$ norm, employ the triangle inequality, and use \eqref{eq:dginterp} with $m=2$.  
\end{proof} 

%
%

\subsection{Geometric Error Estimate} 
A central step in our error analysis is control of ``geometric errors'' (variational crimes) due to formulating the finite element method on the discrete surface $\Gamma_h^k$ instead of on $\gamma$.  Before stating and proving our main result in this regard we state two technical lemmas.
\begin{lemma}
Given 
$\bw_h\in \bV_h$ we have
\begin{align}\label{inverse_ineq_contin}
h^2
||\mathrm{D}_{\gamma}\mathrm{Def}_{\gamma}\pt \bw_h||^2_{L^2(K^{\gamma})}
&\lesssim
||\pt \bw_h||^2_{H^1_h(K^{\gamma})}, \\
\label{eq:discinverse}
h^2
||\mathrm{D}_{\Gamma_h^k}\mathrm{Def}_{\Gamma_h^k}\bw_h||^2_{L^2(K)}
&\lesssim
||\bw_h||^2_{H^1_h(K)}.
\end{align}
\end{lemma}
\begin{proof} \hspace{.1cm}
 From norm equivalence between surfaces and an inverse inequality on the planar triangle we obtain:
\begin{align*}
h^2
||\mathrm{D}_{\gamma}\mathrm{Def}_{\gamma}\pt \bw_h||^2_{L^2(K^{\gamma})}
&\lesssim
h^2
||\pt \bw_h||^2_{H^2(K^{\gamma})}
\lesssim
h^2||\bar \bw_h||^2_{H^2(\overline{K})} \\
&\lesssim
||\bar \bw_h||^2_{H^1_h(\overline{K})}
\lesssim
||\pt \bw_h||^2_{H^1_h(K^{\gamma})}.
\end{align*}
The second inequality is proved similarly.
\end{proof}

\begin{lemma}
Let $ \bw \in H({\rm div},\Gamma_h^k) \cap H_h^1(\Gamma_h^k)$.  Then
\begin{align}\label{relat_order_in_inf1}
\frac{\rho}{h} \|[\bw^\ell]_{\gamma} - \bPi[\bw^\ell]_{\Gamma_h^k} \|_{L_2(\varSigma)}^2
& \lesssim h^{2k} \|\bw\|_{H_h^1(\Gamma_h^k)}^2,
 \\
\frac{\rho}{h} \|[\bw^\ell]_{\gamma}\|_{L_2(\varSigma^\gamma)}^2
&\lesssim \label{relat_order_in_inf2} h^{2k-2} \|\bw\|_{H_h^1(\Gamma_h^k)}^2 +\frac{\rho}{h}\|[\bw]_{\Gamma_h^k}\|_{L_2(\Sigma)}^2.
\end{align}    
\end{lemma}
\begin{proof} \hspace{.1cm}
Assume first that we are assessing the above quantities on an edge $e=K_1 \cap K_2$ and denote by $\bw_i$ the function $\bw$ evaluated on $K_i$, $i=1,2$.  

Employing $[\bw^\ell]_\gamma=[\bw^\ell]\cdot \bt_\gamma \bt_\gamma$, $[\bw^\ell]_{\Gamma_h^k}=[\bw^\ell]\cdot \bt_h^k \bt_h^k$, adding and subtracting terms, recalling that  $\bPi - \bI = -\bnu \otimes \bnu$, and noting that $[\bw^\ell]\cdot \bnu = \bw_1^\ell \cdot (\bnu-\bnu_{h,1}) -\bw_2^\ell\cdot  (\bnu-\bnu_{h,2})$, we obtain
\begin{equation}
\label{eq337}
\begin{aligned}
|[\bw^\ell]_\gamma & -\bPi [\bw^\ell]_{\Gamma_h^k}|  =| [\bw^\ell]\cdot (\bt_\gamma-\bPi \bt_h^k) \bt_\gamma +[\bw^\ell]\cdot \bPi \bt_h^k (\bt_\gamma-\bPi \bt_h^k) 
\\ & ~~~~ + [\bw^\ell] \cdot (\bPi \bt_h^k -\bt_h^k) \bPi \bt_h^k + [\bw^\ell] \cdot \bt_h^k \bPi \bt_h^k - [\bw^\ell] \cdot \bt_h^k \bPi \bt_h^k|
\\ & =| [\bw^\ell]\cdot (\bt_\gamma-\bPi \bt_h^k) \bt_\gamma +[\bw^\ell]\cdot \bPi \bt_h^k (\bt_\gamma-\bPi \bt_h^k) -[\bw^\ell] \cdot \bnu (\bt_h^k \cdot \bnu) \bPi \bt_h^k|
\\ & =|  [\bw^\ell]\cdot (\bt_\gamma-\bPi \bt_h^k) \bt_\gamma +[\bw^\ell]\cdot \bPi \bt_h^k (\bt_\gamma-\bPi \bt_h^k) 
\\ & ~~~~~ - ( \bw_1^\ell \cdot (\bnu-\bnu_{h,1}) -\bw_2^\ell\cdot  (\bnu-\bnu_{h,2})) (\bt_h^k \cdot (\bnu-\bnu_{h1})) \bPi \bt_h^k |
\\ & \lesssim (h^{k+1} + h^{k+1} + h^{2k} + h^{2k})(|\bw_1^\ell| + |\bw_2^\ell|) \lesssim h^{k+1} (|\bw_1^\ell| + |\bw_2^\ell|).
\end{aligned}
\end{equation}
In the next to last step we have employed \eqref{eq:geoest} and \eqref{eq:tangent_conormal}.  By the scaled trace inequality \eqref{eq:scaledtrace} there holds for $\bw \in [H_h^1(\Gamma_h^k)]^3$ and for $e=K_1 \cap K_2$ that $\|\bw_i\|_{L_2(e)}^2 \lesssim h^{-1} \|\bw\|_{L_2(K_1 \cup K_2)}^2 + h\|\bw \|_{H_h^1(K_1 \cup K_2)}$.  Taking edge $L_2$ norms of both sides of \eqref{eq337}, squaring, and multiplying both sides by $\frac{\rho}{h}$ thus yields
\begin{equation}
\begin{aligned}
\frac{\rho}{h} \|[\bw^\ell]_{\gamma} - \bPi[\bw^\ell]_{\Gamma_h^k} \|_{L_2(\varSigma)}^2 & \lesssim h^{2k+1} ( h^{-1} \|\bw\|_{L_2(K_1 \cup K_2)}^2 + h\|\nabla_{\Gamma_h^k} \bw\|_{L_2(K_1 \cup K_2)}) 
\\ & \lesssim h^{2k} \|\bw\|_{H_h^1(\Gamma_h^k)}^2,  
\end{aligned}
\end{equation}
 which gives the first desired estimate after summing over the edges.  Similarly, 
\begin{equation}
|[\bw^\ell]_\gamma| =\left |  [\bw^\ell]\cdot (\bt_\gamma-\bt_h^k) \bt_\gamma + [\bw^\ell] \cdot \bt_h^k (\bt_\gamma-\bt_h^k)+ [\bw^\ell]\cdot \bt_h^k \bt_h^k \right |  \lesssim h^k |[\bw^\ell]| + |[\bw]_{\Gamma_h^k}|.
\end{equation}
Taking $L_2$ norms of both sides and squaring the result, applying the scaled trace inequality \eqref{eq:scaledtrace} to the first term, and multiplying by $\frac{\rho}{h}$ yields 
\begin{equation} 
\frac{\rho}{h} \|[\bw^\ell]_{\gamma}\|_{L_2(e^\gamma)}^2 \lesssim h^{2k-2} \|\bw\|_{H_h^1(\Gamma_h^k)} +\frac{\rho}{h}\|[\bw]_{\Gamma_h^k}\|_{L_2(e)}^2.
\end{equation}
The second desired result follows after summing over the edges.
\end{proof}

We now define a bilinear form which encodes the main portion of the geometric error.  
\begin{equation}
\label{eq:Gdef}
\begin{aligned}
\mathrm{G} (\bu_h, & \bw_h)
: =
a_{\gamma}(\pt \bu_h, \pt \bw_h)
-
a_{\Gamma_h^k}(\bu_h, \bw_h) \\
  = &   
\left (2\int_{\gamma}
\mathrm{Def}_{\gamma}\pt \bu_h
:
\mathrm{Def}_{\gamma}\pt \bw_h
-
2\int_{\Gamma_h^k}
\mathrm{Def}_{\Gamma_h^k}\bu_h
:
\mathrm{Def}_{\Gamma_h^k}\bw_h
+\int_{\gamma}\pt \bu_h
\cdot
\pt \bw_h
-
\int_{\Gamma_h^k}
\bu_h
\cdot
\bw_h  \right )
\\ 
\hspace{.2cm} & - 
\left(\sum_{e^{\gamma}\in \calE^\gamma}
\int_{e^{\gamma}}[\pt \bu_h]_{\gamma}\cdot\{{\mathrm{Def}_{\gamma}\pt \bw_h}\}_{\gamma}
-
\sum_{e\in \calE}
\int_{e}[\bu_h]_{\Gamma_h^k}
\cdot \{{\mathrm{Def}_{\Gamma_h^k}\bw_h}\}_{\Gamma_h^k}\right) \\  ~~
&-
\left(\sum_{e^{\gamma}\in \calE^\gamma}
\int_{e^{\gamma}}[\pt \bw_h]_{\gamma}\cdot\{{\mathrm{Def}_{\gamma}\pt \bu_h}\}_{\gamma}
-
\sum_{e\in \calE}
\int_{e}[\bw_h]_{\Gamma_h^k}
\cdot \{{\mathrm{Def}_{\Gamma_h^k}\bu_h}\}_{\Gamma_h^k}\right) \\ ~~
&+
\left(\sum_{e^{\gamma}\in \calE^\gamma}
\frac{\rho}{h}
\int_{e^{\gamma}}
[\pt \bu_h]_{\gamma} \cdot
[\pt \bw_h]_{\gamma}
-
\sum_{e\in \calE}
\frac{\rho}{h}\int_{e}
[\bu_h]_{\Gamma_h^k} \cdot
[\bw_h]_{\Gamma_h^k}\right)
\\ =: &  I - II - III + IV. 
\end{aligned}
\end{equation}

\begin{theorem} \label{geotheorem}
Given $\bu_h, \bw_h\in H({\rm div}, \Gamma_h^k) \cap H_h^2(\Gamma_h^k)$, the following inequality holds:
\begin{align} \label{eq:G_bound_mod}
|G(\bu_h, \bw_h)| 
\lesssim
h^k \vvvert \pt \bu_h \vvvert_{\gamma} \vvvert \pt \bw_h \vvvert_{\gamma}.
\end{align}   
If $\bu_h, \bw_h \in \bV_h$, then
\begin{align} \label{G_bound}
|G(\bu_h, \bw_h)| 
\lesssim
h^k \| \pt \bu_h \|_{DG, \gamma} \| \pt \bw_h \|_{DG,\gamma}.
\end{align}   
If instead $\bu, \bw \in {\bf HT}^1(\gamma) \cap [H^2(\gamma)]^3$ and $k\ge 2$, then
\begin{equation}
\label{eq:GL2est}
G(\ipt \bu, \ipt \bw) \lesssim h^{k+1}  \|\bu\|_{H^2(\gamma)}\|\bw\|_{H^2(\gamma)}.  
\end{equation}
\eqref{eq:GL2est} holds when $k=1$ under the additional assumption that for any two vertices $z_1, z_2$ of $\bar \Gamma_h=\Gamma_h^1$ sharing an edge of $\bar \Gamma_h$, there holds $|d(z_1)-d(z_2)| \lesssim h^3$. 

\end{theorem}
\begin{proof} \hspace{.1cm}

\noindent {\bf Proof of \eqref{eq:G_bound_mod}: Geometric error bound for piecewise smooth functions.}
Lemma 5.1 of \cite{DN26} and the definition of $\vvvert \cdot \vvvert$ immediately yield
\begin{align}\label{comparison_0013}
  \left|I\right|    
\lesssim
h^k \vvvert \pt{\bw}_h \vvvert_{\gamma}
\vvvert \pt{\bu}_h\vvvert_{\gamma}.   
\end{align}
Note that \cite{DN26} concerns Taylor-Hood rather than $BDM$ elements.  However, the velocity spaces for both elements are identical elementwise and the proof does not rely on interlement continuity.


We next analyze the term $IV$ in \eqref{eq:Gdef}. After using a change of variables and adding and subtracting terms, we obtain
\begin{align}
\begin{split}
\label{jump_terms_2}
  \frac{\rho}{h}\int_{e}
[\bu_h]_{\Gamma_h^k} 
 \cdot
[\bw_h]_{\Gamma_h^k}
&=  
\frac{\rho}{h}\int_{e}\left(1-\mu_e \right)
[\bu_h]_{\Gamma_h^k} \cdot
[\bw_h]_{\Gamma_h^k}
+
\frac{\rho}{h}\int_{e^{\gamma}}
[\bu_h^\ell]_{\Gamma_h^k} \cdot
[\bw_h^\ell]_{\Gamma_h^k}.
\end{split}
\end{align}
Since $|1-\mu_e| \lesssim h^{k+1}$, after using norm equivalence we obtain for the first term in \eqref{jump_terms_2} that
\begin{equation} 
\begin{aligned} \label{jump_terms_5-11} 
\left|
\sum_{e \in \calE}
\int_{e}
\frac{\rho}{h}
\left(1-\mu_e \right )
[\bu_h]_{\Gamma_h^k} \cdot
[\bw_h]_{\Gamma_h^k}\right|
& \lesssim h^{k+1} \frac{\rho}{h} \sum_{e \in \calE} \|[\bu_h]_{\Gamma_h^k} \|_{L_2(e)} \|[\bw_h]_{\Gamma_h^k}\|_{L_2(e)}
\\ &\lesssim
h^{k+1}
\|\pt{\bu}_h\|_{DG, \Gamma_h^k}
\|\pt{\bw}_h\|_{DG,\Gamma_h^k}. 
\end{aligned}
\end{equation}
Now, we will focus on the second term of \eqref{jump_terms_2}. We start with the definition of the jump of the Piola transformation on the continuous surface. By adding and subtracting terms and noting that $\bPi[\bu_h^\ell]_\gamma = [\bu_h^\ell]_\gamma$, we have
\begin{align} \label{jump_terms_5}
[\pt{\bu}_h]_{\gamma} -[\bu_h^\ell]_\gamma \nonumber
&=
\left[(\bPi-d\mathbf{H})\left(
\frac{\bu_{h1}^\ell}{\mu_{h1}}
-
\frac{\bu_{h2}^\ell}{\mu_{h2}}
\right)\right]\cdot \mathbf{t}_{\gamma}
\mathbf{t}_{\gamma}-  [\bu_h^\ell]_\gamma 
\\
\begin{split}
&=
\left[(\bPi-d\mathbf{H})\left(
\left(\frac{1}{\mu_{h1}}-1\right)\bu_{h1}^\ell
+
\left(1
-
\frac{1}{\mu_{h2}}
\right)\bu_{h2}^\ell
\right)\right]\cdot \mathbf{t}_{\gamma}
\mathbf{t}_{\gamma} 
\\ & ~~+
(\bPi-d\mathbf{H})[\bu_h^\ell]_{\gamma} -[\bu_h^\ell]_\gamma 
\end{split}
\\ 
\begin{split} 
\nonumber 
& = [(\bPi- d \bH) \left (\frac{1}{\mu_h} -1 \right) \bu_h^\ell]_\gamma - d \bH [\bu_h^\ell]_\gamma.
\end{split}
\end{align}
From \eqref{jump_terms_5} we can write the following:
\begin{equation}
     \label{jump_terms_6} 
\begin{aligned}
\sum_{e^{\gamma}}\frac{\rho}{h}\int_{e^{\gamma}} & [\pt{\bu}_h]_{\gamma} \cdot [\pt{\bw}_h]_{\gamma}
 - [\bu^\ell_h]_{\Gamma_h^k} \cdot [\bw^\ell_h]_{\Gamma_h^k}
 \\ & =  \sum_{e^{\gamma}}\frac{\rho}{h}\int_{e^{\gamma}} \left [  ([\pt \bu_h]_\gamma- [\bu_h^\ell]_\gamma)\cdot [\pt \bw_h]_\gamma + [ \bu_h^\ell]_\gamma \cdot ([\pt \bw_h]_\gamma -[\bw_h^\ell]_\gamma) \right ]
 \\  & ~~~~~+ 
 \sum_{e^{\gamma}}\frac{\rho}{h}\int_{e^{\gamma}} \left [ [\bu_h^\ell]_{\gamma}\cdot [\bw_h^\ell]_{\gamma}
- [\bu^\ell_h]_{\Gamma_h^k} \cdot [\bw^\ell_h]_{\Gamma_h^k}\right ]
\\ & = \sum_{e^{\gamma}}\frac{\rho}{h}\int_{e^{\gamma}} \left [ [\bu_h^\ell]_{\gamma}\cdot [\bw_h^\ell]_{\gamma}
- [\bu^\ell_h]_{\Gamma_h^k} \cdot [\bw^\ell_h]_{\Gamma_h^k}\right ]
 \\  &~~~~+ \sum_{e^{\gamma}}\frac{\rho}{h}\int_{e^{\gamma}}   [(\bPi- d \bH) (\frac{1}{\mu_h}-1) \bu_h^\ell]_\gamma \cdot [\pt \bw_h]_\gamma + [\bu_h^\ell]_\gamma  [(\bPi- d \bH) (\frac{1}{\mu_h}-1) \bw_h^\ell]_\gamma
 \\ & ~~~~ - \sum_{e^{\gamma}}\frac{\rho}{h}\int_{e^{\gamma}} d\bH [\bu_h^\ell]_\gamma \cdot [\pt \bw_h]_\gamma + [\bu_h^\ell]_\gamma \cdot (d \bH [\bw_h^\ell]_\gamma).
 \end{aligned}
\end{equation}
Employing $|1-(\mu_h)^{-1}|\lesssim h^{k+1}$, scaled trace and inverse inequalities,  \eqref{relat_order_in_inf2}, and norm equivalence yields
\begin{equation}
\label{jumpterms9}
\begin{aligned}
\sum_{e^{\gamma}} & \frac{\rho}{h}\int_{e^{\gamma}}   [(\bPi- d \bH) \left (\frac{1}{\mu_h} -1\right) \bu_h^\ell]_\gamma \cdot [\pt \bw_h]_\gamma + [\bu_h^\ell]_\gamma  [(\bPi- d \bH) \left (\frac{1}{\mu_h} -1 \right ) \bw_h^\ell]_\gamma
\\ & \lesssim h^{k+1} \sum_{e^{\gamma}} \frac{\rho}{h} \left ( \|\bu_h^\ell\|_{L_2(e^\gamma)}\|[\pt \bw_h^\ell]_\gamma\|_{L_2(e^\gamma)} + \|[\bu_h^\ell]_\gamma\|_{L_2(e^\gamma)} \|\bw_h^\ell\|_{L_2(e^\gamma)} \right ) 
\\ & \lesssim h^k \vvvert \pt{\bu}_h\vvvert_{\gamma}
\vvvert \pt{\bw}_h\vvvert_{\gamma}.
 \end{aligned}
\end{equation}
 Using  $|d|\lesssim h^{k+1}$, \eqref{relat_order_in_inf2}, and norm equivalence, we next obtain 
\begin{equation}
 \label{jump_terms7} 
\begin{aligned}
\sum_{e^{\gamma}}\frac{\rho}{h}\int_{e^{\gamma}} d\bH [\bu_h^\ell]_\gamma \cdot [\pt \bw_h]_\gamma + [\bu_h^\ell]_\gamma \cdot (d \bH [\bw_h^\ell]_\gamma) \lesssim
h^{k+1}
\vvvert\pt{\bu}_h\vvvert_{\gamma}
\vvvert\pt{\bw}_h\vvvert_{\gamma}.
\end{aligned}
\end{equation}

We finally focus on the last term of \eqref{jump_terms_6}.  Note that $\bPi[\bu_h^\ell]_{\Gamma_h^k} \cdot \bPi [ \bw_h^\ell]_{\Gamma_h^k} = \bPi [ \bu_h^\ell]_{\Gamma_h^k} \cdot [\bw_h^\ell]_{\Gamma_h^k}$ since $\bPi^\intercal = \bPi=\bPi^2$.  
As a result,  after adding and subtracting terms we obtain
\begin{align} \label{jump_terms_new1}
\begin{split}
\sum_{e^{\gamma}}\frac{\rho}{h} & 
\int_{e^{\gamma}}[\bu_h^\ell]_{\gamma} 
\cdot
[\bw_h^\ell]_{\gamma}
-
[\bu_h^\ell]_{\Gamma_h^k} 
\cdot
[\bw_h^\ell]_{\Gamma_h^k}
=
\sum_{e^{\gamma}}\frac{\rho}{h}\int_{e^{\gamma}}[\bu_h^\ell]_{\gamma} 
\cdot
([\bw_h^\ell]_{\gamma}-\bPi[\bw_h^\ell]_{\Gamma_h^k}) \\
&+
\sum_{e^{\gamma}}\frac{\rho}{h}\int_{e^{\gamma}}([\bu_h^\ell]_{\gamma}-\bPi[\bu_h^\ell]_{\Gamma_h^k})
\cdot
\bPi [\bw_h^\ell]_{\Gamma_h^k} 
+
\sum_{e^{\gamma}}\frac{\rho}{h}\int_{e^{\gamma}}\left[(\bPi-\mathbf{I})
[\bu_h^\ell]_{\Gamma_h^k}\right] 
\cdot
[\bw_h^\ell]_{\Gamma_h^k}. 
\end{split}
\end{align}
We will start with the last term of \eqref{jump_terms_new1}. Since $(\bnu\bnu^\top)^2=\bnu\bnu^\top$ and $(\bnu\bnu^\top)^\top=\bnu\bnu^\top$ we have the following result:
\begin{align}\label{jump_terms_new2}
\left[(\bPi-\mathbf{I})
[\bu_h^\ell]_{\Gamma_h^k}\right] 
\cdot
[\bw_h^\ell]_{\Gamma_h^k}
&=
-
(\bnu\bnu^\top
[\bu_h^\ell]_{\Gamma_h^k}, [\bw_h^\ell]_{\Gamma_h^k})
=
-
(\bnu\bnu^\top
[\bu_h^\ell]_{\Gamma_h^k}, \bnu\bnu^\top[\bw_h^\ell]_{\Gamma_h^k}).
\end{align}
Since  $\mathbf{t}_h^k \perp \bnu^k_{hi}$ we have
\begin{align} \label{jump_terms_new3}
\bnu\bnu^\top
[\bu_h]_{\Gamma_h^k}
=
\bnu (\mathbf{t}_h^k,  [\bu^\ell_h])
(\bnu-\bnu^k_{hi}, \mathbf{t}_h^k).
\end{align}
Using $|\bnu-\bnu^k_{hi}| \lesssim h^k$, \eqref{jump_terms_new2}, \eqref{jump_terms_new3}, and norm equivalence  we obtain
\begin{align}\label{jump_terms_new4} 
 \left
 |\sum_{e^{\gamma}}
\int_{e^{\gamma}} \frac{\rho}{h}
\left[(\bPi-\mathbf{I})
[\bu_h^\ell]_{\Gamma_h^k}\right] 
\cdot
[\bw_h^\ell]_{\Gamma_h^k}  
\right|
&\lesssim
h^{2k}\vvvert \pt{\bu}_h\vvvert_{\gamma}
\vvvert \pt{\bw}_h\vvvert_{\gamma}.
\end{align}
Using \eqref{relat_order_in_inf1}, \eqref{relat_order_in_inf2}, and norm equivalence yields 
\begin{align} \label{jump_terms_new5}
\begin{aligned}
 \left
 |\sum_{e^{\gamma}}
\int_{e^{\gamma}} \frac{\rho}{h}
[\bu_h^\ell]_{\gamma} 
\cdot
([\bw_h^\ell]_{\gamma}-\bPi[\bw_h^\ell]_{\Gamma_h^k})
\right|
 & +  \left |\sum_{e^{\gamma}}
\int_{e^{\gamma}} \frac{\rho}{h}
([\bu_h^\ell]_{\gamma}-\bPi[\bu_h^\ell]_{\Gamma_h^k})\cdot ( \bPi [\bw_h^\ell]_{\Gamma_h^k}) 
\right|
\\ & \lesssim
h^k
\vvvert \pt{\bu}_h\vvvert_{\gamma}
\vvvert \pt{\bw}_h\vvvert_{\gamma}.
\end{aligned}
\end{align}
From inequalities \eqref{jump_terms_5-11}, \eqref{jump_terms7}, \eqref{jumpterms9}, \eqref{jump_terms_new1}, 
 \eqref{jump_terms_new4}, and \eqref{jump_terms_new5} 
we obtain the following upper bound:
\begin{equation}\label{G_bound_jumps}
\left|IV\right|
\lesssim
h^k \vvvert \pt{\bu}_h\vvvert_{\gamma}
\vvvert \pt{\bw}_h\vvvert_{\gamma}.
\end{equation}
\\
\indent
We will continue with the jump-average terms defined in $II$. After using a change of variable and adding and subtracting terms, we have
\begin{align}\label{jump_av_1-1}
\begin{split}
 \int_{e}
[\bu_h]_{\Gamma^k_h}
\cdot \{{\mathrm{Def}_{\Gamma^k_h}\bw_h}\}_{\Gamma^k_h} 
&=
\int_{e^{\gamma}}
\left(\frac{1}{\mu_e}-1\right)[\bu_h^\ell]_{\Gamma_h^k}
\cdot \{\mathrm{Def}_{\Gamma_h^k}\bw_h \}_{\Gamma^k_h}^\ell \\
&+
\int_{e^{\gamma}}
[\bu_h^\ell]_{\Gamma^k_h}
\cdot \{\mathrm{Def}_{\Gamma_h^k}\bw_h\}_{\Gamma^k_h}^\ell.
\end{split}
\end{align}
By using $|1-({\mu_e)}^{-1}|\lesssim h^{k+1}$, a scaled trace inequality \eqref{eq:scaledtrace} on the $\mathrm{Def}$ part, $|\bPi-\bPi_h|\lesssim h^k$, a change of variable, \eqref{eq:defequiv}, and norm equivalence, we have
\begin{align}\label{jump_av_1-2}
\begin{aligned}
\sum_{e^{\gamma} \in \calE^\gamma} & 
\left|\int_{e^{\gamma}}
\left(\frac{1}{\mu_e}-1\right)  [\bu_h^\ell]_{\Gamma_h^k}
\cdot   \{\mathrm{Def}_{\Gamma^h_k}\bw_h\}_{\Gamma_h^k}^\ell
\right| 
\\ &  \lesssim h^{k+1} h^{-1/2} \|[ \bu_h^\ell]_{\Gamma_h^k}\|_{L_2(\Sigma^\gamma)} h^{1/2} \|
 \{\mathrm{Def}_{\Gamma^h_k}\bw_h\}_{\Gamma_h^k}^\ell \|_{L_2(\Sigma^\gamma)}
\\ &\lesssim
h^{k+1}||\bu_h||_{DG,\Gamma_h^k} \vvvert \bw_h\vvvert_{\Gamma_h^k} 
\lesssim
h^{k+1}\vvvert \pt{\bu}_h\vvvert_{\gamma}
\vvvert \pt{\bw}_h\vvvert_{\gamma}.
\end{aligned} 
\end{align}
By the definition of the averages we have
\begin{align} \label{jump_av_2}
\begin{split}
[\pt{\bu}_h]_{\gamma}
\cdot
\{{\mathrm{Def}_{\gamma}\pt{\bw}_h}\}_{\gamma}
&=   
[\pt{\bu}_h]_{\gamma}\cdot  \left  ( \mathrm{Def}_{\gamma}\pt{\bw}_{h1}\bn_{{\gamma}1}
-
\mathrm{Def}_{\gamma}\pt{\bw}_{h2}\bn_{{\gamma}2} \right )\\
[\bu_h^\ell]_{\Gamma^k_h}
\cdot \{\mathrm{Def}_{\Gamma_h^k}\bw_h\}_{\Gamma_h^k}^\ell
&= 
[\bu_h^\ell]_{\Gamma^k_h}
\cdot 
\left ( (\mathrm{Def}_{\Gamma^k_h}\bw_{h1})^\ell\bn_{h1}^k 
- (\mathrm{Def}_{\Gamma^k_h}\bw_{h2})^\ell\bn_{h2}^k \right ).
\end{split}
\end{align}
We will compare the integrals of the terms given in \eqref{jump_av_2}. For $i\in\{1,2\}$, by adding and subtracting terms, we have:
\begin{align}
\begin{split}
\label{jump_av_add_subtract}
\int_{e^{\gamma}}[\pt{\bu}_h]_{\gamma}\cdot \mathrm{Def}_{\gamma}\pt{\bw}_{hi}\bn_{{\gamma}i}
&-
[\bu_h^\ell]_{\Gamma^k_h}
\cdot 
((\mathrm{Def}_{\Gamma^k_h}\bw_{hi})^\ell\bn_{hi}^k \\
&=
\int_{e^{\gamma}}[\pt{\bu}_h]_{\gamma}\cdot \mathrm{Def}_{\gamma}\pt{\bw}_{hi}(\bn_{{\gamma}i}
-\bPi\bn_{hi}^k) \\
&+
\int_{e^{\gamma}}([\pt{\bu}_h]_{\gamma}
-
[\bu_h^\ell]_{\Gamma^k_h})
\cdot 
\mathrm{Def}_{\gamma}\pt{\bw}_{hi}\bPi\bn_{hi}^k\\
&+
\int_{e^{\gamma}}[\bu_h^\ell]_{\Gamma^k_h}
\cdot 
(\mathrm{Def}_{\gamma}\pt{\bw}_{hi}
-
(\mathrm{Def}_{\Gamma^k_h}\bw_{hi})^\ell)\bPi\bn_{hi}^k\\
&+
\int_{e^{\gamma}}[\bu_h^\ell]_{\Gamma^k_h}
\cdot 
((\mathrm{Def}_{\Gamma^k_h}\bw_{hi})^\ell(\bPi\bn_{hi}^k
-\bn_{hi}^k).
\end{split}
\end{align}

 Using $|\bn_{{\gamma}i}-\bPi\bn_{hi}^k|\lesssim h^{k+1}$ from \eqref{eq:tangent_conormal}, multiplying  the first and second factors in the integral below respectively by $h^{-1/2}$ and $h^{1/2}$ as in \eqref{jump_av_1-2}, 
employing a scaled trace inequality for the second term, and using  \eqref{inverse_ineq_contin} and the definition of the DG norm, we obtain:
\begin{align}\label{jump_av_add_subtract_1}
\left|\sum_{e^{\gamma} \in \calE^\gamma}
\int_{e^{\gamma}}  [\pt{\bu}_h]_{\gamma}\cdot  \mathrm{Def}_{\gamma}\pt{\bw}_{hi}(\bn_{{\gamma}i}
-\bPi\bn_{hi}^k)\right|  
\lesssim
h^{k+1}\| \pt{\bu}_h\|_{DG, \gamma}
\vvvert \pt{\bw}_h\vvvert_{\gamma}
\lesssim 
h^{k+1}\vvvert \pt{\bu}_h\vvvert_{\gamma}
\vvvert \pt{\bw}_h\vvvert_{\gamma}.
\end{align}
\indent
Define $\mathrm{Def}_{\gamma}\pt{\bw}_{hi}\bPi\bn_{hi}^k=\mathbf{z}$. Since $\bPi\mathbf{z}=\mathbf{z}$, we have $([\pt{\bu}_h]_{\gamma}
-
[\bu_h^\ell]_{\Gamma^k_h})\cdot \mathbf{z}=([\pt{\bu}_h]_{\gamma}
-
\bPi[\bu_h^\ell]_{\Gamma^k_h})\cdot \mathbf{z}$. Multiplying the first and second terms below by $h^{-1/2}$ and $h^{1/2}$ respectively and then using this identity, \eqref{relat_order_in_inf1}, and a trace inequality for the second term we obtain
\begin{align}\label{jump_av_add_subtract_2} 
\left|\sum_{e^{\gamma} \in \calE^\gamma} \int_{e^\gamma}
([\pt{\bu}_h]_{\gamma}
-
[\bu_h^\ell]_{\Gamma^k_h})
\cdot 
\mathrm{Def}_{\gamma}\pt{\bw}_{hi}\bPi\bn_{hi}^k
\right|
&\lesssim
h^k\vvvert \pt{\bu}_h\vvvert_{\gamma}
\vvvert \pt{\bw}_h\vvvert_{\gamma}.
\end{align}
\indent
Multiplying the first and second terms below by $h^{-1/2}$ and $h^{1/2}$ respectively and employing \eqref{eq:defequiv}, a trace inequality \eqref{eq:scaledtrace} for the second term, and a change of variables, we obtain:
\begin{align}\label{jump_av_add_subtract3} \begin{aligned}
\left|\sum_{e^{\gamma} \in \calE^\gamma} \int_{e^\gamma}
[\bu_h^\ell]_{\Gamma^k_h}
\cdot 
(\mathrm{Def}_{\gamma}\pt{\bw}_{hi}
-
(\mathrm{Def}_{\Gamma^k_h}\bw_{hi})^\ell)\bPi\bn_{hi}^k
\right|
&\lesssim
h^k||\pt{\bu}_h||_{DG,\gamma}\vvvert \pt{\bw}_h\vvvert_{\gamma} 
\\ & \lesssim 
h^k\vvvert \pt{\bu}_h\vvvert_{\gamma}
\vvvert \pt{\bw}_h\vvvert_{\gamma}.
\end{aligned}
\end{align}
\indent
Using $|\bPi-\bPi_{hi}|\lesssim h^k$, $\bPi_{hi}\bn_{hi}^k=\bn_{hi}^k$, change of variables, a scaled trace inequality for the second term, \eqref{eq:tangent_conormal}, and norm equivalence we obtain
\begin{align}\label{jump_av_add_subtract4} 
 \left|\sum_{e^{\gamma} \in \calE^\gamma} \int_{e^\gamma}
[\bu_h^\ell]_{\Gamma^k_h}
\cdot 
((\mathrm{Def}_{\Gamma^k_h}\bw_{hi})^\ell(\bPi\bn_{hi}^k
-\bn_{hi}^k) 
\right|
&\lesssim
h^k||\pt{\bu}_h||_{DG,\gamma}
\vvvert \pt{\bw}_h\vvvert_{\gamma} \lesssim h^k \vvvert \pt{\bu}_h\vvvert_{\gamma}\vvvert \pt{\bw}_h\vvvert_{\gamma}.
\end{align}
As a result, from \eqref{jump_av_1-2}, \eqref{jump_av_add_subtract}, \eqref{jump_av_add_subtract_1}, \eqref{jump_av_add_subtract_2}, \eqref{jump_av_add_subtract3}, and \eqref{jump_av_add_subtract4} we obtain
\begin{align}\label{jump_av_add_subtract5} 
|II|
\lesssim h^k
\vvvert \pt{\bu}_h\vvvert_{\gamma}\vvvert \pt{\bw}_h\vvvert_{\gamma}.
\end{align}
\indent
Similar analysis can be done for the other term for the jump-average part. Hence
\begin{align}\label{jump_av_add_subtract5_add} 
|III|
\lesssim
h^k \vvvert \pt{\bu}_h\vvvert_{\gamma}\vvvert \pt{\bw}_h\vvvert_{\gamma}.
\end{align}
As a result, from \eqref{comparison_0013}, \eqref{G_bound_jumps}, \eqref{jump_av_add_subtract5}, and \eqref{jump_av_add_subtract5_add} we obtain the upper bound \eqref{eq:G_bound_mod}.

\vspace{.2cm}
\noindent {\bf Proof of \eqref{G_bound}: Geometric error bound for functions in $\bV_h$.}  \eqref{G_bound} follows directly from \eqref{eq:G_bound_mod} by employing \eqref{eq:triplebarequiv} (equivalence of the $DG$ and $\vvvert \cdot \vvvert$ norms on $\bV_h$).

\vspace{.2cm}
\noindent {\bf Proof of \eqref{eq:GL2est}: Geometric error bound for smooth functions.}   First, by \cite[Lemma 5.2]{DN26} there holds
\begin{equation}
\label{eq:geovol}
\begin{aligned}
2\int_{\gamma} \mathrm{Def}_{\gamma} \bu 
:\mathrm{Def}_{\gamma}\bw
& + \int_{\gamma} \bu\cdot \bw
- \left (2\int_{\Gamma_h^k}\mathrm{Def}_{\Gamma_h^k}\ipt \bu
:\mathrm{Def}_{\Gamma_h^k}\ipt \bw
+
\int_{\Gamma_h^k} \ipt\bu \cdot \ipt \bw \right) 
\\ & \lesssim h^{k+1} \|\bu\|_{H^2(\gamma)}\|\bw\|_{H^2(\gamma)}.
\end{aligned}
\end{equation}
Also, $j_\gamma(\bu, \bw)=0$ because $\bu, \bw \in {\bf HT}^1(\gamma)$.  It is thus left to bound $j_{\Gamma_h^k}(\ipt \bu, \ipt \bw)$.  Let $e=K_1 \cap K_2$.  Using \eqref{eq:pioladef} and \eqref{eq:mudef}, we compute that on $e$ 
\begin{equation}
\label{eq:jumpbound}
\begin{aligned}
 | & [\ipt \bu]  \cdot \bt_h^k |   = \left | (\ipt \bu_{K_1}-\ipt \bu_{K_2}) \cdot \bt_h^k \right |
\\ & = \Big|\Big ((1-d\kappa_1)(1-d\kappa_2) ( (\bnu \cdot \bnu_{h1}-\bnu \cdot \bnu_{h2})\bI
(\bnu \otimes \bnu_{h1}-\bnu \otimes \bnu_{h2}) ) [\bI - d\bH]^{-1} \bu \Big ) \cdot \bt_h^k \Big |
\\ & \le \left | (1-d\kappa_1)(1-d\kappa_2)  (\bnu \cdot \bnu_{h1}-\bnu \cdot \bnu_{h2}) (\bt_h^k)^\top  [\bI - d\bH]^{-1} \bu \right | 
\\ & ~~~~~+ \left |(1-d\kappa_1)(1-d\kappa_2) (\bt_h^k)^\top (\bnu-\bnu_{h1}) \otimes (\bnu_{h1}-\bnu_{h2}) [\bI - d\bH]^{-1} \bu \right |
\\ & \lesssim h^{2k} |\bu|.
\end{aligned}
    \end{equation}
Here we have used $|\bt_h^k\cdot \bnu |= |\bt_h^k \cdot (\bnu-\bnu_{h1})| \lesssim h^{k}$, $|\bnu \cdot \bnu_{h1}-\bnu \cdot \bnu_{h2}| \le |\bnu \cdot \bnu_{h1}-1| + |\bnu\cdot \bnu_{h2} -1| =\frac{1}{2} (|\bnu-\bnu_{h1}|^2+ |\bnu-\bnu_{h2}|^2) \lesssim h^{2k}$,  and \eqref{eq:geoest}.  Taking edgewise norms as needed and employing Cauchy-Schwarz, a scaled trace inequality, and norm equivalence to each term in $j_{\Gamma_h^k} (\ipt \bu, \ipt \bw)$ then yields 
\begin{equation}
\begin{aligned}
| j_{\Gamma_h^k} &(\ipt \bu, \ipt \bw)|  \lesssim  \sum_{e \in \calE}  \left [ \|\{{\rm Def}_{\Gamma_h^k} \ipt \bu \}_{\Gamma_h^k}\|_{L_2(e)} h^{2k} \| \bw^\ell\|_{L_2(e)} + h^{2k}\|\bu^\ell\|_{L_2(e)} \|\{{\rm Def}_{\Gamma_h^k} \ipt \bw \}_{\Gamma_h^k}\|_{L_2(e)} \right ]  
\\ &  ~~~~ +\frac{\rho}{h} \sum_{e \in \calE} h^{2k}\|\bu^\ell\|_{L_2(e)} h^{2k}\|\bw^\ell\|_{L_2(e)}
\\ & \lesssim h^{2k} \sum_{e \in \calE} \Big [(h^{-1/2} \|\bu^\ell\|_{H_h^1(K_{1e} \cup K_{2e})} + h^{-1/2} \|\ipt \bu\|_{H_h^1(K_{1e} \cup K_{2e})} +h^{1/2}\|\ipt \bu\|_{H_h^2(K_{1e} \cup K_{2e})}) 
\\ & ~~~\cdot(h^{-1/2} \|\bw^\ell\|_{H_h^1(K_{1e} \cup K_{2e})} + h^{-1/2} \|\ipt \bw \|_{H_h^1(K_{1e} \cup K_{2e})} +h^{1/2}\|\ipt \bw\|_{H_h^2(K_{1e} \cup K_{2e})}) \Big ] 
\\ & \lesssim h^{2k-1} \|\bu\|_{H^2(\gamma)} \|\bw\|_{H^2(\gamma)}.
\end{aligned}
\end{equation}
Combining this bound with \eqref{eq:geovol} completes the proof of \eqref{eq:GL2est} when $k \ge 2$ after noting that in this case $2k-1 \ge k+1$.  

\vspace{.2cm}
\noindent{\bf Proof of \eqref{eq:GL2est} when $k=1$.}  When $k=1$ the proof of \eqref{eq:GL2est} is significantly more involved, so we divide it into several steps. Let $\calE \ni e=K_1 \cap K_2 \subset \bar \Gamma_h$.  We denote by $\bn_{h1}$ and $\bn_{h2}$ the outward unit cornomals to $K_1, K_2$ on $e$ (so that $\bn_{h1} \cdot \bn_{h2}<0$) and let $\bt_h$ denote the unit tangent on $e$ pointing in the counterclockwise direction on $K_1$.  We also assume that $\bn \cdot \bn_{h1}>0$.  We additionaly define $\beta(e), \beta(\gamma)$ as generic quantities satisfying
\begin{equation}
\label{eq:betadef}
|\beta(e)| \lesssim  h^2 \|\bu\|_{H^2(K_1^\gamma \cup K_2^\gamma)} \|\bw\|_{H^2(K_1^\gamma \cup K_2^\gamma)}, 
 \hspace{.5cm} |\beta(\gamma)| \lesssim  h^2 \|\bu\|_{H^2(\gamma )} \|\bw\|_{H^2(\gamma)}.
    \end{equation}
Note that $\sum_{e \in \calE} |\beta(e)| \lesssim |\beta(\gamma)|$ due to Cauchy-Schwarz and finite overlap of element pairs, and that bounding $|j_{\Gamma_h^k} (\ipt \bu, \ipt \bw)|$ by $|\beta(\gamma)|$ yields the desired result \eqref{eq:GL2est} by the second inequality in \eqref{eq:betadef}.  

\noindent {\bf Step 1:  Bound the jump-jump terms.}  Using \eqref{eq:jumpbound}, a scaled trace inequality \eqref{eq:scaledtrace}, and norm equivalence yields
\begin{equation}
\frac{\rho}{h} \sum_{e \in \calE} \int_e [\ipt \bu]_{\Gamma_h^k} \cdot [\ipt \bw]_{\Gamma_h^k} \lesssim h^{4k-2} \|\bu\|_{H_h^1(\gamma)}\|\bw\|_{H_h^1(\gamma)} \lesssim h^2 \|\bu\|_{H_h^1(\gamma)}\|\bw\|_{H_h^1(\gamma)} \lesssim |\beta(\gamma)|.
\end{equation}
In the last step we have used $4k-2=2$ when $k=1$.  

\noindent {\bf Step 2:  Reduce the average-tangential jump terms to average-full jump terms.}
 On $e$ we have 
\begin{equation}
   [\ipt \bw] = (\ipt \bw_1-\ipt \bw_2) = [\ipt \bw]_{\Gamma_h^k} + (\ipt \bw \cdot \bn_{h1}\bn_{h1} -\ipt \bw \cdot \bn_{h2} \bn_{h2}).
\end{equation}
Next we note that 
\begin{equation}
\begin{aligned}
\Pi_{h1} (\bn_{h1}+ \bn_{h2}) & =\bn_{h1} + \Pi_{h1} ( \bn_{h2} \cdot \bnu_{h1} \bnu_{h1} + \bn_{h2} \cdot \bt_{h} \bt_h + \bn_{h2} \cdot \bn_{h1} \bn_{h1})
\\ & = (1+\bn_{h1} \cdot \bn_{h2}) \bn_{h1} \lesssim h^2,
\end{aligned}
\end{equation}
and similarly $\Pi_{h2} (\bn_{h1} + \bn_{h2}) = (1+\bn_{h1} \cdot \bn_{h2}) \bn_{h2}$.  Here we have used $\bn_{h2} \cdot \bt_h=0$, $\Pi_{h1} \bnu_{h1}=0$, and $1+\bn_{h1}\cdot \bn_{h2}=\frac{1}{2} |\bn_{h1} + \bn_{h2}|^2\lesssim h^2$. 
In addition, because $\ipt \bw \in H({\rm div},\bar \Gamma_h)$ there holds $\ipt \bw_1 \cdot \bn_{h1}+ \ipt \bw_2 \cdot \bn_{h2}=0$.  Thus 
\begin{equation}
\ipt \bw_1 \cdot \bn_{h1}\bn_{h1} -\ipt \bw_2 \cdot \bn_{h2} \bn_{h2}=\ipt \bw_1 \cdot \bn_{h1} (\bn_{h1}+\bn_{h2})= - \ipt \bw_2 \cdot \bn_{h2} (\bn_{h1} + \bn_{h2}).  
\end{equation}
Thus denoting by ${\rm Def}_{\Gamma_{hi}}$ the rate of strain tensor computed on $K_i$ and recalling that $\bPi_{hi} {\rm Def}_{\Gamma_{hi}} \ipt \bu_i={\rm Def}_{\Gamma_{hi}} \ipt \bu_i$, we have
\begin{equation}
\begin{aligned}
\int_e  & ({\rm Def}_{\bar \Gamma_{h1}} \ipt \bu_1 \cdot \bn_{h1} -{\rm Def}_{\bar \Gamma_{h2}} \ipt \bu_2 \cdot \bn_{h2}) (\ipt \bw_1 \cdot \bn_{h1}\bn_{h1} -\ipt \bw_2 \cdot \bn_{h2} \bn_{h2}) 
\\ & =  \int_{e} \Big [ {\rm Def}_{\bar \Gamma_{h1}} \ipt \bu_1 \cdot \bn_{h1} (\ipt \bw_1 \cdot \bn_{h1}) \bPi_{h1}( \bn_{h1} + \bn_{h2})
\\ & ~~~~+{\rm Def}_{\bar \Gamma_{h2}} \ipt \bu_2 \cdot \bn_{h2} (\ipt \bw_2 \cdot \bn_{h2}) \bPi_{h2} ( \bn_{h1} + \bn_{h2}) \Big ]
\\ & = (1+ \bn_{h1} \cdot \bn_{h2}) \int_e \left [ (\bn_{h1}^\intercal {\rm Def}_{\bar \Gamma_{h1}} \ipt \bu_1 \cdot \bn_{h1})( \ipt \bw_1 \cdot \bn_{h1})+ (\bn_{h2}^\intercal {\rm Def}_{\bar \Gamma_{h2}} \ipt \bu_2 \cdot \bn_{h2}) (\ipt \bw_2 \cdot \bn_{h2})  \right ]
 \\ & = (1+ \bn_{h1} \cdot \bn_{h2}) \int_e [\bn_{h1}^\intercal {\rm Def}_{\bar \Gamma_{h1}} \ipt \bu_1 \cdot \bn_{h1}-\bn_{h2}^\intercal {\rm Def}_{\bar \Gamma_{h2}} \ipt \bu_2 \cdot \bn_{h2}] \ipt \bw_1 \cdot \bn_{h1}.
\end{aligned}
\end{equation}
Using $1+\bn_{h1} \cdot \bn_{h2} \lesssim h^2$, $|\bn_{h1}+\bn_{h2}| + |\bn_{h1}-\bn| \lesssim h$, \eqref{eq:defequiv}, a scaled trace inequality, and norm equivalence we find that 
\begin{equation}
\label{eq:edgetobulk}
\begin{aligned}
    \int_e  \{   {\rm Def}_{\bar \Gamma_h} \ipt \bu   \}_{\bar \Gamma_h}  (\ipt \bw_1 \cdot \bn_{h1}\bn_{h1} -\ipt \bw_2 \cdot \bn_{h2} \bn_{h2}) 
     \lesssim h^3  \int_e (|\ipt \bu|+|\nabla_{\bar \Gamma_h} \ipt \bu|) |\bw^\ell|
\lesssim   |\beta(e)| 
 \end{aligned}
\end{equation}
and thus
\begin{equation}
\left |\sum_{e \in \calE} \int_e  \{   {\rm Def}_{\bar \Gamma_h} \ipt \bu   \}_{\bar \Gamma_h} [\ipt \bw]_{\bar \Gamma_h} \right |
\lesssim \left |\sum_{e \in \calE} \int_e  \{   {\rm Def}_{\bar \Gamma_h} \ipt \bu   \}_{\bar \Gamma_h}[\ipt \bw] \right |   + |\beta(\gamma)|.
\end{equation}

\noindent {\bf Step 3:  Split the average-jump terms.}  Calculating as in \eqref{eq:jumpbound} yields 
\begin{equation}
[\ipt \bw] = (1-d\kappa_1)(1-d\kappa_2) \left [ (\bnu \cdot \bnu_{h1}-\bnu\cdot \bnu_{h2})\bI -(\bnu \otimes \bnu_{h1}-\bnu \otimes \bnu_{h2})\right ](\bI -d \bH)^{-1} \bw^\ell.    
\end{equation}
Noting that $|d| \lesssim h^2$, $|\bI-(\bI-d\bH)^{-1}|\lesssim h^2$,  $|\bPi_{hi} \left [ (\bnu \cdot \bnu_{h1}-\bnu\cdot \bnu_{h2})\bI -(\bnu \otimes \bnu_{h1}-\bnu \otimes \bnu_{h2})\right ] | \lesssim h$ for $i=1,2$, and calculating as in the last inequality in \eqref{eq:edgetobulk} yields
\begin{equation}
\begin{aligned}
\int_{e} &  \{ {\rm Def}_{\bar \Gamma_h} \ipt \bu \cdot \bn_{h}^1 \}_{\bar \Gamma_h} \cdot [\ipt \bw] 
\\ & =
 \frac{1}{2} \int_e ({\rm Def}_{\bar \Gamma_{h1}} \ipt \bu_1 \cdot \bn_{h1} -{\rm Def}_{\bar \Gamma_{h2}} \ipt \bu_2 \cdot \bn_{h2}) 
 \\ & ~~~~~~~~\cdot \left [ (\bnu \cdot \bnu_{h1}-\bnu\cdot \bnu_{h2})\bI -(\bnu \otimes \bnu_{h1}-\bnu \otimes \bnu_{h2})\right ] \bw^\ell \\ & ~~~~~ +  \beta(e)
\\ & = \frac{1}{2} (I_e-II_e + III_e) +  \beta(e),
\end{aligned}
\end{equation}
where 
\begin{equation}
\begin{aligned}
    I_e &  = \int_e ({\rm Def}_{\bar \Gamma_{h1}} \ipt \bu_1 \cdot \bn_{h1} -{\rm Def}_{\bar \Gamma_{h2}} \ipt \bu_2 \cdot \bn_{h2}) (\bnu\cdot \bnu_{h1}-\bnu \cdot \bnu_{h2}) \bw^\ell,
    \\ II_e & = \int_e \left [ (  {\rm Def}_{\bar \Gamma_{h1}} \ipt \bu_1 \cdot \bn_{h1}) \bnu \otimes \bnu_{h1} +({\rm Def}_{\bar \Gamma_{h2}} \ipt \bu_2 \cdot \bn_{h2}) \bnu \otimes \bnu_{h2}) \right ] \bw^\ell,
    \\ III_e & = \int_e \left [ ({\rm Def}_{\bar \Gamma_{h1}} \ipt \bu_1 \cdot \bn_{h1} ) \bnu \otimes \bnu_{h2} +({\rm Def}_{\bar \Gamma_{h2}} \ipt \bu_2 \cdot \bn_{h2}) \bnu \otimes \bnu_{h1}  \right ]\bw^\ell.
    \end{aligned}
\end{equation}

In the last three steps we complete the proof by bounding $\sum_{e \in \calE} I_e$, $\sum_{e \in \calE} II_e$, and $\sum_{e \in \calE} III_e$ by $h^2 \|\bu\|_{H^2(\gamma)} \|\bw\|_{H^2(\gamma)}$. 

{\bf Step 4:  Bound $\sum_{e \in \calE} I_e$.}
We first compute using $|1-\bnu\cdot \bnu_{hi}|\le h^2$, \eqref{eq:defequiv},  $|1-\mu_e| \le h^2$ and a scaled trace inequality that
\begin{equation}
\begin{aligned}
    I_e & = \int_e ({\rm Def}_{\bar \Gamma_{h1}} \ipt \bu_1 \cdot \bn_{h1} -{\rm Def}_{\bar \Gamma_{h2}} \ipt \bu_2 \cdot \bn_2) (\bnu\cdot \bnu_{h1}-1) \bw^\ell 
    \\ & ~~~~~- \int_e ({\rm Def}_{\bar \Gamma_{h1}} \ipt \bu \cdot \bn_{h1} -{\rm Def}_{\bar \Gamma_{h2}} \ipt \bu_2 \cdot \bn_2) (\bnu \cdot \bnu_{h2}-1 )\bw^\ell 
    \\ & =  2\int_e ({\rm Def}_\gamma \bu \cdot \bn)^\ell (\bnu \cdot \bnu_{h1} -1) \bw^\ell - 2\int_e ({\rm Def}_\gamma \bu \cdot \bn)^\ell (\bnu \cdot \bnu_{h2} -1)\bw^\ell + \beta(e),
    \\ & = 2\int_{e^\gamma} ({\rm Def}_\gamma \bu \cdot \bn) (\bnu \cdot \bnu_{h1} -1) \bw - 2\int_{e^\gamma} ({\rm Def}_\gamma \bu \cdot \bn) (\bnu \cdot \bnu_{h2} -1)\bw + \beta(e).
    \end{aligned}
\end{equation}
    Thus integrating by parts (cf. Equation (3) and following of \cite{Fries18}), using $|1-\bnu\cdot \bnu_h| \lesssim h^2$, ${\bA} :[ {\bf a} \otimes {\bf b}] = {\bf a}^\intercal {\bA} {\bf b}$, $\bPi \bH \bPi = \bH$, and $|1-\mu_h|\lesssim h^2$ we find
    \begin{equation}
    \begin{aligned}
\sum_{e \in \calE} I_e &= 2 \sum_{K \in \calT_{h}^1} \int_{\partial K^\gamma}  ({\rm Def}_\gamma \bu \cdot \bn) (\bnu \cdot \bnu_h -1) \bw + \beta(\gamma) 
\\ & = 2 \Big ( \int_{\gamma} ({\rm Div}_{\gamma} {\rm Def}_\gamma \bu) \cdot \bw (\bnu \cdot \bnu_h-1) +\int_{\gamma} [{\rm Def}_\gamma \bu :\nabla_\gamma \bw ](\bnu\cdot \bnu_h -1) 
\\ & ~~~~~ + \int_\gamma {\rm Def}_\gamma \bu : \left [ \bPi \left ( \bw \otimes \nabla (\bnu\cdot \bnu_h) \right ) \bPi \right ]  \Big ) + \beta(\gamma)
\\ & = 2\int_\gamma {\rm Def}_\gamma \bu :  (\bw \otimes (\bH \bnu_h) ) + \beta(\gamma)
\\ & = 2 \int_\gamma    \bw^\intercal ({\rm Def}_\gamma \bu) \bH \bnu_h  + \beta(\gamma)
\\ & = 2\int_{\bar \Gamma_h}  [\bH^\ell ({\rm Def}_\gamma \bu)^\ell \bw^\ell] \cdot ( \bPi \bnu_h )+\beta(\gamma).  
\end{aligned}
    \end{equation}
We now recall \cite[Lemma 4.2]{HansboLarsonLarsson20} (cf. \cite{WZPP} for generalizations), which after slight modification yields
\begin{equation}
\label{eq:HLL}
    \left |\int_{\bar \Gamma_h} \bz \cdot \bPi \bnu_h \right | + \left | \int_{\bar \Gamma_h} \bz \cdot \bPi_h \bnu \right | \lesssim h^2 \|\bz\|_{W_1^1(\bar \Gamma_h)}, ~~~\bz \in [W_1^1(\bar \Gamma_h)]^3.
\end{equation}
Applying this result along with Cauchy-Schwarz and norm equivalence yields 
\begin{equation}
\left |2\int_{\bar \Gamma_h}  [\bH^\ell ({\rm Def}_\gamma \bu)^\ell \bw^\ell] \cdot ( \bPi \bnu_h ) \right | \lesssim |\beta(\gamma)|,
\end{equation}
which completes the proof that 
\begin{equation}
\sum_{e \in \calE} I_e \lesssim h^2 \|\bu\|_{H^2(\gamma)} \|\bw\|_{H^2(\gamma)}.
    \end{equation}

\noindent {\bf Step 5:  Bound for $\sum_{e \in \calE} II_e$.}  
Using \eqref{eq:vecinterp}, $|\bPi_{\bar \Gamma_{hi}} \bnu  |\lesssim h$, and $|\bw^\ell \cdot \bnu_{hi}|  = |\bw^\ell \cdot (\bnu_{hi}-\bnu)| \lesssim h|\bw^\ell|$, we find that
\begin{equation}
\int_e ({\rm Def}_{\bar\Gamma_{hi}} \ipt \bu_i \cdot \bn_{hi}) ( \bnu \otimes \bnu_{hi}) \bw^\ell= \int_{e} \left ( ({\rm Def}_{\bar \Gamma_{hi}} (I_h (\ipt \bu))_i \cdot \bn_{hi}) (\bw^\ell \cdot \bnu_{hi}) \right ) \cdot ( \bPi_{hi} \bnu) + \beta(e).
\end{equation}
Summing over edges and integrating by parts while recalling that ${\rm Def}_{\bar \Gamma_h} (I_h (\ipt \bu))_i$ is constant, we find that
\begin{equation}
\begin{aligned}
& \sum_{e \in \calE}  II_e  = \sum_{K \in \calT_h^1} \int_{\partial K} ({\rm Def}_{\bar \Gamma_h} I_h (\ipt \bu )\cdot \bn_h^1) (\bw^\ell \cdot \bnu_h) \bnu+ \beta(\gamma)
\\ & = \int_{\bar \Gamma_h} ({\rm Def}_{\bar \Gamma_h} I_h (\ipt \bu)): \nabla_{\bar \Gamma_h} [(\bw^\ell \cdot \bnu_h) \bnu] + \beta(\gamma)
\\ & = \int_{\bar \Gamma_h} \left [ ({\rm Def}_{\bar \Gamma_h} I_h (\ipt \bu)):(\bPi_h\bH \bPi_h) (\bw^\ell \cdot \bnu_h)  +({\rm Def}_{\bar \Gamma_h} I_h (\ipt \bu)) :(\bPi_h \bnu) \otimes ( \bPi_h \nabla (\bw^\ell)^\intercal \bnu_h ) \right ] 
\\ & ~~~~+ \beta(\gamma)
\\ & =: II_1+II_2 + \beta(\gamma).
    \end{aligned}
\end{equation}
Next noting that $\bPi \bH = \bH$, $|\bPi_h - \bPi| \lesssim h$ and $| \bw^\ell \cdot \bnu_h| = |\bw^\ell \cdot \bPi \bnu_h| \lesssim h|\bw^\ell|$ while applying \eqref{eq:vecinterp} and \eqref{eq:defequiv} yields
\begin{equation}
    II_1=  \int_{\bar \Gamma_h} ((( {\rm Def}_{\gamma} \bu)^\ell:\bH )\bw^\ell ) \cdot (\bPi \bnu_h) +\beta(\gamma)
\end{equation}
Applying \eqref{eq:HLL} and the Cauchy-Schwarz inequality, we obtain
\begin{equation}
|II_1| \lesssim \beta(\gamma).
\end{equation}
Employing ${\bA} :[ {\bf a} \otimes {\bf b}] = {\bf a}^\intercal {\bA} {\bf b}$, $\bPi_h {\rm Def}_{\bar \Gamma_h} I_h (\ipt \bu) \bPi_h = {\rm Def}_{\bar \Gamma_h} I_h (\ipt \bu)$, \eqref{eq:vecinterp}, \eqref{eq:defequiv}, $|\bnu-\bnu_h|+ |\bPi_h \bnu| \lesssim h$, and \eqref{eq:HLL} yields
\begin{equation}
\begin{aligned}
   | II_2 | & = \left | \int_{\bar \Gamma_h}  \left [ {\rm Def}_{\bar \Gamma_h} I_h (\ipt \bu) \nabla (\bw^\ell)^\intercal\bnu_h)\right ] \cdot \bPi_h \bnu \right |
    \\ & = \left | \int_{\bar \Gamma_h} \left [  ({\rm Def}_{\gamma} u)^\ell (\nabla \bw^\ell)^\intercal \bnu \right] \cdot (\bPi_h \bnu) + \beta(\gamma) \right | 
    \\ & \lesssim |\beta(\gamma)|. 
    \end{aligned}
\end{equation}

\noindent {\bf Step 6:  Bound for $\sum_{e \in \calE} III_e$.}  Using \eqref{eq:defequiv}, $|\bPi_{hi} \bnu||\bPi \bnu_{hj}| \lesssim h^2$ and $|\bn_{h1}-\bn|+|\bn_{h1}+ \bn_{h2}| \lesssim h$, we have
\begin{equation} 
\label{eq:IIIstart}
\begin{aligned}
III_e  & = \int_e \left [ ({\rm Def}_{\bar \Gamma_{h1}} \ipt \bu_1 \cdot \bn_{h1} ) (\bPi_{h1} \bnu) \otimes (\bPi \bnu_{h2}) + ({\rm Def}_{\bar \Gamma_{h2}} \ipt \bu_2 \cdot \bn_{h2}) (\bPi_{h2} \bnu) \otimes (\bPi \bnu_{h1})  \right ]\cdot \bw^\ell
\\ & = \int_e ({\rm Def}_\gamma \bu \cdot \bn)^\ell [(\bPi_{h1} \bnu) \otimes (\bPi \bnu_{h2}) - (\bPi_{h2} \bnu) \otimes (\bPi \bnu_{h1})] \cdot \bw^\ell + \beta(e).
\end{aligned}
\end{equation}
Using $|\bPi_{hi} \bnu+ \bPi \bnu_{hi}| \lesssim h^2$ from \cite{HansboLarsonLarsson20} and then rearranging terms, we have that
\begin{equation}
\begin{aligned}
(\bPi_{h1} \bnu) & \otimes (\bPi \bnu_{h2})  - (\bPi_{h2} \bnu) \otimes (\bPi \bnu_{h1})  \\ & = (\bPi \bnu_{h1}) \otimes (\bPi \bnu_{h2}) - (\bPi \bnu_{h2}) \otimes (\bPi \bnu_{h1}) +O(h^3)
\\ & = ( \bPi\bnu_{h1}-\bPi \bnu_{h2}) \otimes (\bPi \bnu_{h2}) -(\bPi \bnu_{h2}) \otimes (\bPi \bnu_{h1}-\bPi \bnu_{h2}) + O(h^3).    
\end{aligned}
\end{equation}
where $O(h^m)$ denotes a generic geometric error term bounded by $h^m$. Next note that because $\bnu_{hi} \cdot \bt_h=0$,
\begin{equation}
\begin{aligned}
\bPi(\bnu_{h1}-\bnu_{h2})& = (\bnu_{h1}-\bnu_{h2})-(\bnu\cdot \bnu_{h1}-\bnu\cdot \bnu_{h2})\bnu
\\ & = \bnu_{h1}-(\bnu_{h2} \cdot \bnu_{h1} \bnu_{h1}+ \bnu_{h2} \cdot \bn_{h1} \bn_{h1} + \bnu_{h2} \cdot \bt_h \bt_h)  -(\bnu\cdot \bnu_{h1}-\bnu\cdot \bnu_{h2})\bnu
\\ & = -\bnu_{h2} \cdot \bn_{h1} \bn_{h1} + (1-\bnu_{h2} \cdot \bnu_{h1}) \bnu_{h1} -(\bnu\cdot \bnu_{h1}-\bnu\cdot \bnu_{h2})\bnu
\\ & =  -\bnu_{h2} \cdot \bn_{h1} \bn_{h1} + O(h^2).
\end{aligned}
\end{equation}
Here we have also used that $|\bnu_{h1} \cdot \bnu_{h2}-1| + |\bnu \cdot \bnu_{hi}-1| \lesssim h^2$. 
Combining the previous two equations thus yields
\begin{equation}
\begin{aligned}
(\bPi_{h1} \bnu) \otimes (\bPi \bnu_{h2})  & - (\bPi_{h2} \bnu) \otimes (\bPi \bnu_{h1})
\\ & = - \bnu_{h2} \cdot \bn_{h1} \left [ \bn_{h1} \otimes (\bPi \bnu_{h1}) - (\bPi \bnu_{h1}) \otimes \bn_{h1} \right ] + O(h^3).
\end{aligned}
\end{equation}
Writing $\bPi \bnu_{h1} = \bnu_{h1}\cdot \bt \bt + \bnu_{h1} \cdot \bn \bn$, we have
\begin{equation}
\begin{aligned}
\bn_{h1} \otimes (\bPi \bnu_{h1}) & - (\bPi \bnu_{h1}) \otimes \bn_{h1}
\\ & = \bn_{h1} \otimes (\bnu_{h1}\cdot \bt \bt + \bnu_{h1} \cdot \bn \bn) - (\bnu_{h1}\cdot \bt \bt + \bnu_{h1} \cdot \bn \bn) \otimes \bn_{h1}
\\ & =  \bnu_{h1} \cdot \bt ( \bn_{h1} \otimes \bt -\bt \otimes \bn_{h1}) + \bnu_{h1} \cdot \bn ( \bn_{h1} \otimes \bn-\bn \otimes \bn_{h1}).
\end{aligned}
\end{equation}
Using $|\bnu_{h1} \cdot \bt| + |\bnu_{h2} \cdot \bn_{h1}| + |\bn-\bn_{h1}| \lesssim h$ then finally yields
\begin{equation}
\begin{aligned}
    (\bPi_{h1} \bnu) \otimes (\bPi \bnu_{h2}) &  - (\bPi_{h2} \bnu) \otimes (\bPi \bnu_{h1}) 
    \\ & = -( \bnu_{h2} \cdot \bn_{h1})( \bnu_{h1} \cdot \bt) (\bn_{h1} \otimes \bt-\bt \otimes \bn_{h1}) + O(h^3),
    \end{aligned}
\end{equation}
which after inserting into \eqref{eq:IIIstart} and employing a scaled trace inequality yields
\begin{equation}
III_e = -\bnu_{h2} \cdot \bn_{h1} \int_e ({\rm Def}_\gamma \bu \cdot \bn)^\ell (\bn_{h1} \otimes \bt-\bt \otimes \bn_{h1}) \bw^\ell (\bnu_{h1} \cdot \bt) + \beta(e).
\end{equation}

Next using \eqref{eq:tangent_conormal}, $\bPi \bnu_{hi} + \bPi_{hi} \bnu = O(h^2)$, and $\bnu=\nabla d$ we compute 
\begin{equation}
\begin{aligned}
    \bnu_{h1} \cdot \bt & = (\bPi  \bnu_{h1}) \cdot \bt = 
    (\bPi \bnu_{h1})\cdot (\bt-\bPi \bt_h) + (\bPi \bnu_{h1}) \cdot (\bPi \bt_h)
    \\ & = (\bPi \bnu_{h1})\cdot \bt_h + O(h^2)= (\bPi_{h1} \bnu) \cdot \bt_h + O(h^2) = \frac{\partial d}{\partial \bt_h} + O(h^2). 
    \end{aligned}
\end{equation}
Let $\bar \bw= \frac{1}{|K_1|} \int_{K_1} \bw^\ell$ so that $\|\bw^\ell -\bar \bw\|_{L_2(K_1)} \lesssim h \|\bw\|_{H^1(K_1^\gamma)}$.  Using $|\bt-\bt_h| \lesssim h$, \eqref{eq:defequiv}, and \eqref{eq:vecinterp}, we thus have that
\begin{equation}
\begin{aligned}
III_e & = -\bnu_{h2} \cdot \bn_{h1}  ({\rm Def}_{\bar \Gamma_h} I_h \ipt \bu_1 \cdot \bn_{h1}) (\bn_{h1} \otimes \bt_h - \bt_h \otimes \bn_{h1}) \bar \bw \int_e \frac{\partial d}{\partial \bt_h} + \beta(e)
\\ & = \bnu_{h2} \cdot \bn_{h1}  ({\rm Def}_{\bar \Gamma_h} I_h \ipt \bu_1 \cdot \bn_{h1}) (\bn_{h1} \otimes \bt_h - \bt_h \otimes \bn_{h1}) \bar \bw   (d(z_1)-d(z_0)) + \beta(e),
\end{aligned}
\end{equation}
where $z_1$ and $z_2$ are the vertices of $e$.  Recalling that $\bn_{h2} \cdot \bnu_{h1}=O(h)$ and using the assumption that $d(v_1)-d(v_0)=O(h^3)$ along with \eqref{eq:defequiv} and \eqref{eq:vecinterp}, a scaled trace inequality, and $\|\bw^\ell -\bar \bw\|_{L_2(K_1)} \lesssim h \|\bw\|_{H^1(K^\gamma)}$, we thus have that
\begin{equation}
\begin{aligned}
|III_e| & \lesssim h^4 |({\rm Def}_{\bar \Gamma_h} I_h \ipt \bu_1 \cdot \bn_{h1}) (\bn_{h1} \otimes \bt_h - \bt_h \otimes \bn_{h1}) \bar \bw| + |\beta_e|
\\ & \lesssim h^3 \int_e |({\rm Def}_{\gamma} \bu)^\ell| |\bar \bw| + |\beta(e)|
\\ & \lesssim |\beta(e)|.
\end{aligned}
\end{equation}
Summing over the edges completes the proof that
\begin{equation}
|III_e| \lesssim h^2 \|\bu \|_{H^2(\gamma)}\|\bw\|_{H^2(\gamma)}
\end{equation}
when $k=1$, and thus of Theorem \ref{geotheorem}.
\end{proof}

\section{Finite element method and error analysis}
\label{sec:errors}

In this section we define finite element approximations to the surface Stokes problem, prove stability results for the associated bilinear forms, and finally establish optimal error estimates in the energy norm.  

\subsection{Weak form and finite element method}
The weak form of the surface Stokes problem is:  Find $(\bu, p) \in {\bf HT}^1(\gamma) \times L_2^0(\gamma)$ such that
\begin{equation}
\label{eq:weakform}
\begin{aligned}
a_\gamma (\bu, \bw) -b_\gamma (\bw, p) & = ({\bf f}, \bw)_\gamma \hbox{ for all } \bw \in {\bf HT}^1(\gamma),
\\ b_\gamma (\bu, q) & = 0 \hbox{ for all } q \in L_2^0(\gamma).
\end{aligned}
\end{equation}
Recall that the form $a_\gamma$ includes the jump form $j_\gamma$.  However, $j_\gamma(\bu, \bw)=0$ when $\bu, \bw \in {\bf HT}^1(\gamma)$, so \eqref{eq:weakform} reduces to a standard weak form.  In addition, \eqref{eq:weakform} additionally holds for $\bw=\pt{\bw}_h$ with $\bw_h \in \bV_h$.  This follows from multiplying the strong form \eqref{eq:surfstokes} by $\pt \bw_h$ and integrating by parts elementwise while recalling that $[\bu]_\gamma=0$ for $\bu \in {\bf HT}^1(\gamma)$.  This relationship also uses that since $\pt \bw_h \in \pt \bV_h \subset H({\rm div}, \gamma)$ there holds on edges $e^\gamma$ that $[\pt \bw_h]=[\pt \bw_h]\cdot \bt_\gamma \bt_\gamma+ [\pt \bw_h] \cdot \bn_\gamma \bn_\gamma = [\pt \bw_h]_\gamma$.    

A standard elliptic  shift theorem also holds for this problem.  In particular, if ${\bf f} \in [H^m(\gamma)]^3$ and $\gamma$ is sufficiently smooth, then $(\bu, p) \in [H^{m+2}(\gamma)]^3 \times H^{m+1}(\gamma)$ and
\begin{equation}
\label{eq:shifttheorem}
\|\bu\|_{H^{m+2}(\gamma)} + \|p\|_{H^{m+1}(\gamma)} \lesssim \|{\bf f}\|_{H^m(\gamma)}.
\end{equation}
This theorem is contained in \cite[Theorem 1.2]{BNSPP} for a Stokes operator employing a Bochner Laplacian in the elliptic term rather than the surface diffusion operator employed here.  As is pointed in \cite{BNSPP}, in the presence of the surface incompressiblity condition the two operators differ only by a zero-th order term and thus \eqref{eq:shifttheorem} may be easily derived for the surface Stokes problem as well.  

The corresponding finite element method is:  Find $(\bu_h, p_h) \in \bV_h \times Q_h^0$ such that
\begin{equation}
\label{eq:FEM}
\begin{aligned}
a_{\Gamma_h^k} (\bu_h, \bw_h) -b_{\Gamma_h^k} (\bw_h, p_h) & = ({\bf f}_h, \bw_h)_{\Gamma_h^k} \hbox{ for all } \bw_h \in \bV_h,
\\ b_{\Gamma_h^k} (\bu_h, q_h) & = 0 \hbox{ for all } q_h \in Q_h.
\end{aligned}
\end{equation}

Here ${\bf f}_h$ is chosen so that the functional $\bw_h \mapsto ({\bf f}_h, \bw_h)_{\Gamma_h^k}$ approximates the functional $\bw_h \mapsto ({\bf f}, \pt{\bw}_h)_\gamma$.  Note that ${\bf f}_h$ need not necessarily be tangent to $\Gamma_h^k$.  Following the discussion in \cite[Remark 5.3]{DN26}, let 
\begin{equation}
{\bf F}_h = [{\bf I}-d{\bf H} ]^{-1} \left [ {\bf I} - \frac{\bnu_h \otimes \bnu}{\bnu_h \cdot \bnu} \right ] \bPi_h {\bf f}_h.
\end{equation}
Then if $\bw \cdot \bnu_h=0$,
\begin{equation}
\int_{\Gamma_h^k} {\bf f}_h \cdot \bw = \int_\gamma {\bf F}_h^\ell \cdot \pt{\bw}.
\end{equation}
Choosing either ${\bf f}_h={\bf f}^\ell$ or ${\bf f}_h= \ipt{\bf f}$ then yields for $\bw \in L_2(\Gamma_h^k)$ with $\bw \cdot \bnu_h=0$ that
\begin{equation}
\label{eq:fapprox}
\int_\gamma {\bf f} \cdot \pt \bw -\int_{\Gamma_h^k} {\bf f}_h \cdot \bw = \int_\gamma ({\bf f}-{\bf F}_h^\ell) \cdot \pt{\bw} \lesssim h^{k+1} \|{\bf f}\|_{L_2(\Gamma)}\|\pt{\bw}\|_{L_2(\gamma)} \hbox{ and } \|{\bf f}_h\|_{L_2(\Gamma_h^k)} \lesssim \|{\bf f}\|_{L_2(\gamma)}.
\end{equation}
We assume below that \eqref{eq:fapprox} is satisfied.

\subsection{Discrete Inf-Sup Condition}

\begin{theorem}[Discrete inf-sup condition]
Assume that $h$ is sufficiently small.  Then there exists a constant $\alpha>0$ independent of $h$ such that
\begin{equation}
\label{eq:discinfsup}
\sup_{\bw_h \in \bV_h \setminus {\bf 0}} \frac{b_{\Gamma_h^k}(\bw_h, q_h)}{\|\bw_h\|_{DG, \Gamma_h^k}} \ge \alpha \|q_h\|_{L_2(\Gamma_h^k)}~~~ \forall q_h \in Q_h^0.
\end{equation}
\end{theorem}
\begin{proof} \hspace{.1cm}
Assume that $h$ is small enough to guarantee that $\|1-(\mu_h^k)^{-1}\|_{L_\infty(\Gamma_h^k)} \le \frac{1}{2}$.  For $q_h \in Q_h^0$, we have $\frac{1}{|\gamma|} \int_\gamma q_h^\ell  = \frac{1}{|\gamma|} \int_\gamma (1-(\mu_h^k)^{-1}) q_h^\ell \le \frac{\|1-(\mu_h^k)^{-1}\|_{L_\infty(\Gamma_h^k)}}{|\gamma|^{1/2}} \|q_h^\ell\|_{L_2(\gamma)}$ and so $\|\frac{1}{|\gamma|} \int_\gamma q_h^\ell\|_{L_2(\gamma)}  \le \|1-(\mu_h^k)^{-1} \|_{L_\infty(\Gamma_h^k)} \|q_h^\ell\|_{L_2(\gamma)} \le \frac{1}{2} \|q_h^\ell \|_{L_2(\gamma)}.$  Employing the triangle inequality and reabsorbing the resulting term thus yields
\begin{equation}
\label{eq:1}
\|q_h^\ell\|_{L_2(\gamma)} \lesssim \left \|q_h^\ell-    \frac{1}{|\gamma|} \int_\gamma q_h^\ell \right \|_{L_2(\gamma)}.          
\end{equation}                                  
            
Using $b_\gamma(\bw, C)=0$ for any constant $C$,  the continuous inf-sup condition from \cite{JankuhnEtal18}, \eqref{eq:interp_stability}, \eqref{eq:piola_property}, \eqref{eq:weakcdp}, and \eqref{eq:dgequiv}, we next compute for $q_h \in Q_h \cap L_2^0(\gamma_h^k)$ that
\begin{equation}
\begin{aligned} 
\left \|q_h^\ell-    \frac{1}{|\gamma|} \int_\gamma q_h^\ell \right  \|_{L_2(\gamma)} & \lesssim \sup_{\bw \in {\bf HT}^1(\gamma) \setminus 0} \frac{b_\gamma(\bw, q_h^\ell)}{\|\bw\|_{H^1(\gamma)} }
 \lesssim \sup_{\bw \in {\bf HT}^1(\gamma) \setminus 0} \frac{b_\gamma(\bw, q_h^\ell)}{\|\pt {I_h \ipt{\bw}}\|_{DG,\gamma}}
\\ & = \sup_{\bw \in {\bf HT}^1(\gamma) \setminus 0} \frac{b_{\Gamma_h^k} (I_h \ipt{\bw}, q_h)}{\|\pt {I_h \ipt{\bw}}\|_{DG,\gamma}}
 \lesssim \sup_{\bw_h \in \bV_h \setminus 0} \frac{b_{\Gamma_h^k}(\bw_h, q_h)}{\|\bw_h\|_{DG,\Gamma_h^k}}.
\end{aligned}
\end{equation}
Combining this result with \eqref{eq:1} and norm equivalence yields the desired estimate.
\end{proof}

\subsection{Coercivity and continuity}
We first establish coercivity of the bilinear form $a_{\Gamma_h^k}$ with respect to the norm $\|\cdot \|_{DG, \Gamma_h^k}$.  
\begin{lemma}  Assuming the penalty parameter $\rho>0$ is sufficiently large (independent of $h$), there holds for $\bu_h \in \bV_h$
\begin{equation} \label{eq:coercivity}
a_{\Gamma_h^k} (\bu_h, \bu_h) \gtrsim \|\bu_h\|_{DG, \Gamma_h^k}^2.
\end{equation}
\end{lemma}
\begin{proof} \hspace{.1cm}
Using the scaled trace inequality \eqref{eq:scaledtrace}, inverse estimates, the discrete Korn inequality \eqref{eq:korndiscsurf} with $\rho$ momentarily taken to be 1, and norm equivalence, we find that there is a constant $C_1$ independent of $\rho$ and $h$ such that 
\begin{equation}
\begin{aligned}
  h \|\{ {\rm Def}_{\Gamma_h^k} \bu_h\}_{\Gamma_h^k} \|_{L_2(\Sigma)}^2 & \le C (\|{\rm Def}_{\Gamma_h^k} \bu_h \|_{L_2(\Gamma_h^k)}^2 + h^2 |{\rm Def}_{\Gamma_h^k} \bu_h|_{H^1(\Gamma_h^k)}^2)  
  \le C \|\bu_h\|_{H_h^1(\Gamma_h^k)}^2 \\ & \le C_1 ( \|{\rm Def}_{\Gamma_h^k}\bu_h\|_{L_2(\Gamma_h^k)}^2 + \|\bu_h\|_{L_2(\Gamma_h^k)}^2 + h^{-1} \|[\bu_h]_{\Gamma_h^k}\|_{L_2(\Sigma)}^2).
\end{aligned}
\end{equation}
Using H\"older's inequality and Young's inequality, we thus compute that for any $\epsilon>0$ that
\begin{equation}
\begin{aligned}
\int_{\varSigma} & \{\mathrm{Def}_{\Gamma_h^k}\bu_h \}_{\Gamma_k^h}  
 \cdot [\bu_h]_{\Gamma_h^k}  \le h^{1/2} \|\{{\rm Def}_{\Gamma_h^k} \bu_h\}_{\Gamma_h^k} \|_{L_2(\Sigma)} h^{-1/2} \|[\bu_h]_{\Gamma_h^k}\|_{L_2(\varSigma)}
  \\& \le \frac{\epsilon}{2} \left (\|\mathrm{Def}_{\Gamma_h^k}\bu_h\|_{L_2(\Gamma_h^k)}^2 + \|\bu_h\|_{L_2(\Gamma_h^k)}^2 + h^{-1} \|[\bu_h]_{\Gamma_h^k}\|_{L_2(\Sigma)}^2 \right ) +\frac{C_1}{2 \epsilon h} \|[\bu_h]_{\Gamma_h^k}\|_{L_2(\varSigma)}^2.
   \end{aligned}
 \end{equation}
Recalling the definition \eqref{eq:jdef}, we then have that 
\begin{equation}
j_{\Gamma_h^k} (\bu_h, \bu_h) \ge h^{-1} ({2 \rho}-\frac{2C_1}{\epsilon}-2\epsilon) \|[\bu_h]_{\Gamma_h^k} \|_{L_2(\varSigma)}^2 - 2\epsilon (\|{\rm Def}_{\Gamma_h^k} \bu_h \|_{L_2(\Gamma_h^k)}^2+\|\bu_h\|_{L_2(\Gamma_h^k)}^2),
\end{equation}
and thus
\begin{equation}
\begin{aligned}
a_{\Gamma_h^k}(\bu_h, \bu_h) & \ge (2-2\epsilon)\| {\rm Def}_{\Gamma_h^k} \bu_h \|_{L_2(\Gamma_h^k)}^2 +  (1-2\epsilon) \|\bu_h\|_{L_2(\Gamma_h^k)}^2 
\\ & ~~~~+ h^{-1} (\frac{2\rho}{h}-\frac{2C_1}{\epsilon h}-2\epsilon) \|[\bu_h]_{\Gamma_h^k} \|_{L_2(\varSigma)}^2 .
\end{aligned}
\end{equation}
Taking $\epsilon$ sufficiently small to ensure that $2-2\epsilon> 1-2\epsilon \ge c_0$ (for example, $\epsilon=\frac{1}{8}$) and then $\rho$ sufficiently large to ensure that $(\frac{2\rho }{h}-\frac{2C_1}{\epsilon h}-2\epsilon) \ge c_1 \frac{\rho}{h}$ with $c_0, c_1>0$ yields the desired result.
\end{proof}

We also state a continuity result.
\begin{lemma}  There holds for $\bu, \bw \in H_h^2(\gamma) \cap H({\rm div}; \gamma)$
\begin{equation}
\label{eq:continuity}
a_\gamma(\bu, \bw)  \lesssim \vvvert \bu\vvvert_\gamma \vvvert \bw\vvvert_\gamma.
\end{equation}
If in addition $\ipt \bw \in \bV_h$, then
\begin{equation}
\label{eq:disccontinuity}
a_\gamma(\bu, \bw) \lesssim \vvvert \bu \vvvert_\gamma \|\bw\|_{DG, \gamma}.
\end{equation}
\end{lemma}
\eqref{eq:continuity} can essentially be found in \cite[Equation (4.28)]{BDL20}. The proof is very similar here, and we do not repeat it except to note that the additional terms involving $H_h^2$ norms arise from applying a scaled trace inequality to the average terms in the jump form $j_\gamma$.  \eqref{eq:disccontinuity} follows from \eqref{eq:continuity} by applying \eqref{eq:triplebarequiv} (cf. \cite[(4.22)]{BDL20}).

\subsection{Error analysis}
We first define divergence free subspaces $\bX_h =\{\bw_h \in \bV_h: {\rm div}_{\Gamma_h^k} \bw_h=0\}$ and $\bX=\{\bw \in {\bf HT}^1(\gamma): {\rm div}_{\gamma} \bw = 0 \}$.   There then hold $\bu \in \bX$ and $\bu_h \in \bX_h$.  In addition, recall that $a_\gamma$ is a consistent form in the sense that $a_{\gamma}(\bu, \pt{\bw}_h) -b_\gamma(\pt{\bw}_h, p) = ({\bf f}, \pt{\bw}_h)$ for all $\bw_h \in \bV_h$.  Thus for $\bw \in \pt \bX_h\cup \bX$ and $\bw_h \in \bX_h$, 
\begin{equation} \label{eq:divfree}
a_{\gamma}( \bu, \pt{\bw}) = ( {\bf f}, \pt{\bw}), ~~~~a_{\Gamma_h^k}(\bu_h, \bw_h) = ({\bf f}_h, \bw_h). 
\end{equation}
Before proceeding further, note that coercivity, continuity, Korn inequalities, the continuous and discrete inf-sup conditions, and \eqref{eq:fapprox} yield
\begin{equation}
\label{stability}
\|\bu\|_{H^1(\gamma)}+ \|\pt \bu_h \|_{DG, \gamma} + \|p\|_{L_2(\gamma)} + \|p_h\|_{L_2(\Gamma_h^k)}  \lesssim \|{\bf f} \|_{L_2(\gamma)}.
\end{equation}

We now prove the following error estimate for $\bu-\pt{\bu}_h$.
\begin{lemma} If $h$ is sufficiently small and $\rho$ is sufficiently large independent of $h$ and $\bu$, there holds
\begin{equation}
\label{eq:uerror}
\|\bu -\pt{\bu}_h\|_{DG, \gamma}  \lesssim \inf_{{\bf \chi}_h \in \bV_h} \vvvert \bu-\pt {\bf \chi}_h\vvvert_\gamma + h^k \|{\bf f}\|_{L_2(\gamma)}.
\end{equation}
\end{lemma}
\begin{proof} \hspace{.1cm}
Let $\bw_h = I_h \ipt \bu$.  By the divergence relationship \eqref{eq:pioladiv} and the commuting diagram property \eqref{eq:cdp}, there holds  $\bw_h  \in \bX_h$.  Letting ${\bf e}_h = \bu_h-\bw_h \in \bX_h$, employing \eqref{eq:coercivity}, \eqref{eq:divfree},  \eqref{eq:disccontinuity}, \eqref{eq:fapprox},  the geometric error estimate \eqref{G_bound}, and \eqref{eq:interp_stability} while recalling $\bw_h =I_h \ipt \bu$ then yields
\begin{equation}
\begin{aligned}
\|{\bf e}_h \|_{DG,\Gamma_h^k}^2 & \lesssim a_{\Gamma_h^k}({\bf e}_h, {\bf e}_h) = ({\bf f}_h, {\bf e}_h) - a_{\Gamma_h^k} (\bw_h, {\bf e}_h)
\\ & = a_\gamma(\bu, \pt{\bf e}_h)-a_\gamma(\pt{\bw}_h, \pt{\bf e}_h)+ G({\bw}_h, {\bf e}_h)+({\bf F}_h^\ell, \pt {\bf e}_h)_{\Gamma_h^k} -({\bf f}, \pt{\bf e}_h)_\gamma
\\ & \lesssim  \vvvert \bu-\pt{\bw}_h\vvvert_\gamma \|\pt{\bf e}_h\|_{DG, \gamma}
+  h^k \|\pt{\bw}_h\|_{DG,\gamma} \|\pt{\bf e}_h\|_{DG, \gamma} + h^{k+1} \|{\bf f}\|_{L_2(\gamma)} \|\pt{\bf e}_h\|_{L_2(\gamma)}
\\ & \lesssim \left [\vvvert \bu-\pt{\bw}_h\vvvert_\gamma + h^k \|\bu\|_{H^1(\gamma)}+ h^{k+1} \|{\bf f}\|_{L_2(\gamma)}  \right] \|\pt {\bf e}_h\|_{DG,\gamma}.
\end{aligned}
\end{equation} 
Applying equivalence of $DG$ norms on $\gamma$ and $\Gamma_h^k$ to the factor $\|\pt{\bf e}_h\|_{DG, \gamma}$, dividing through by the resulting term, noting that $\|{\bf u}\|_{H^1(\Gamma)} \lesssim \|{\bf f}\|_{L_2(\gamma)}$ by \eqref{stability}, and finally writing $\bu-\pt{\bw}_h = (\bu - \pt \chi_h)-I_h (\bu-\pt \chi_h)$ and  applying the triangle equality along with \eqref{eq:interp_stability2} to the first two terms above yields the first desired result. 
\end{proof}

We now provide estimates for the pressure error $p^\ell -p_h$.
\begin{lemma} If $h$ is sufficiently small and $\rho$ is sufficiently large independent of $h$ and $\bu$, then
\begin{equation}
\label{eq:perror}
\|p^\ell-p_h\|_{L_2(\gamma)} \lesssim \inf_{q_h \in Q_h^0} \|p-q_h^\ell\|_{L_2(\gamma)}+ \inf_{{\bf \chi}_h \in \bV_h} \vvvert \bu-\pt {\bf \chi}_h\vvvert_\gamma + h^k \|{\bf f}\|_{L_2(\gamma)}.
\end{equation}
\end{lemma}
\begin{proof} \hspace{.1cm}
Given $q_h \in Q_h^0$, we may employ \eqref{eq:discinfsup} to find $\bw_h \in \bV_h$ such that $\|\bw_h\|_{DG, \Gamma_h^k} =1$ and 
\begin{equation}
\|q_h-p_h \|_{L_2(\Gamma_h^k)} \lesssim b_{\Gamma_h^k} (\bw_h, q_h-p_h).
\end{equation}
Using \eqref{eq:piola_property}, \eqref{eq:FEM}, and \eqref{eq:weakform}, we next compute that
\begin{equation}
\begin{aligned}
b_{\Gamma_h^k}& (\bw_h, q_h-p_h)  = b_\gamma(\pt{\bw}_h, q_h^\ell)-a_{\Gamma_h^k} (\bu_h, \bw_h) +({\bf f}_h, \bw_h)_{\Gamma_h^k} 
\\ & =  b_\gamma(\pt{\bw}_h, q_h^\ell -p) + b_\gamma(\pt{\bw}_h, p) -a_\gamma(\bu, \pt{\bw}_h)+a_\gamma(\bu-\pt{\bu}_h, \pt{\bw}_h) 
\\ & ~~~~+G(\bu_h, \bw_h) +({\bf f}_h, \bw_h)_{\Gamma_h^k} 
\\ & =b_\gamma(\pt{\bw}_h, q_h^\ell -p)+a_\gamma(\bu-\pt{\bu}_h, \pt{\bw}_h)+G(\bu_h, \bw_h) + ({\bf f}_h, \bw_h)_{\Gamma_h^k} -({\bf f}, \pt{\bw}_h).
\end{aligned}
\end{equation}
Employing continuity of $a_\gamma$ and $b_\gamma$, $\|\bw_h\|_{DG, \Gamma_h^k}=1$, the triangle inequality, \eqref{G_bound}, and equivalence of relevant norms  on $\gamma$ and $\Gamma_h^k$ thus yields for any $q_h \in Q_h^0$ 
\begin{equation}
\label{eq:pressure1}
\begin{aligned}
\|p^\ell-p_h\|_{L_2(\Gamma_h^k)} & \lesssim \|p^\ell-q_h\|_{L_2(\Gamma_h^k)} + \vvvert \bu-\pt{\bu}_h\vvvert_\gamma + h^k\|\bu_h\|_{DG, \Gamma_h^k} 
\\ &  ~~~~+ ({\bf f}_h, \bw_h)_{\Gamma_h^k} -({\bf f}, \pt{\bw}_h)_\gamma.
\end{aligned}
\end{equation}
Using the triangle inequality, \eqref{eq:triplebarequiv}, the triangle inequality again, and once again \eqref{eq:triplebarequiv} yields 
\begin{equation}
\label{eq:tripletodg}
\begin{aligned}
    \vvvert \bu-\pt \bu_h\vvvert_\gamma & \lesssim  \vvvert \bu-\pt {\bf \chi}_h \vvvert_\gamma + \vvvert \pt {\bf \chi}_h-\pt \bu_h\vvvert_\gamma 
    \lesssim \vvvert \bu-\pt {\bf \chi}_h \vvvert_\gamma+ \| \pt {\bf \chi}_h-\pt \bu_h\|_{DG,\gamma}
 \\ &  \lesssim     
    \|\bu-\pt \bu_h \|_{DG, \gamma} + \vvvert \bu-\pt {\bf \chi}_h \vvvert_\gamma.
    \end{aligned}
\end{equation}
In addition, by \eqref{eq:fapprox} and \eqref{stability} we have
\begin{equation}
\label{eq:pressure3}
h^k\|\bu_h\|_{DG, \Gamma_h^k} + ({\bf f}_h, \bw_h)_{\Gamma_h^k} -({\bf f}, \pt{\bw}_h) \lesssim h^k \|{\bf f}\|_{L_2(\gamma)}.
\end{equation}
Collecting \eqref{eq:pressure1}, \eqref{eq:tripletodg}, and \eqref{eq:pressure3} and then inserting \eqref{eq:uerror} while applying norm equivalence as needed yields the desired result.  
\end{proof}

\begin{theorem} Assume that $h$ is sufficiently small and that $\rho$ is sufficiently large independent of $h$ and $\bu$.  Assume also that $\bu \in [H^{r+1} (\gamma)]^3$ and $p \in H^r(\gamma)$.  Then
\begin{equation}
\label{eq:energyerror}
\|\bu-\pt \bu_h\|_{DG, \gamma} + \|p-p_h^\ell\|_{L_2(\gamma)} \lesssim h^r + h^k.
\end{equation}
If $k \ge 2$, or if $k \ge 1$ and $|d(z_0)-d(z_1)| \lesssim h^3$ for all vertex pairs $z_0, z_1$ sharing an edge in $\calE$, then
\begin{equation}
\label{eq:optl2}
\|\bu-\pt \bu_h\|_{L_2(\gamma)} \lesssim h^{r+1}\|\bu\|_{H^{r+1}(\gamma)} + h^{k+1}  \|{\bf f}\|_{L_2(\gamma)}.
\end{equation}
\end{theorem}
\begin{proof} \hspace{.1cm}
All terms in the upper bounds in \eqref{eq:uerror} and \eqref{eq:perror} are bounded directly by the right hand side of \eqref{eq:energyerror} upon applying the interpolation estimates \eqref{eq:vecinterp} and \eqref{eq:dginterp} with the exception of $\inf_{q_h \in Q_h^0} \|p-q_h^\ell\|_{L_2(\gamma)}$.  For the latter term we must account for the fact that $q_h \in Q_h^0$ does not imply $q_h^\ell \in L_2^0(\gamma)$.  Using norm equivalence, we compute
\begin{equation}
\begin{aligned}   
\inf_{q_h \in Q_h^0} \|p-q_h^\ell\|_{L_2(\gamma)} & \lesssim \inf_{q_h \in Q_h^0}   \|p^\ell-q_h\|_{L_2(\Gamma_h^k)} 
\\ & \le \inf_{q_h \in Q_h^0}  \left \|p^\ell-|\Gamma_h^k|^{-1} \int_{\Gamma_h^k} p^\ell -q_h \right \|_{L_2(\Gamma_h^k)} +  |\Gamma_h^k  |^{-1/2} \left |\int_{\Gamma_h^k} p^\ell \right |.
\end{aligned}
\end{equation}
Using $|1-\mu_h| \lesssim h^{k+1}$, $0=\int_\gamma p = \int _{\Gamma_h^k} \mu_h p^\ell $, and \eqref{stability} yields
\begin{equation}
\left |\int_{\Gamma_h^k} p^\ell \right |  =\left | \int_{\Gamma_h^k} (1-\mu_h) p^\ell \right |  \lesssim h^{k+1} \|p\|_{L_2(\Gamma)} \lesssim h^{k+1} \|f\|_{L_2(\gamma)}.
\end{equation}
Noting that $\pi_h^k L_2^0(\Gamma_h^k)=Q_h^0$ and employing \eqref{eq:l2interp} yields after choosing $q_h = \pi_h^k (p^\ell-|\Gamma_h^k|^{-1} \int_{\Gamma_h^k} p^\ell)$ that
\begin{equation}
\inf_{q_h \in Q_h^0}  \left \|p^\ell-|\Gamma_h^k|^{-1} \int_{\Gamma_h^k} p^\ell -q_h \right \|_{L_2(\Gamma_h^k)} \lesssim h^r \|p\|_{H^r(\gamma)}.
\end{equation}
Combining these inequalities yields \eqref{eq:energyerror}.

We next prove \eqref{eq:optl2}.  Let $\bz \in \bX$ satisfy
\begin{equation}
a_\gamma(\bw, \bz) = (\bw, \bu-\pt \bu_h)_\gamma, ~~\bw \in \bX.  
\end{equation}
By \eqref{eq:shifttheorem} and \eqref{eq:divfree}, there then also hold
\begin{equation}
\label{eq:dualdef}
a_\gamma(\bw, \bz) =(\bw, \bu-\pt \bu_h)_\gamma, ~~~\bw \in \bX \cup \pt \bX_h, \hspace{1cm} \| \bz \|_{H^2(\gamma)} \lesssim \|\bu -\pt \bu_h\|_{L_2(\gamma)}.
\end{equation}

Let $\bz_h= I_h \ipt \bz \in \bX_h$.  Then
\begin{equation}
\label{eq:L2_1}
\begin{aligned}
\|\bu& -\pt \bu_h\|_{L_2(\gamma)}^2  = a_\gamma (\bu-\pt \bu_h, \bz) = a_\gamma (\bu-\pt \bu_h, \bz-\pt \bz_h) + a_\gamma (\bu-\pt \bu_h, \pt \bz_h)
\\ & =a_\gamma (\bu-\pt \bu_h, \bz-\pt \bz_h) + a_\gamma (\bu , \pt \bz_h)-a_\gamma(\pt \bu_h, \pt \bz_h)
\\ & = a_\gamma (\bu-\pt \bu_h, \bz-\pt \bz_h) +({\bf f}, \pt \bz_h)_\gamma  -G(\pt \bu_h, \pt \bz_h) -a_{\Gamma_h^k} (\bu_h, \bz_h)
\\ & = a_\gamma (\bu-\pt \bu_h, \bz-\pt \bz_h) + ({\bf f}, \pt \bz_h)_\gamma -G(\pt \bu_h-\bu, \pt \bz_h )- G(\bu, \pt \bz_h - \bz) -G(\bu, \bz) -({\bf f}_h, \bz_h)
\\ & = a_\gamma (\bu-\pt \bu_h, \bz-\pt \bz_h) + ({\bf f}-{\bf F}_h^\ell,  \pt \bz_h)_\gamma
 -G(\pt \bu_h-\bu, \pt \bz_h )- G(\bu, \pt \bz_h - \bz) -G(\bu, \bz) .
 \\ & =: I+II+III+IV+V.  
\end{aligned}
\end{equation}
By \eqref{eq:continuity}, \eqref{eq:tripletodg} with ${\bf \chi}_h = I_h \ipt \bu$, \eqref{eq:energyerror}, \eqref{eq:vecinterp}, \eqref{eq:dginterp}, and \eqref{eq:dualdef} there holds
\begin{equation}
\begin{aligned}
I  & \lesssim \vvvert \bu-\pt \bu_h \vvvert_{\gamma}  \vvvert \bz-\pt{I_h \ipt \bz} \vvvert_{\gamma} 
\\ & \lesssim  ( h^r \|\bu\|_{H^{r+1} (\gamma)} + h^k \|{\bf f}\|_{L_2(\gamma)}) h \|\bz \|_{H^2(\gamma)}
\\ & \lesssim ( h^{r+1} \|\bu\|_{H^{r+1}(\gamma)} + h^{k+1}  \|{\bf f}\|_{L_2(\gamma)}) \|\bu-\pt \bu_h\|_{L_2(\gamma)}.
\end{aligned}
\end{equation}
Next, by \eqref{eq:fapprox}, \eqref{eq:interp_stability}, and \eqref{eq:dualdef}  we have
\begin{equation}
II \lesssim h^{k+1} \|{\bf f}\|_{L_2(\gamma)} \| \pt \bz_h\|_{L_2(\gamma)} \lesssim h^{k+1} \|{\bf f}\|_{L_2(\gamma)} \|\bu-\pt \bu_h\|_{L_2(\gamma)}. 
\end{equation}
Using \eqref{eq:G_bound_mod}, \eqref{eq:tripletodg}, \eqref{eq:uerror},  \eqref{eq:vecinterp} and \eqref{eq:dginterp} with $m=2$, and \eqref{eq:dualdef} yields
\begin{equation}
\begin{aligned}
III  & \lesssim h^k \vvvert \pt\bu_h-\bu\vvvert_\gamma \vvvert \pt \bz_h \vvvert_{\gamma}  
\\ & \lesssim h^k (  \|\bu - \pt \bu_h\|_{DG,\gamma} + \vvvert \bu-\pt{ I_h \ipt \bu}\vvvert_\gamma) ( \vvvert \pt \bz_h-\bz \vvvert_\gamma + \vvvert \bz \vvvert_\gamma)
\\ & \lesssim h^k (h \|\bu \|_{H^2(\gamma)} + h^k \|{\bf f}\|_{L_2(\gamma)})  h \|\bz\|_{H^2(\gamma)}
\\ & \lesssim h^{k+1} \|{\bf f} \|_{L_2(\gamma)}  \|\bz \|_{H^2(\gamma)}  \lesssim h^{k+1} \|{\bf f}\|_{L_2(\gamma)} \|\bu -\pt \bu_h\|_{L_2(\gamma)}.
\end{aligned}
\end{equation}

Similarly,
\begin{equation}
IV \lesssim h^k \vvvert \bu \vvvert_\gamma \vvvert \pt \bz_h - \bz \vvvert_\gamma \lesssim h^{k+1} \|{\bf f}\|_{L_2(\gamma)} \|\bu-\pt \bu_h\|_{L_2(\gamma)}.
\end{equation}
Finally, \eqref{eq:GL2est}, \eqref{stability}, and \eqref{eq:dualdef} yield
\begin{equation}
V \lesssim h^{k+1}  \|\bu\|_{H^2(\gamma)} \|\bz\|_{H^2(\gamma)} \lesssim h^{k+1} \|{\bf f}\|_{L_2(\gamma)} \|\bu -\pt \bu_h\|_{L_2(\gamma)}.
\end{equation}
Collecting the previous inequalities into \eqref{eq:L2_1} and dividing through by $\|\bu-\pt \bu_h\|_{L_2(\gamma)}$ completes the proof.  
\end{proof}

\begin{remark}
 The case $k=1$ has presented special challenges in analyzing geometric errors for $L_2$ estimates for vector-Laplace and surface Stokes methods in other contexts as well.  In \cite{DN24} a tangentially conforming method based on the MINI element was analyzed on linear surfaces ($k=1$).  The construction is very similar to the Taylor-Hood method for which optimal-order $L_2$ estimates as in \eqref{eq:optl2} were proved in \cite{DN26} for $r=k \ge 2$.  Numerical experiments indicate that optimal bounds hold for the MINI element with $r=k=1$ as well.  However, the proof method used in \cite{DN26} for the Taylor-Hood element does not apply to the MINI element, and accordingly proof of optimal $L_2$ estimates remains open for the MINI method with $k=1$.  Another well-established method for surface Stokes and vector Laplace problems employs non-tangential velocity spaces composed of tri-$H^1(\Gamma_h^k)$ conforming (componentwise Lagrange) elements.  Tangentiality is enforced weakly by penalizing the normal component of the solution on $\Gamma_h^k$.  In order to achieve optimal convergence in both the $L_2$ and energy norms, it is necessary to employ a higher-order approximation to the continuous normal $\bnu$ in the tangential penalty term than holds for the natural choice of the normal $\bnu_h$ to $\Gamma_h^k$ \cite{HansboLarsonLarsson20}; this restriction is confirmed by numerical experiments.  In \cite{HP23, HP25} it was confirmed that optimal estimates in the energy and $L_2$ norms are achieved by employing $\bnu_h$ in the penalty formulation if only the tangential portion of the error is considered, except in the case of the $L_2$ error when $k=1$.
\end{remark}

\section{Numerical examples}
\label{sec:numerics}

In this section we document numerical experiments which confirm that the orders of convergence proved in \eqref{eq:energyerror} are sharp.  We also present numerical experiments which confirm the $L_2$ velocity error estimate \eqref{eq:optl2} and indicate that the extra condition $|d(z_1)-d(z_2)| \lesssim h^3$ required by our proofs when $k=1$ may not be needed.

\subsection{Numerical results}
We took $\gamma$ to be an ellipsoid with principal axes $(1.1, 1.2, 1.3)$, i.e., the zero level set of $\frac{x^2}{1.1^2}+\frac{y^2}{1.2^2}+\frac{z^2}{1.3^2}$.  The exact solutions for the velocity and the pressure are
\begin{align*}
\mathbf{u} &= \bPi [-z^2, x, y]^\top \\
p &= xy^3+z 
\end{align*}
with ${\bf f}$ computed accordingly.  Notice that $\mathbf{u}$ is not a polynomial since we use the projection operator.  In addition, ${\rm div}_\gamma \bu \neq 0$, so the Stokes system must be solved with nonzero incompressibility constraint. We employed a MATLAB code built on top of the iFEM library \cite{Ch09PP}.  The code allowed for arbitrary combinations of $k$ and $r$ for $1 \le r,k \le 4$.

In Figure \ref{fig_n1} we document convergence rates for isoparametric surface approximations $k=r$ for $r=1,2,3,4$. In the left plot we see the expected energy convergence rate $h^r+h^k$, while in the right plot the convergence rate $h^{r+1}+ h^{k+1}$ is observed for the velocity $L_2$ error as discussed above.  
\begin{figure}[h]
\begin{center}
\includegraphics[alt={Plot of numerical experiments showing optimal energy error decrease},scale=.26]{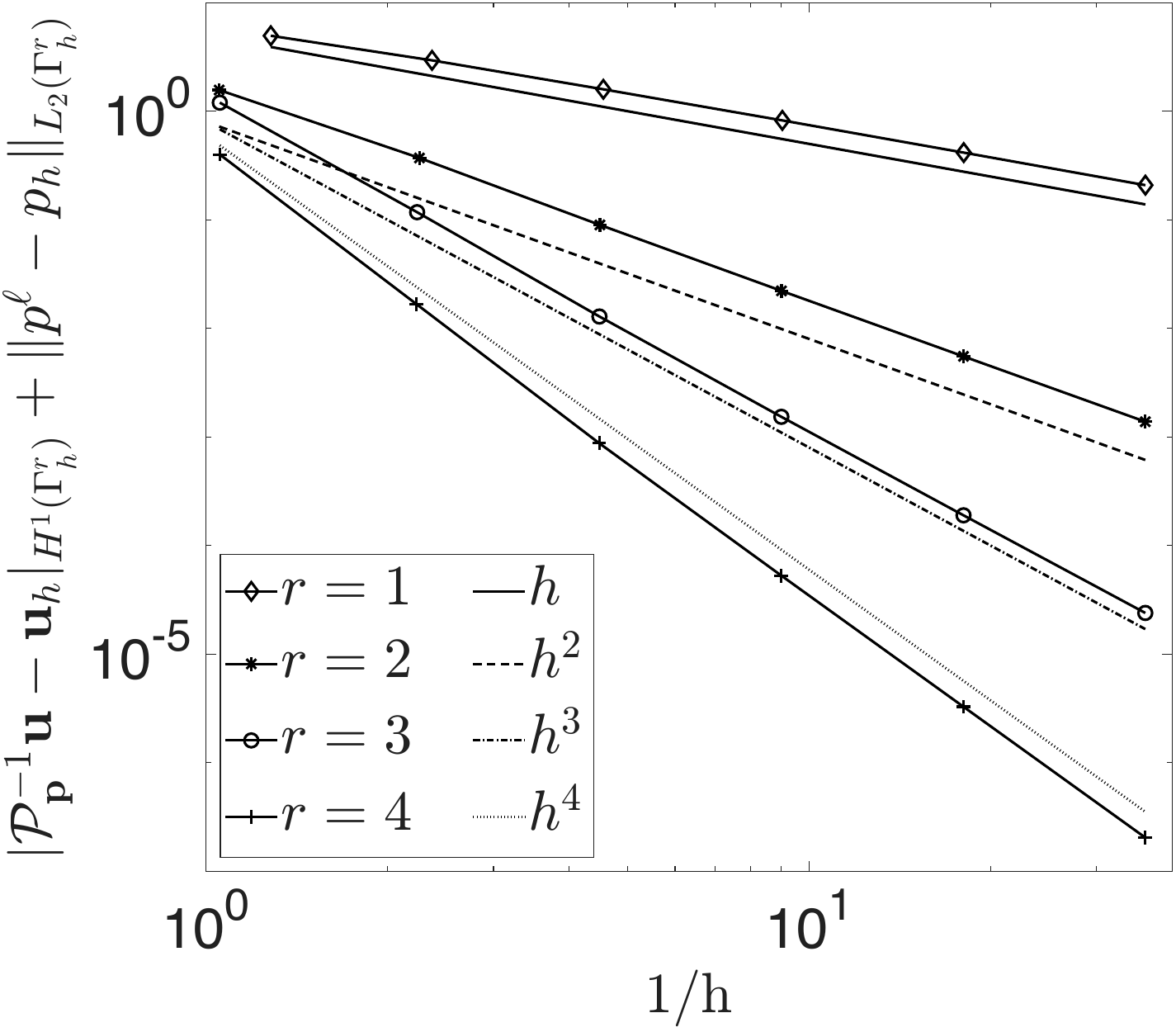}
\hspace{.3cm}
\includegraphics[alt={Plot of numerical experiments showing optimal L2 error decrease},scale=.26]{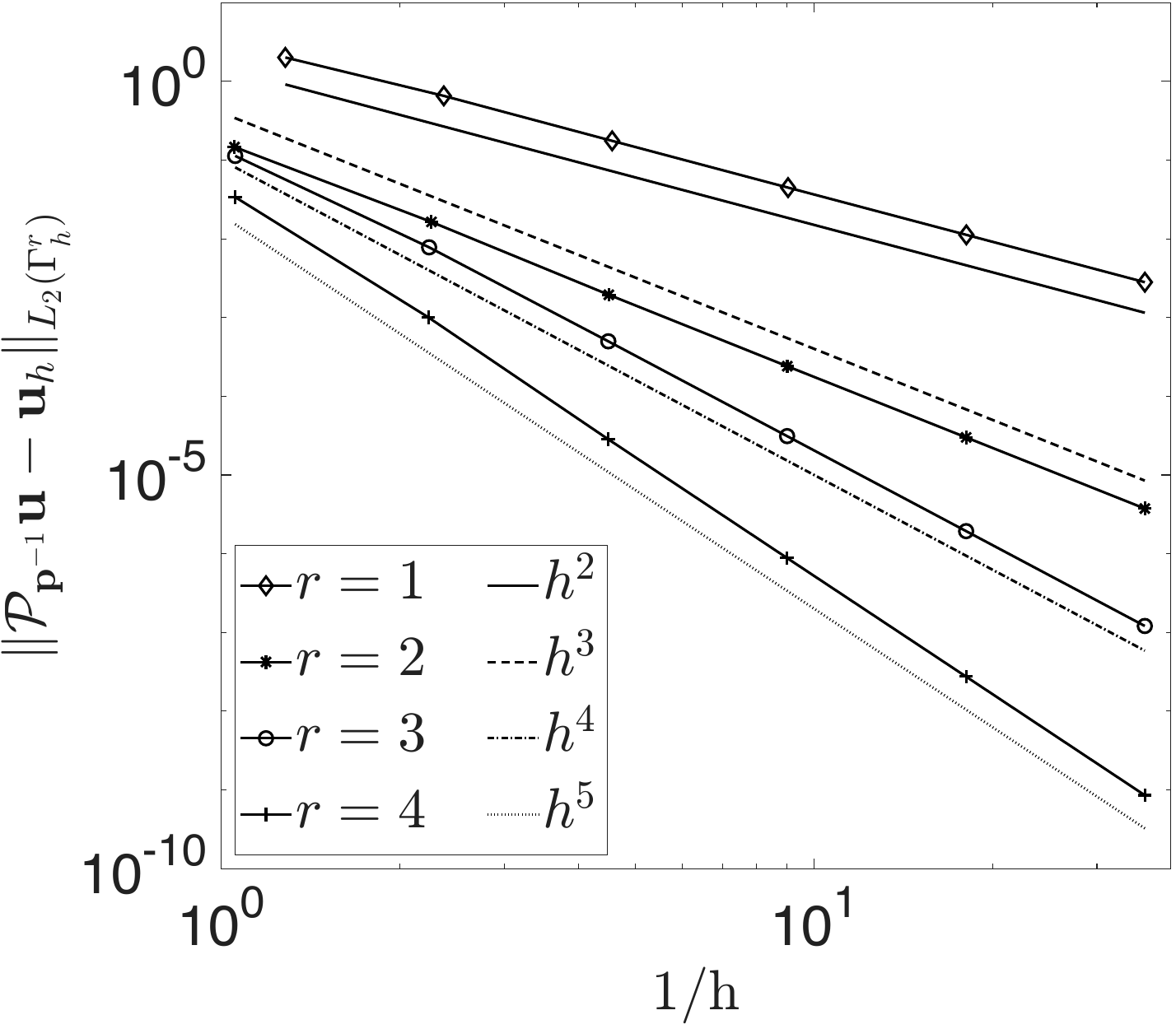}
\end{center}
\caption{Energy error $|\ipt \bu-\bu_h|_{H_h^1(\Gamma_{h}^k)}+ \|p^e-p_h\|_{L_2(\Gamma_{h}^k)}$ (left) and $L_2$ error $\|\ipt \bu -\bu_h\|_{L_2(\Gamma_{h}^k)}$ (right) for linear, quadratic, cubic, and quartic $BDM$ ($r=1,2,3,4$) elements with $k=r$.}
\label{fig_n1}
\end{figure}

We next consider surfaces with surface approximation degree one less than that of the finite element space, i.e., $k=r-1$.  In Figure \ref{fig_n2} below we confirm that \eqref{eq:energyerror} is sharp as taking $k=r-1$ leads to suboptimal convergence $O(h^{r-1})=O(h^k)$ in both the velocity $H^1$ and pressure errors.  Standard error estimates for the scalar Laplace-Beltrami problem yield order of convergence $h^r+h^{k+1}$ in the energy norm \cite{Demlow09, Dziuk88}.  As has been consistently observed in other works on the topic, the geometric error for the energy norm thus converges more slowly for problems of surface vector Laplace type than for scalar surface problems.  However, an interesting phenomenon occurs when $k$ is even.  In both plots in Figure \ref{fig_n2} below we see in the case $k=2$, $r=3$ that the error initially converges with order close to $h^3$, but then eventually decreases to order $h^2$ convergence as $h$ decreases.  In contrast, for $k$ odd the convergence rates are clearly of the expected order $h^k$ throughout the convergence history.  To explain this difference, we also plot a convergence history when $\gamma$ is the unit sphere for $r=3$ and $k=2$; for this plot the mesh nodes were placed on the sphere.  Here we see clear order $h^3$ convergence, which is better than expected.  This is due to superconvergence $\|d\|_{L_\infty(\Gamma_h^k)} \lesssim h^{k+2}$ and $\|\bnu-\bnu_h\|_{L_\infty(\Gamma_h^k)} \lesssim h^{k+1}$ that occurs on the sphere when $k$ is even (cf. \cite{Kil25} for more numerical experiments and \cite{HPZPP} for more general exploration of superconvegence phenomena in even-order geometry approximations).  This can be proved by elementary Taylor expansion arguments in the one-dimensional case.  This superconvergence phenomenon likely also explains the increased initial order of convergence observed in Figure \ref{fig_n2} on an ellipsoid.

\begin{figure}[h]
\begin{center}
\includegraphics[alt={Plot of numerical experiments showing velocity H^1 error decrease when k=r-1},scale=.26]{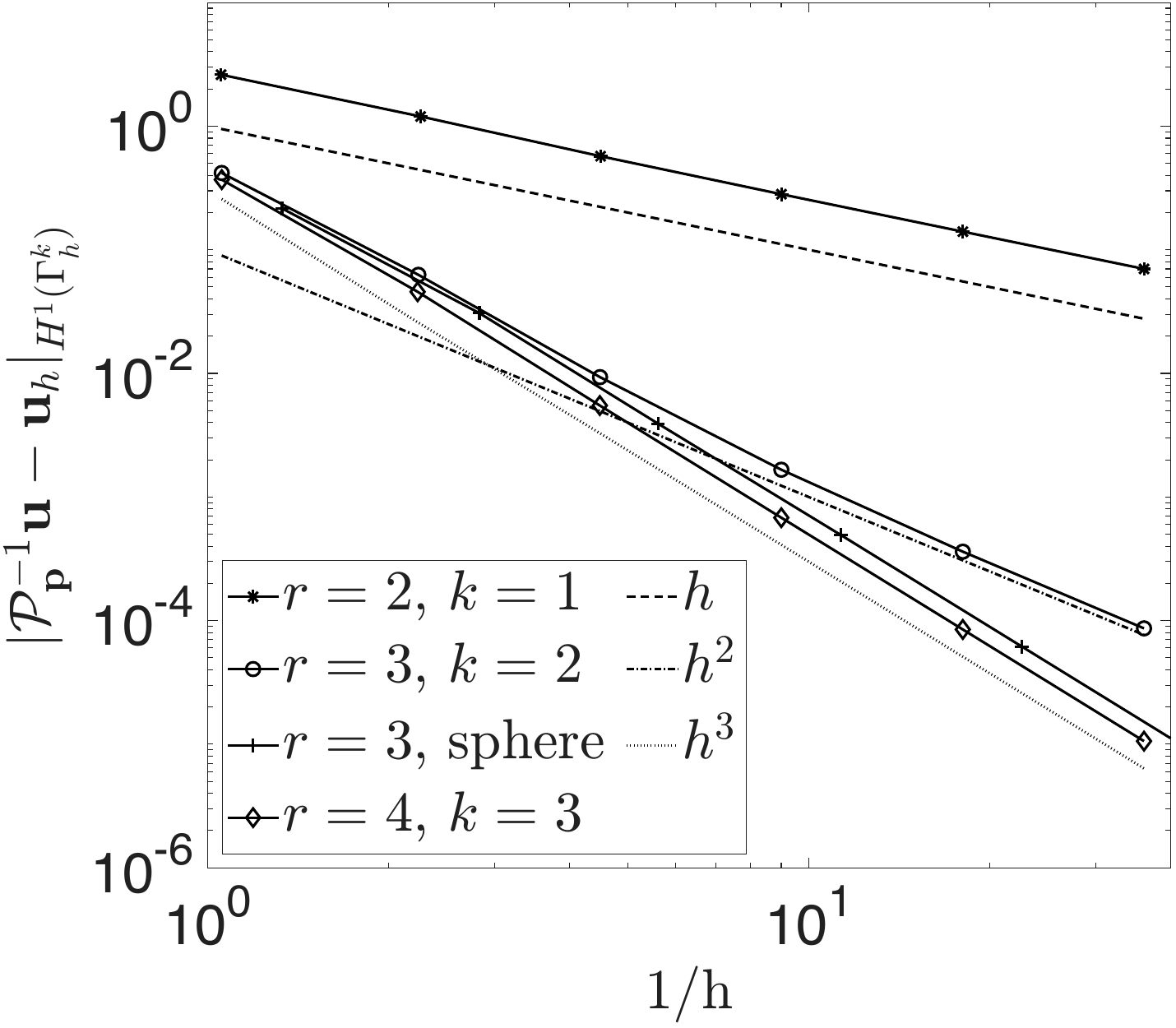}
\hspace{.3cm}
\includegraphics[alt={Plot of numerical experiments showing L2 pressure error decrease when r=k-1},scale=.26]{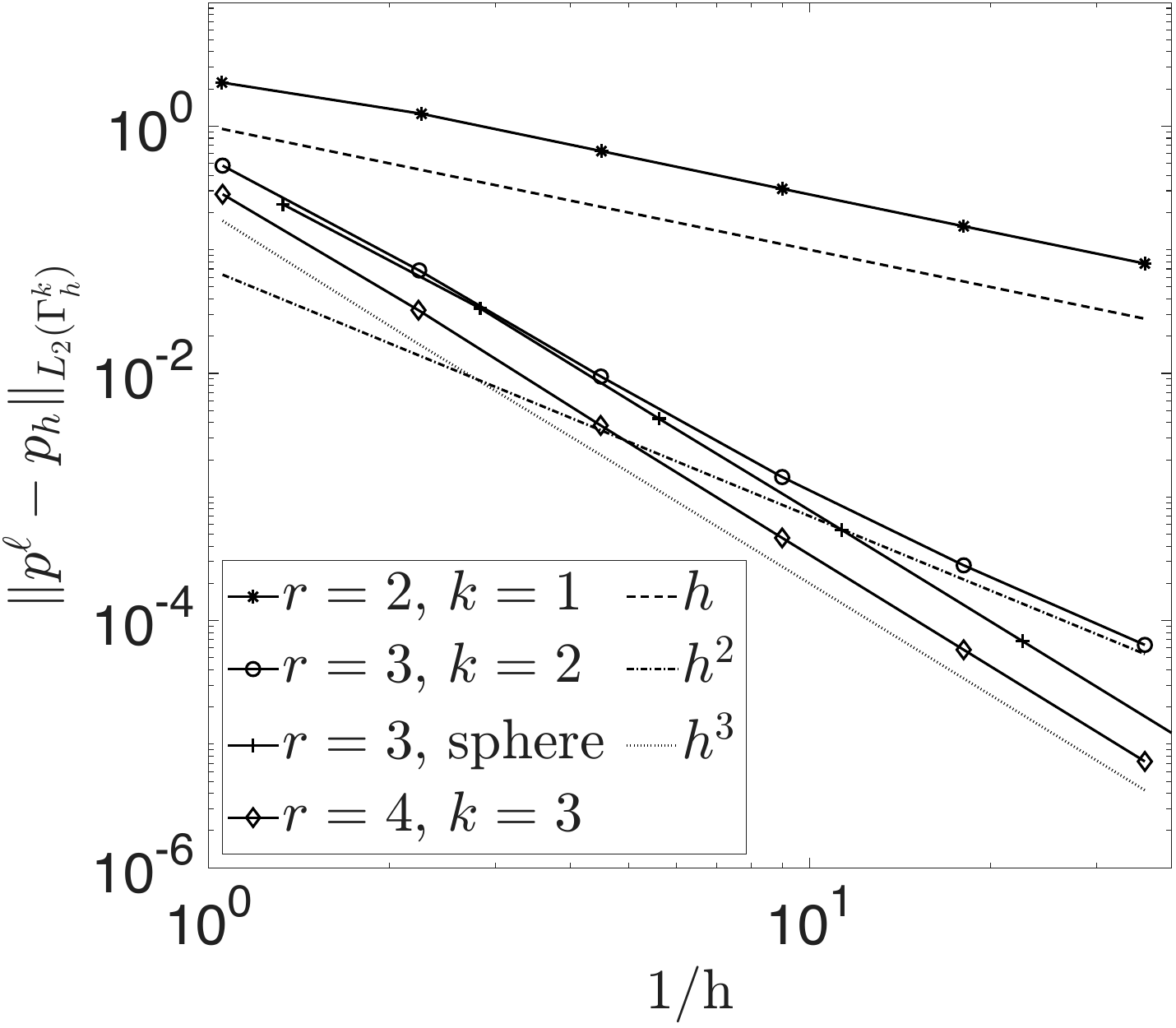}
\end{center}
\caption{$H^1$ error $|\ipt \bu-\bu_h|_{H_h^1(\Gamma_{h}^k)}$ (left) and pressure $L_2$ error $ \|p^e-p_h\|_{L_2(\Gamma_{h}^k)}$ (right) for linear, quadratic, cubic, and quartic $BDM$ ($r=1,2,3,4$) elements with $k=r-1$.}
\label{fig_n2}
\end{figure}

To test the necessity of the condition $\max_{e \in \calE} |d(z_{1,e})-d(z_{2,e})| \lesssim h^3$ for obtaining optimal $L_2$ error estimates when $r=k=1$, we conducted five experiments where this condition is violated.  In all experiments a mesh was initially constructed with nodes lying on $\gamma$.  In Experiment 1 the mesh nodes were randomly perturbed in the normal direction to $\gamma$ with uniformly distributed signed distances lying in $[-0.5 h^2, 0.5 h^2]$.  In Experiment 2 even-indexed nodes $z$ were perturbed in the normal direction with $d(z) \in [0, h^2]$ and odd nodes were perturbed so that $d(z) \in [-h^2,0]$, in both cases with uniform distribution.  In Experiment 3 the nodes were perturbed with uniform random distribution so that $d(z) \in [0, h^2]$.  In Experiment 4 even nodes were perturbed with uniform distribution $d \in [0, h^2]$ and odd nodes with uniform distribution $ \in [-0.5 h^2, 0]$.  In Experiment 5, a squared random uniform distribution was used with $d(z) \in [-0.5 h^2, 0.5 h^2]$.  As seen in Figure \ref{fig_n3}, in all cases this led to optimal order $h^2$ convergence of the $L_2$ velocity error even though $\max_{e \in \calE} |d(z_{1,e})-d(z_{2,e})| \approx h^2$. Thus, there is no indication that the extra condition used in our proof of $L_2$ estimates when $r=k=1$ is necessary to obtain optimal error decrease.   

\begin{figure}[h]
\begin{center}
\includegraphics[alt={Plot of numerical experiments showing L2 velocity error decrease when r=k=1 with mesh nodes perturbed off of the surface in various ways},scale=.26]{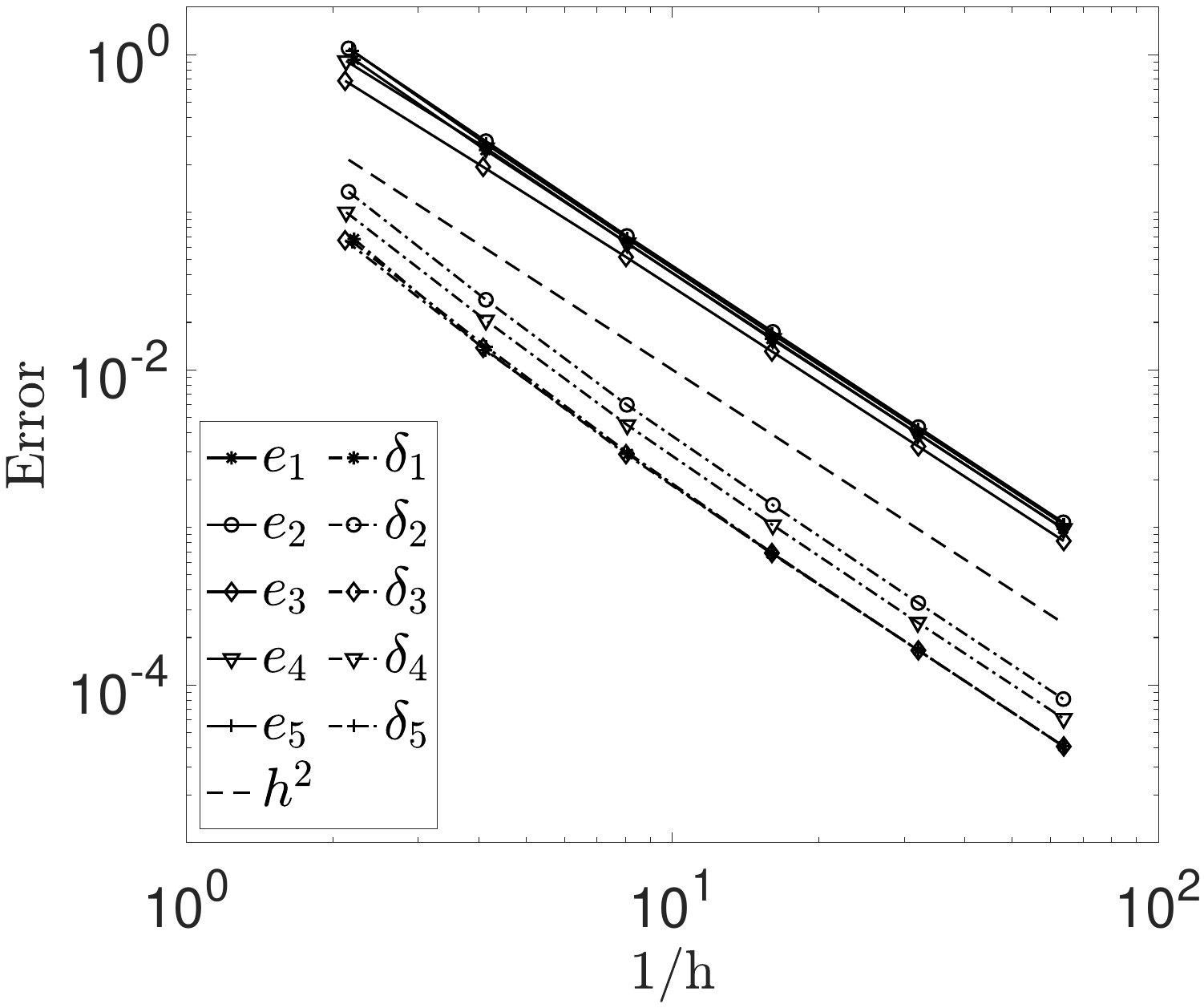}
\end{center}
\caption{$L_2$ errors $e_i=|\ipt \bu-\bu_{h,i}|_{L_2(\Gamma_{h,k})}$ and maximum differences $\delta_i=\max_{e \in \calE_i} |d(z_{1,e})-d(z_{2,e})|$.  Here Experiment $i$ ($ 1 \le i \le 5$) is described in the text. }
\label{fig_n3}
\end{figure}

\section*{Funding}
 Both authors were partially supported by NSF grant DMS-2012326.

\bibliographystyle{siam}
\bibliography{literatur}

\end{document}